\documentclass[12pt]{article}

\usepackage{latexsym}
\usepackage{a4wide}
\usepackage{amscd}
\usepackage{graphics}
\usepackage{amsmath}
\usepackage{amssymb}
\usepackage{soul,xcolor}
\usepackage{esint}
\usepackage{mathrsfs}
\usepackage{amsthm}
\usepackage{bbm}
\usepackage{color}
\usepackage{accents}
\usepackage{enumerate}
\usepackage{hyperref}
\input xy
\xyoption{all}

\newcommand{\N}{\mathbb{N}}                     
\newcommand{\Z}{\mathbb{Z}}                     
\newcommand{\R}{\mathbb{R}}                     
\newcommand{\C}{\mathbb{C}}                     
\newcommand{\D}{\mathbb{D}}                     
\newcommand{\re}{\mathrm{Re\,}}                 
\newcommand{\ran}{\mathrm{ran\,}}               

\newtheorem{mainthm}{\sc Theorem}           
\newtheorem{maincor}{\sc Corollary}           
\newtheorem{thm}{\sc Theorem}[section]               
\newtheorem*{thm*}{\sc Theorem}               
\newtheorem*{cor*}{\sc Corollary}        
\newtheorem{lem}[thm]{\sc Lemma}            
\newtheorem{prop}[thm]{\sc Proposition}     
\newtheorem{defn}[thm]{\sc Definition}      
\newtheorem{rem}[thm]{\sc Remark}           

\title{Sharp systolic inequalities for Reeb flows on the three-sphere} 

\author{Alberto Abbondandolo, Barney Bramham, \\ Umberto L.~Hryniewicz, Pedro A.~S.~Salom\~ao}

\date{}

\begin{document}

\setstcolor{red}

\maketitle

\begin{abstract}
The systolic ratio of a contact form $\alpha$ on the three-sphere is the quantity
\[
\rho_{\mathrm{sys}}(\alpha) = \frac{T_{\min}(\alpha)^2}{\mathrm{vol}(S^3,\alpha\wedge d\alpha)},
\]
where $T_{\min}(\alpha)$ is the minimal period of closed Reeb orbits on $(S^3,\alpha)$. 
A Zoll contact form is a contact form such that all the orbits of the corresponding Reeb flow are closed and have the same period. 
Our first main result is that $\rho_{\mathrm{sys}}\leq 1$ in a neighbourhood of the space of Zoll contact forms on $S^3$, with equality holding precisely at Zoll contact forms.  This implies a particular case of a conjecture of Viterbo, a local middle-dimensional non-squeezing theorem, and a sharp systolic inequality for Finsler metrics on the two-sphere which are close to Zoll ones. Our second main result is that $\rho_{\mathrm{sys}}$ is unbounded from above on the space of tight contact forms on $S^3$. 
\end{abstract}

\tableofcontents

\section{Introduction}

\subsection{The main theorems}

Let $\alpha$ be a contact form on the three-sphere $S^3$, that is, a smooth 1-form such that $\alpha \wedge d\alpha$ is a volume form. The contact form $\alpha$ defines a vector field $R_{\alpha}$ on $S^3$, which is called the Reeb vector field of $\alpha$, by the conditions
\[
\imath_{R_{\alpha}} d\alpha = 0, \qquad \alpha (R_{\alpha}) =1.
\]
The flow of $R_{\alpha}$ is called the Reeb flow of $\alpha$. It is well known that this flow always has periodic orbits, and we denote by $T_{\min}(\alpha)$ the minimum over all their periods. For arbitrary closed contact three-manifolds, the existence of a periodic orbit was proved by Taubes in \cite{tau07}, but in the case of $S^3$ we are considering here this follows from previous results of Rabinowitz \cite{rab78} and Hofer \cite{hof93}.

We denote by $\alpha_0$ the standard contact form on $S^3$, which is given by the restriction of the Liouville 1-form of $\R^4$
\[
\lambda_0 = \frac{1}{2} \sum_{j=1}^2 (x_j \, dy_j - y_j\, dx_j), \qquad (x_1,y_1,x_2,y_2)\in \R^4,
\]
to the unit sphere $S^3 = \{z\in \R^4 \mid |z|=1\}$. All the orbits of the Reeb flow of $\alpha_0$ are periodic and have period $\pi$. Actually, the orbits of $R_{\alpha_0}$ are the leaves of the Hopf fibration $S^3 \rightarrow S^2=\mathbb{CP}^1$. In particular, $T_{\min}(\alpha_0) = \pi$ and, since the volume of $S^3$ with respect to the volume form $\alpha_0\wedge d\alpha_0$ is $\pi^2$, we have the identity 
\begin{equation}
\label{equality}
T_{\min}(\alpha)^2 = \mathrm{vol}(S^3,\alpha\wedge d\alpha)
\end{equation}
when $\alpha=\alpha_0$. Since $T_{\min}(c\, \alpha)=|c| T_{\min}(\alpha)$, this identity holds also for $\alpha = c\, \alpha_0$, for every non-zero number $c$.

We say that the contact form $\alpha$ is {\em Zoll} if all the orbits of $R_{\alpha}$ are closed and have the same minimal period. Zoll contact forms are also called {\em regular} in the literature, see e.g.\ \cite{bw58} and \cite[Section 7.2]{gei08}, but we find the term ``Zoll'' more self-explanatory, due to the analogy with Riemannian geometry: indeed, a Riemannian metric on a manifold is called Zoll if all its geodesics are closed and have the same length. One can show that an arbitrary Zoll contact form $\alpha$ on $S^3$ is strictly contactomorphic to some multiple of $\alpha_0$, that is, there exists a diffeomorphism $\varphi:S^3 \rightarrow S^3$ such that $\varphi^* \alpha = c \,\alpha_0$ (see Proposition \ref{zoll} below). Here $c$ is a real number such that $|c|=T/\pi$, where $T$ is the common period of the orbits of $R_{\alpha}$. In particular, the identity (\ref{equality}) holds for every Zoll contact form $\alpha$.

The first aim of this article is to prove the following sharp bound on $T_{\min}(\alpha)$ in terms of the contact volume of $(S^3,\alpha)$:

\begin{mainthm}
\label{main1}
There exists a $C^3$-neighborhood $\mathscr{A}$ of the space of Zoll contact forms on $S^3$ such that 
\begin{equation}
\label{ineq}
T_{\min}(\alpha)^2 \leq \mathrm{vol}(S^3,\alpha\wedge d\alpha) \qquad \forall \alpha\in \mathscr{A},
\end{equation}
with equality holding if and only if $\alpha$ is Zoll.
\end{mainthm}

Another way to state this theorem is that the space of Zoll contact forms is a set of strict local maximisers in the $C^3$-topology for the {\em systolic ratio}
\[
\rho_{\mathrm{sys}}(\alpha) = \frac{T_{\min}(\alpha)^2}{\mathrm{vol}(S^3,\alpha\wedge d\alpha)}.
\]
Evidence in favour of this fact has been recently given by \'Alvarez Paiva and Balacheff in 
\cite{apb14}. Here, the authors prove that any smooth path of contact forms $t\mapsto \alpha(t)$ such that $\alpha(0)$ is Zoll satisfies the following alternative: either the systolic ratio of $\alpha(t)$ has a strict local maximum at $t=0$, or $t\mapsto \alpha(t)$ is tangent up to every order to the space of Zoll contact forms. This result holds for contact forms on manifolds of arbitrary dimension, and its proof is based on a reduction to normal forms.

A natural question is whether the inequality (\ref{ineq}) holds also for contact forms which are far from Zoll ones (see also \cite[Remark 8.5]{hut11}, \cite{cghr15} and \cite[Section 2.3]{apb14} for related questions and results about the relationship between the period of closed Reeb orbits and the contact volume).  The second aim of this paper is to give a negative answer to this question.

\begin{mainthm}
\label{main2}
For every neighborhood $\mathscr{R}$ of the flow of $R_{\alpha_0}$ in the $C^0_{\mathrm{loc}}(\R \times S^3,S^3)$-topology there exists a contact form $\alpha$ on $S^3$ which is smoothly isotopic to $\alpha_0$ such that the flow of $R_{\alpha}$ belongs to $\mathscr{R}$ and
\[
T_{\min}(\alpha)^2 > \mathrm{vol}(S^3,\alpha\wedge d\alpha).
\]
Moreover, for every $c>0$ there exists a contact form $\alpha$ on $S^3$ which is smoothly isotopic to $\alpha_0$ and such that
\[
T_{\min}(\alpha)^2 \geq c\,  \mathrm{vol}(S^3,\alpha\wedge d\alpha).
\]
\end{mainthm}

By Gray's stability theorem, the contact forms $\alpha$ which appear in the above theorem are tight. Another way of stating the second conclusion of the above theorem is that the systolic ratio is unbounded from above on the space of tight contact forms on $S^3$. 
It is worth noticing that these tight contact forms with arbitrarily  high systolic ratio can be chosen to be close to the standard one in the following weak sense: after a time reparametrization, their Reeb flow becomes $C^0_{\mathrm{loc}}$-close to the standard Reeb flow (see Remark \ref{howclose} below).

The proofs of the above two theorems are based on reading the dynamics of a Reeb flow by means of a disk-like global surface of section. Theorem \ref{main1} makes use of a fixed point theorem for area-preserving diffeomorphisms of the disk which are $C^1$-close to the identity, while the construction of the contact forms of Theorem \ref{main2} exploits the fact that this fixed point theorem fails for maps which are $C^1$-far from the identity. These results are presented in Section \ref{fixpointsec}, but since they might have some independent interest they are also informally discussed in Section \ref{intro-proof}, which concludes the Introduction.

Theorem \ref{main1} has several consequences, which we now discuss. Details about the precise derivation of these consequences from Theorem \ref{main1} are given in Section \ref{proofs}.

\subsection{A conjecture of Viterbo}

Let $\omega_0$ be the standard symplectic form on $\R^{2n}$:
\[
\omega_0 = \sum_{j=1}^n dx_j \wedge dy_j, \qquad (x_1,y_1,\dots,x_n,y_n)\in \R^{2n}.
\]
A symplectic capacity on $\R^{2n}$ is a function 
\[
c: \mathscr{P}(\R^{2n}) \rightarrow [0,+\infty]
\]
on the set of all subsets of $\R^{2n}$ which is monotone with respect to inclusion, invariant under symplectomorphisms, homogeneous of degree 2 with respect to homotheties, i.e. $c(rA) = r^2 c(A)$ for all $A\in \mathscr{P}(\R^{2n})$ and $r>0$, and such that
\[
c(B) = c(Z) = \pi,
\]
where $B$ is the unit ball and $Z$ is the cylinder consisting of all $(x_1,y_1,\dots,x_n,y_n)\in \R^{2n}$ such that $x_1^2 + y_1^2 < 1$.
The mere existence of a symplectic capacity is a non-trivial fact, as it immediately implies Gromov's non-squeezing theorem from \cite{gro85}: the ball $rB$ cannot be symplectically embedded into the cylinder $sZ$ if $r>s$. Many non-equivalent symplectic capacities have been constructed so far, such as the Ekeland-Hofer capacity $c_{\mathrm{EH}}$ and the Hofer-Zehnder capacity $c_{\mathrm{HZ}}$ (see \cite{eh89,hz90,hz94}). See also the survey \cite{chls07} for more information on symplectic capacities.

Let $c$ be a symplectic capacity on $(\R^{2n},\omega_0)$. In \cite[Section 5]{vit00}, Viterbo stated the following conjecture: for every bounded convex domain $C\subset \R^{2n}$ there holds
\begin{equation}
\label{vit}
c(C)^n \leq n! \, \mathrm{vol}(C),
\end{equation}
with equality holding if and only if $C$ is symplectomorphic to a ball. Here $\mathrm{vol} (C)$ is the volume of $C$ with respect to the standard volume form $dx_1\wedge dy_1 \wedge \dots \wedge dx_n \wedge dy_n$ on $\R^{2n}$. Partial positive results have been obtained in \cite{amo08} and \cite{apb14}, 
but the conjecture is still open. The validity of this conjecture would have far reaching consequences, as it would imply for instance the Mahler conjecture in convex geometry, see \cite{ako14}.

Let $C\subset \R^{2n}$ be a convex bounded domain with smooth boundary. Up to a translation, we may assume that $C$ is a neighborhood of the origin. Then the Liouville 1-form
\begin{equation}
\label{liouville-n}
\lambda_0 = \frac{1}{2} \sum_{j=1}^n ( x_j\, dy_j - y_j \, dx_j)
\end{equation}
restricts to a contact form on $\partial C$. It is well known that the Ekeland-Hofer and the Hofer-Zehnder capacities of $C$ coincide and equal the minimal period of periodic Reeb orbits on $(\partial C, \lambda_0|_{\partial C})$:
\[
c_{\mathrm{EH}}(C) = c_{\mathrm{HZ}}(C) = T_{\min}(\lambda_0|_{\partial C}).
\]
See \cite[Proposition 3.10]{vit89} and \cite[Proposition 4]{hz90}.
Actually, it is conjectured that all symplectic capacities agree on convex domains, but this conjecture is even stronger than Viterbo's one, since the smallest capacity
\[
c_B(C)= \sup \{ \pi r^2 \mid rB \mbox{ embeds symplectically into } C\}
\]
trivially satisfies (\ref{vit}). Theorem \ref{main1} allows us to prove the following special case of Viterbo's conjecture, where $\mathscr{B}$ denotes the set of all convex domains in $\R^4$ whose closure is symplectomorphic to a closed 4-ball:

\begin{maincor}
\label{cor1}
Let $c$ be either the Ekeland-Hofer or the Hofer-Zehnder capacity on $(\R^4,\omega_0)$. There exists a $C^3$-neighborhood $\mathscr{C}$ of $\mathscr{B}$ within the set of all convex smooth domains in $\R^4$ such that
\[
c(C)^2 \leq 2 \, \mathrm{vol} (C) \qquad \forall C\in \mathscr{C},
\]
with equality holding if and only if $C$ belongs to $\mathscr{B}$.
\end{maincor}

\subsection{Middle dimensional non-squeezing}

The already mentioned non-squeezing theorem of Gromov can be formulated in the following equivalent way: let $V$ be a 2-dimensional symplectic subspace of $(\R^{2n},\omega_0)$ and let $P$ be the symplectic projector onto it, that is, the linear projector onto $V$ along the symplectic orthogonal of $V$. Then
\[
\mathrm{area}(P \varphi(B),\omega_0|_V) \geq \pi,
\]
for every symplectic embedding $\varphi$ of the $2n$-dimensional unit ball $B$ into $\R^{2n}$ (see e.g.\ \cite{eg91}). A natural question is whether this 2-dimensional rigidity phenomenon has the following middle dimensional generalization: if $P$ is the symplectic projector onto a $2k$-dimensional symplectic subspace $V$ of $(\R^{2n},\omega_0)$, is it true that
\begin{equation}
\label{midnonsque}
\mathrm{vol}(P \varphi(B),\omega_0^k|_V) \geq \pi^k,
\end{equation}
for every symplectomorphism $\varphi$ of $B$ into $\R^{2n}$? The constant $\pi^k$ on the right-hand side of the inequality is the volume with respect to $\omega^k_0$ of the $2k$-dimensional unit ball. This question has been studied in \cite{am13}, where it is shown that the answer is positive for linear symplectomorphism, but that it is in general negative for nonlinear symplectictomorphisms: if $P$ is a symplectic projector onto a symplectic subspace $V$ of dimension $2k$, with $1<k<n$, then for every $\epsilon>0$ there is a symplectomorphism $\varphi: B \rightarrow \R^{2n}$ such that
\[
\mathrm{vol}(P \varphi(B),\omega_0^k|_V) < \epsilon.
\]
In \cite{am13} it is conjectured that the sharp inequality (\ref{midnonsque}) holds for symplectomorphisms which are close enough to a linear one, and some partial results are given in order to support this conjecture. By using Theorem \ref{main1}, we can give a positive answer to this conjecture in the case $k=2$:

\begin{maincor}
\label{cor2}
Let $P$ be a symplectic projector onto a 4-dimensional symplectic subspace $V$ of $(\R^{2n},\omega_0)$.
Then there exists a $C^3$-neighborhood $\mathscr{S}$ of the set of linear symplectomorphisms within the space of all symplectomorphisms from $\overline{B}$ into $\R^{2n}$ such that
\[
\mathrm{vol}(P \varphi(B),\omega_0^2|_V) \geq \pi^2 \qquad \forall \varphi\in \mathscr{S}.
\]
\end{maincor} 

An immediate consequence of the above result is the following 4-dimensional non-squeezing result for small balls:

\begin{maincor}
\label{cor3}
Let $P$ be a symplectic projector onto a 4-dimensional symplectic subspace $V$ of $(\R^{2n},\omega_0)$. Let $\varphi: A \rightarrow \R^{2n}$ be a smooth symplectomorphism of a domain $A\subset \R^{2n}$ into $\R^{2n}$. Then for every $z\in A$ there exists a positive number $\rho(z)$ such that
\[
\mathrm{vol}(P \varphi(z+rB),\omega^2_0|_V) \geq \pi^2 r^4 \qquad \forall z\in A, \; \forall r\in [0,\rho(z)].
\]
Moreover, the function $\rho$ is bounded away from zero on compact subsets of $A$.
\end{maincor} 

Since in this paper we work in the smooth category, we have stated the above result for smooth symplectomorphisms. But by density it extends readily to symplectomorphisms of class $C^3$. In the case of analytic symplectomorphisms, the above form of middle dimensional non-squeezing has been recently proved by Rigolli \cite{rig15}, with $V$ a symplectic subspace of any dimension, by using methods from the already mentioned \cite{apb14}.

\subsection{A systolic inequality for Finsler metrics on the two-sphere}

Given a Finsler metric $F$ on $S^2$, we denote by $\ell_{\min}(F)$ the length of the shortest closed geodesic on $(S^2,F)$. By $\mathrm{area}(S^2,F)$ we denote the Holmes-Thompson area of $(S^2,F)$, that is, the volume of the unit cotangent disk bundle in $T^* S^2$ with respect to the volume form $\omega\wedge \omega$ divided by $2\pi$, where $\omega$ is the standard symplectic form on $T^* S^2$ (see Section \ref{proofcor4} for more details). The Holmes-Thompson area coincides with the standard Riemannian area in the case of Riemannian metrics.

Since the work of Croke \cite{cro88}, a classical question in systolic geometry has been to find upper bounds for the systolic ratio
\[
\rho_{\mathrm{sys}}(F) = \frac{\ell_{\min}(F)^2}{\mathrm{area}(S^2,F)}.
\]
Indeed, this quantity is bounded from above on the space of Finsler metrics on $S^2$, as proved by Croke \cite{cro88} in the Riemannian case and by \'{A}lvarez Paiva, Balacheff and Tzanev \cite[Theorem VI]{apbt16} in the general Finsler case.

If the Finsler metric $F$ is Zoll, meaning that all its geodesics are closed and have the same length, then $\rho_{\mathrm{sys}}(F)=\pi$. This fact was proved by Weinstein in \cite{wei74} in the Riemannian case, but the proof goes through also in the Finsler setting.  
It is well known that Zoll metrics on $S^2$ do not maximise $\rho_{\mathrm{sys}}$, even if one restricts the attention to Riemannian metrics. Indeed, the supremum of $\rho_{\mathrm{sys}}(g)$ among all Riemannian metrics $g$ is at least $2\sqrt{3} > \pi$, as one can show by smoothing a singular metric known as the Calabi-Croke metric, see \cite{cro88}.  

In \cite{abhs17}, we proved that the round metric is nevertheless a local maximiser of $\rho_{\mathrm{sys}}$ among Riemannian metrics, as conjectured by Balacheff in \cite{bal06}. More precisely, if the curvature of a Riemannian metric $g$ is positive and sufficiently pinched, then $\rho_{\mathrm{sys}}(g) \leq \pi$, with equality holding if and only if $g$ is Zoll. By using Theorem \ref{main1} we can generalise this positive answer to Finsler metrics, possibly far from the round one:

\begin{maincor}
\label{cor4}
There exists a $C^3$-neighborhood $\mathscr{F}$ of the space of all Zoll Finsler metrics within the space of all Finsler metrics on $S^2$ such that
\[
\rho_{\mathrm{sys}}(F) \leq \pi \qquad \forall F\in \mathscr{F},
\]
with equality holding if and only if $F$ is Zoll.
\end{maincor}

Unlike in the main result of \cite{abhs17}, here we do not have an explicit description of the neighborhood $\mathscr{F}$ in terms of geometric quantities, and we have maximality of the Zoll metrics in a $C^3$-neighborhood instead of a $C^2$-neighborhood.

The proof of this corollary uses the fact that a Finsler geodesic flow on a two-sphere can be seen as a Reeb flow on $SO(3)$, and hence lifted to a Reeb flow on $S^3$. It should be also remarked that the fact that the round metric does not maximise the systolic ratio is compatible with the conjecture that the sharp inequality of Theorem \ref{main1} holds for all contact forms which are obtained as lifts to $S^3$ of the contact forms on $SO(3)$ induced by Finsler (or Riemannian) geodesic flows. This conjecture is not confuted by Theorem \ref{main2}, and its validity would have the following consequence, which was explained to us some years ago by \'Alvarez Paiva and is nicely discussed in Hutchings' blog \cite{hut13}: Since the unit tangent bundle $T^1 S^2$ of $S^2$ is diffeomorphic to $SO(3)$, closed geodesics might be non-contractible when seen as closed curves in $T^1 S^2$. Define $\tilde{\ell}_{\min}(F)$ to be the shortest length of a closed geodesic of the Finsler metric $F$ which is contractible in $T^1 S^2$. It is easy to see that a primitive closed geodesic is contractible in $T^1 S^2$ if and only if it has an odd number of self-intersections. In particular, simple closed geodesics are never contractible in $T^1 S^2$. Since the fundamental group of $SO(3)$ is $\Z_2$, closed geodesics become contractible  in $T^1 S^2$ once they are iterated twice. Therefore, closed geodesics which are not contractible in $T^1 S^2$ contribute to $\tilde{\ell}_{\min}(F)$ twice their length. One defines $\tilde{\rho}_{\mathrm{sys}}(F)$ as the ratio between the square of $\tilde{\ell}_{\min}(F)$ and the Holmes-Thompson area of $(S^2,F)$. With this definition, the above considerations show that $\tilde{\rho}_{\mathrm{sys}}$ takes the value $4\pi$ at the round metric, and more generally at every Zoll Finsler metric on $S^2$. If the inequality of Theorem \ref{main1} is true for every contact form on $S^3$ which is a lift of a contact form on $SO(3)$ induced by a Riemannian (resp.\ Finsler) metric on $S^2$, then Zoll metrics are maximizers of $\tilde{\rho}_{\mathrm{sys}}$ on the space of Riemannian (resp.\ Finsler) metrics on $S^2$.

\subsection{Sketch of the proof of the main theorems}  
\label{intro-proof}

A common ingredient in the proofs of Theorems \ref{main1} and \ref{main2} is the use of  a disk-like global surface of section for studying Reeb flows on $S^3$. This is a smoothly embedded closed disk $\Sigma$ in $S^3$ which is bounded by a periodic Reeb orbit and such that the Reeb flow is transverse to its interior, which is crossed infintely many times in the future and in the past by the orbit of any point which is not on the boundary of $\Sigma$. The existence of a disk-like global surface of section can be proved under quite general assumptions using $J$-holomorphic curves, see \cite{hwz98}, \cite{hwz03}, \cite{hs11}, \cite{hry14}, and \cite{hls15}. Here we need only construct it for Reeb flows which are close enough to the standard one, and this can be done by using more elementary perturbative arguments (see Section \ref{surfsecsec} below), which have the advantage of giving us some useful quantitative information.

Denote by $\phi_t$ the flow of the Reeb vector field $R_{\alpha}$ which is associated to a contact form $\alpha$ close enough to $\alpha_0$. Let $\Sigma$ be a disk-like global surface of section for $R_{\alpha}$ and let
\[
h: \D \rightarrow S^3
\]
be a smooth embedding of the closed disk $\D= \{z\in \C \mid |z|\leq 1\}$ with image $\Sigma$.  Then we have a first return time 
\[
\tau : \mathrm{int}(\D) \rightarrow \R, \qquad \tau(z) = \inf\{ t>0 \mid \phi_t(h(z)) \in \Sigma\},
\]
and a first return map
\[
\varphi: \mathrm{int}(\D) \rightarrow \mathrm{int}(\D), \qquad \varphi(z) = h^{-1}\bigl(\phi_{\tau(z)}(h(z))\bigr).
\]
Both $\tau$ and $\varphi$ extend smoothly to the closed disk and satisfy
\begin{equation}
\label{eqtau}
\varphi^* \lambda - \lambda = d\tau,
\end{equation}
where $\lambda$ is the smooth 1-form on $\D$ which is obtained as the pull-back by $h$ of the restriction of the contact form $\alpha$ to $\Sigma$. In particular, $\varphi$ preserves the smooth 2-form $\omega= d\lambda$, which is positive on the interior of $\D$ and vanishes on its boundary. When $\alpha$ is $C^3$-close enough to $\alpha_0$, the first return map $\varphi$ is $C^1$-close to the identity. Indeed, if  $\alpha$ is $C^3$-close to $\alpha_0$ then $R_{\alpha}$ is $C^2$-close to $R_{\alpha_0}$, and hence the flow of $R_{\alpha}$ is $C^2_{\mathrm{loc}}$-close to the flow of $R_{\alpha_0}$. This gives the $C^2$-closeness of $\varphi$ in the interior of $\mathbb{D}$ to the first return map of $\alpha_0$, which is the identity, but since the boundary behaviour of $\varphi$ is determined by the linearized flow along the bounding periodic orbit, in the end one gets only $C^1$-closeness of $\varphi$ to the identity on the whole of $\D$.
Moreover, $\varphi$ has a preferred lift $\tilde{\varphi}$ to $\widetilde{\mathrm{Diff}}(\D,\omega)$, the universal cover of the group $\mathrm{Diff}(\D,\omega)$ of $\omega$-preserving smooth diffeomorphisms of the disk $\D$, which is $C^1$-close to the identity of the group $\widetilde{\mathrm{Diff}}(\D,\omega)$.

Denote by $\pi: \widetilde{\mathrm{Diff}}(\D,\omega) \rightarrow \mathrm{Diff}(\D,\omega)$ the standard projection and let $\tilde{\varphi}$ be an element of $\widetilde{\mathrm{Diff}}(\D,\omega)$ with projection $\varphi=\pi(\tilde{\varphi})\in \mathrm{Diff}(\D,\omega)$. The element $\tilde{\varphi}$
has a well-defined Calabi invariant $\mathrm{CAL}(\tilde{\varphi},\omega)$, which is defined as follows. First, one defines a smooth function $\sigma: \D \rightarrow \R$, which is called the action of $\tilde{\varphi}$ with respect to the primitive $\lambda$, by requiring that
\begin{equation}
\label{eqsigma}
d\sigma = \varphi^* \lambda - \lambda,
\end{equation}
and that the value of $\sigma$ at any point $z\in \partial \D$ is the integral of $\lambda$ on the boundary path connecting $z$ and $\varphi(z)$ which is determined, up to homotopy, by the choice of the lift $\tilde{\varphi}$ of $\varphi$. The function $\sigma$ depends on the choice of the primitive $\lambda$ of $\omega$, but its value at fixed points and its integral over $\D$ do not. The latter quantity is the Calabi invariant of $\tilde{\varphi}$,
\[
\mathrm{CAL}(\tilde{\varphi},\omega) = \int_{\D} \sigma\, \omega,
\]
and defines a group homomorphism from $\widetilde{\mathrm{Diff}}(\D,\omega)$ onto $\R$. These notions are well-known for compactly supported area-preserving diffeomorphism of the disk (see e.g.\ \cite{gg95} and references therein). Their extension to diffeomorphisms which are not the identity on the boundary is presented in Sections \ref{calsec1} and \ref{calsec2} below. 

Now assume that the element $\tilde{\varphi}$ of $\widetilde{\mathrm{Diff}}(\D,\omega)$ is obtained from a Reeb flow $R_{\alpha}$ and a disk-like global surface of section $\Sigma$, as explained above. By comparing (\ref{eqtau}) and (\ref{eqsigma}), one sees that the action $\sigma$ of $\tilde{\varphi}$ with respect to $\lambda$ and the first return time $\tau$ differ by a constant. By examining the boundary behaviour, we will show that
\begin{equation}
\label{tautau}
\tau = T + \sigma,
\end{equation}
where $T$ is the period of the periodic orbit which bounds $\Sigma$, which agrees with the $\omega$-area of $\D$. By integrating (\ref{tautau}) over $\D$, we deduce that the contact volume of $(S^3,\alpha)$ and the Calabi invariant of $\tilde{\varphi}$ are related by the identity
\begin{equation}
\label{CALCAL}
\mathrm{vol}(S^3,\alpha\wedge d\alpha) = T^2 + \mathrm{CAL}(\tilde{\varphi},\omega).
\end{equation}
Since the closed orbits of $R_{\alpha}$ different from the boundary of $\Sigma$ are in one-to-one correspondence with the periodic points of $\varphi$, the identities (\ref{tautau}) and (\ref{CALCAL}) allow us to translate statements about the period of closed orbits of $R_{\alpha}$ and the contact volume of $(S^3,\alpha)$ into statements about the action of periodic points of $\varphi$ and the Calabi invariant of $\tilde{\varphi}$. 

Theorem \ref{main1} will be deduced from a fixed point theorem for radially monotone area-preserving diffeomorphisms of the disk, or more precisely from its corollary for area-preserving diffeomorphisms of the disk $C^1$-close to the identity. These results might be of independent interest and are proved in Section \ref{fixpointsec} below. Here a diffeomorphism $\varphi:\D \rightarrow \D$ fixing the origin is said to be radially monotone if the image by $\varphi$ of the foliation by rays emanating from the origin is transverse to the foliation by circles centred at the origin. Diffeomorphisms which fix the origin and are $C^1$-close to the identity are radially monotone. Our fixed point theorem is the following:

\begin{mainthm}
\label{fixpointmon}
Let $\omega$ be a smooth 2-form on $\mathbb{D}$, which is positive on $\mathrm{int}(\mathbb{D})$. Let $\tilde{\varphi}$ be an element of $\widetilde{\mathrm{Diff}}(\mathbb{D},\omega)$ such that:
\begin{enumerate}[(i)] 
\item $\varphi= \pi(\tilde{\varphi})$ fixes the origin and is radially monotone;
\item $\tilde{\varphi} \neq \mathrm{id}$;
\item $\mathrm{CAL}(\tilde{\varphi},\omega) \leq 0$.
\end{enumerate}
Then $\varphi$ has an interior fixed point $z_0$ with negative action. More precisely,
\[
\sigma(z_0) < \frac{1}{2} \frac{\mathrm{CAL}(\tilde{\varphi},\omega)}{\int_{\D} \omega}.
\]
\end{mainthm}

The coefficient $1/2$ in the above upper bound for the action is optimal (see Remark \ref{optimality} below).
If $\omega$ is the standard area form, the above theorem can be proved by using generating functions in polar coordinates. See Theorem 1.12 in \cite{abhs17} for a related result. In what follows, it is important to have the above result for arbitrary area forms (and notice that conjugation to obtain a standard area form might destroy the radial monotonicity of the map). We will treat the general case by showing that radially monotone diffeomorphisms preserving an arbitrary area form admit suitable generalised generating functions. Expressing the action and the Calabi invariant in terms of these generalised generating functions is considerably more complicated than in the case of standard ones, but the resulting formulas are quite neat and lead to the proof of Theorem \ref{fixpointmon}. 

The first return map $\varphi$ introduced above is $C^1$-close to the identity, but need not fix the origin. Being $\omega$-preserving, it has an interior fixed point, but since this fixed point might be arbitrarily close to the boundary, conjugacy by a diffeomorphism which brings the fixed point into the origin produces a map which may not be $C^1$-close to the identity anymore. Nevertheless, by using M\"obius transformations as conjugacies, we can show that if the original map is $C^1$-close to the identity, then the new one fixing the origin is radially monotone (see Proposition \ref{moebius} below). This new map preserves the pull-back of $\omega$ by the M\"obius transformation, and in general we do not have much control on how this two-form looks like. But since it holds for maps preserving an arbitrary two-form, Theorem \ref{fixpointmon}  has the following consequence:

\begin{maincor}
\label{fixpoint}
There exists a $C^1$-neighborhood $\mathscr{U}$ of the identity in the group $\mathrm{Diff}^+(\mathbb{D})$ of orientation preserving diffeomorphisms of the disk with the following property. Let $\omega$ be a smooth 2-form on $\mathbb{D}$, which is positive on $\mathrm{int}(\mathbb{D})$. Let $\tilde{\varphi}$ be an element of $\widetilde{\mathrm{Diff}}(\mathbb{D},\omega)$ such that:
\begin{enumerate}[(i)] 
\item $\varphi= \pi(\tilde{\varphi})$ belongs to $\mathscr{U}$;
\item $\tilde{\varphi} \neq \mathrm{id}$;
\item $\mathrm{CAL}(\tilde{\varphi},\omega) \leq 0$.
\end{enumerate}
Then $\varphi$ has an interior fixed point $z_0$ with negative action. More precisely,
\[
\sigma(z_0) < \frac{1}{2} \frac{\mathrm{CAL}(\tilde{\varphi},\omega)}{\int_{\D} \omega}.
\]

\end{maincor}

The above considerations about the first return map $\varphi$ and Corollary \ref{fixpoint} allow us to prove Theorem \ref{main1}. Indeed, since all Zoll contact forms are equivalent up to rescaling and strict contactomorphisms, see Proposition \ref{zoll} below, Theorem \ref{main1} will be proved if we can exhibit a $C^3$-neighbourhood $\mathscr{A}_0$ of the standard contact form $\alpha_0$ such that for every $\alpha\in \mathscr{A}_0$ which is not Zoll the strict inequality
\[
T_{\min}(\alpha)^2 < \mathrm{vol}(S^3,\alpha\wedge d\alpha)
\]
holds. Assume by contradiction that in any $C^3$-neighborhood of $\alpha_0$ we have non-Zoll contact forms $\alpha$ for which the opposite inequality
\[
T_{\min}(\alpha)^2 \geq \mathrm{vol}(S^3,\alpha\wedge d\alpha)
\]
holds. We can consider a disk-like global surface of section $\Sigma$ which is bounded by a closed orbit of $R_{\alpha}$ of minimal period $T_{\min}(\alpha)$. Then the above inequality and the identity (\ref{CALCAL}) imply that the Calabi invariant of the corresponding $\tilde{\varphi}$ is non-positive. Moreover, $\tilde{\varphi}$ is not the identity and, when $\alpha$ is $C^3$-close enough to $\alpha_0$, the corresponding first return map $\varphi\in \mathrm{Diff}(\D,\omega)$ is $C^1$-close to the identity. Therefore, Corollary \ref{fixpoint} can be applied and implies the existence  of an interior fixed point of $\varphi$ with negative action. By the identity (\ref{tautau}), this fixed point corresponds to a closed orbit of period less than $T_{\min}(\alpha)$, which contradicts the fact that $T_{\min}(\alpha)$ is the minimal period. 

It is worth noticing that this argument proves the sharp bound of Theorem \ref{main1} for any contact form on $S^3$ such that the closed orbit of minimal period bounds a global surface of section whose associated first return map is conjugated to a radially monotone map (the conjugacy need not be area-preserving). 

The conclusion of Theorem \ref{fixpointmon} is false if one does not assume $\varphi$ to be monotone, and similarly the conclusion of Corollary \ref{fixpoint} fails if $\varphi$ is not $C^1$-close to the identity. Indeed, in Section \ref{proofsec}  below we shall construct an element $\tilde{\varphi}$ of $\widetilde{\mathrm{Diff}}(\mathbb{D},\omega)$ whose projection $\varphi\in \mathrm{Diff}(\D,\omega)$ is $C^0$-close but $C^1$-far from the identity and not radially monotone, having negative Calabi invariant and a unique fixed point at the origin with positive action. Such an example can be refined in order to have a good control on the lower bound of the action of all periodic points, see Section \ref{controsec} below. The self-diffeomorphisms of the disk which are constructed in this way can be lifted to Reeb flows of $S^3$ by a general procedure which is described in Section \ref{liftsec}, and this leads to the proof of Theorem \ref{main2}.

\paragraph{Acknowledgments.} We are grateful to Patrice Le Calvez for suggesting the use of M\"obius transformations in Proposition \ref{moebius}, and to Kai Zehmisch for explaining to us how Proposition \ref{isaball} can be deduced from a theorem of Gromov and McDuff. We are also grateful to Michael Hutchings for a conversation which lead us to add the precise upper bound on the action of the fixed point to the statement of Theorem 3 and Corollary 5. Finally, we would like to thank Gabriele Benedetti and Jungsoo Kang for carefully reading our manuscript and correcting a mistake in the proof of Proposition 3.8.
A.\ Abbondandolo is partially supported by the DFG grant AB 360/1-1. 
B.\ Bramham is partially supported by the DFG grant BB 5251/1-1. 
P.\ Salom\~ao is partially supported by the FAPESP grant 2011/16265-8 and by the CNPq grant 306106/2016-7.

\section{The fixed point theorem}
\label{fixpointsec}

\subsection[The action and Calabi homomorphism for area-preserving disk maps]{The action and the Calabi homomorphism for area-preserving diffeomorphisms of the disk}
\label{calsec1}

Let $\mathbb{D}:= \{ z\in \C \mid |z|\leq 1\}$ be the closed unit disk in $\C\cong \R^2$. We denote by $\mathrm{Diff}^+(\mathbb{D})$ the group of smooth orientation preserving diffeomorphisms of $\mathbb{D}$. It is well-known that $\mathrm{Diff}^+(\mathbb{D})$ is connected. The universal cover of $\mathrm{Diff}^+(\mathbb{D})$ is denoted  by
$\widetilde{\mathrm{Diff}}(\mathbb{D})$
and is identified with the space of homotopy classes $[\{\varphi_t\}]$ of paths in $\mathrm{Diff}^+(\mathbb{D})$ starting at the identity. The pointwise composition of these paths defines a structure of topological group on $\widetilde{\mathrm{Diff}}(\mathbb{D})$ and the covering map
\begin{equation}
\label{unicov}
\pi: \widetilde{\mathrm{Diff}}(\mathbb{D}) \rightarrow \mathrm{Diff}^+(\mathbb{D}), \qquad [\{\varphi_t\}] \mapsto \varphi_1,
\end{equation}
is a homomorphism.

Let $\omega$ be a smooth 2-form on $\mathbb{D}$, which is strictly positive on $\mathrm{int}(\mathbb{D})$. We denote by  $\mathrm{Diff}(\mathbb{D},\omega)$ the subgroup of $\mathrm{Diff}^+(\mathbb{D})$ consisting of those diffeomorphisms which preserve $\omega$. The subgroup
\[
\widetilde{\mathrm{Diff}}(\mathbb{D},\omega) := \pi^{-1} \bigl( \mathrm{Diff}(\mathbb{D},\omega) \bigr)
\]
is the total space of a cover of $\mathrm{Diff}(\mathbb{D},\omega)$, which is obtained by restriction from (\ref{unicov}), and which we still denote by $\pi$:
\begin{equation}
\label{unicovomega}
\pi:  \widetilde{\mathrm{Diff}}(\mathbb{D},\omega) \rightarrow \mathrm{Diff}(\mathbb{D},\omega).
\end{equation}
Therefore, the elements of $\widetilde{\mathrm{Diff}}(\mathbb{D},\omega)$ are homotopy classes of isotopies of $\mathbb{D}$ starting at the identity and ending at some $\omega$-preserving diffeomorphism, and $\pi$ maps such a homotopy class of isotopies to its second end-point.

\begin{rem}
If $\omega$ does not vanish on the boundary of the disk, then one can use a theorem of Dacorogna and Moser \cite{dm90} in order to show that $\mathrm{Diff}(\mathbb{D},\omega)$ is connected and (\ref{unicovomega}) is its universal cover. For a general $\omega$ these facts might not be true. For instance, if $\omega$ vanishes only at one boundary point $z_0$, then every $\varphi\in \mathrm{Diff}(\mathbb{D},\omega)$ fixes $z_0$ and one can show that $\mathrm{Diff}(\mathbb{D},\omega)$ is simply connected and (\ref{unicovomega}) is the trivial cover from the disjoint union of a countable number of copies of  $\mathrm{Diff}(\mathbb{D},\omega)$. By considering $\omega$'s which vanish on a disconnected subset of $\partial \mathbb{D}$, one can build examples in which $\mathrm{Diff}(\mathbb{D},\omega)$ is not connected. The 2-form which is  obtained by restricting the differential of a contact form to a disk-like global surface of section vanishes with order one on the boundary, and one could show that this case is similar to the case of an everywhere positive 2-form. However, since the boundary behaviour of $\omega$ does not play any role as far as the results of this section are concerned, we do not make any assumption about it.
\end{rem}

Let $\tilde{\varphi} = [\{\varphi_t\}]$ be an element of  $\widetilde{\mathrm{Diff}}(\mathbb{D},\omega)$ and set $\varphi:= \pi (\tilde{\varphi}) = \varphi_1$. Let $\lambda$ be a smooth primitive of $\omega$ on $\mathbb{D}$. Since $\varphi$ preserves $\omega$, the 1-form $\varphi^* \lambda - \lambda$ is closed and hence exact on $\mathbb{D}$.
The {\em action} of $\tilde{\varphi}$ with respect to $\lambda$ is the unique smooth function  
\[
\sigma = \sigma_{\tilde{\varphi},\lambda} : \mathbb{D} \rightarrow \R
\]
such that
\begin{equation}
\label{uno}
d\sigma = \varphi^* \lambda - \lambda, 
\end{equation}
and 
\begin{equation}
\label{due}
\sigma(z) = \int_{\{ t \mapsto \varphi_t(z)\}} \lambda,
\end{equation}
for every $z\in \partial \mathbb{D}$. 

A function with these properties is clearly unique and, in order to show that it exists, we must check that if a function $\sigma$ satisfies (\ref{uno}) and (\ref{due}) for some $z\in \partial \mathbb{D}$, then the identity (\ref{due}) holds for every boundary point. Indeed, if $z'\in \partial \mathbb{D}$ is another boundary point, we can choose a path $\alpha:[0,1] \rightarrow \partial \mathbb{D}$  such that $\alpha(0)=z$, $\alpha(1) = z'$ and compute
\[
\begin{split}
\sigma(z') &= \sigma(z) + \int_{\alpha} d\sigma = \int_{\{ t \mapsto \varphi_t(z)\}} \lambda + \int_{\alpha} (\varphi^* \lambda - \lambda) \\ &= \int_{\{ t \mapsto \varphi_t(z)\}} \lambda + \int_{\varphi\circ \alpha} \lambda - \int_{\alpha} \lambda = \int_{\alpha^{-1} \# \{ t \mapsto \varphi_t(z)\} \# \varphi\circ \alpha} \lambda,
\end{split}
\]
where the symbol $\#$ denotes concatenation. Since the path $\alpha^{-1} \# \{ t \mapsto \varphi_t(z)\} \# \varphi\circ \alpha$ is homotopic to $\{ t \mapsto \varphi_t(z')\}$ within $\partial \mathbb{D}$ by the homotopy
\[
\bigl( \alpha|_{[s,1]} \bigr)^{-1} \# \{t\mapsto \varphi_t(\alpha(s))\} \# (\varphi\circ \alpha)|_{[s,1]}, \qquad s\in [0,1],
\]
the latter integral coincides with the integral of $\lambda$ along the curve $\{ t \mapsto \varphi_t(z')\}$. This shows that the action $\sigma$ is well defined.

It is also easy to check that the action $\sigma$ does not depend on the path $\{\varphi_t\}$ which represents the homotopy class $\tilde{\varphi}$: if we replace $\{\varphi_t\}$ by a homotopic path, then the path $\{ t \mapsto \varphi_t(z)\}$ which appears in (\ref{due}) is replaced by a homotopic path in $\partial \mathbb{D}$, so the integral of the 1-form $\lambda$ is not affected by this change.

In the next lemma, we investigate how the action $\sigma_{\tilde{\varphi},\lambda}$ is affected by the change of some of its defining data.

\begin{lem}
\label{formule}
Let $\tilde{\varphi}$ and $\tilde{\psi}$ be elements of $\widetilde{\mathrm{Diff}}(\mathbb{D},\omega)$. Let $\lambda$ be a smooth primitive of $\omega$ and let $u$ be a smooth real function on $\mathbb{D}$. Then:
\begin{enumerate}[(i)]
\item $\sigma_{\tilde{\varphi},\lambda+du} = \sigma_{\tilde{\varphi},\lambda} + u\circ \varphi - u$.
\item $\sigma_{\tilde{\psi}\circ \tilde{\varphi},\lambda} = \sigma_{\tilde{\psi},\lambda} \circ \varphi + \sigma_{\tilde{\varphi}, \lambda} = \sigma_{\tilde{\psi},\lambda} + \sigma_{\tilde{\varphi},\psi^* \lambda}$.
\item $\sigma_{\tilde{\varphi}^{-1},\lambda} = - \sigma_{\tilde{\varphi},\lambda} \circ \varphi^{-1} = - \sigma_{\tilde{\varphi},(\varphi^{-1})^*\lambda}$.
\end{enumerate}
\end{lem}

\begin{proof}
Assume that $\tilde{\varphi}$ and $\tilde{\psi}$ are represented by the paths $\{\varphi_t\}$ and $\{\psi_t\}$, respectively, and set $\varphi:=\varphi_1$, $\psi:=\psi_1$. Let $z$ be a point on $\partial \mathbb{D}$.

The first claim follows from the identities
\[
\varphi^* (\lambda + du) - (\lambda + du) = d \sigma_{\tilde{\varphi},\lambda} + \varphi^* (du) - du = d(\sigma_{\tilde{\varphi},\lambda}+u\circ \varphi - u),
\]
and
\[
\sigma_{\tilde{\varphi},\lambda+du}(z) = \int_{\{ t \mapsto \varphi_t(z)\}} (\lambda + du) = \sigma_{\tilde{\varphi},\lambda}(z) + u(\varphi(z)) - u(z).
\]
The first identity in (ii) follows from
\[
(\psi\circ \varphi)^* \lambda - \lambda = \varphi^*(\psi^* \lambda - \lambda) + \varphi^* \lambda - \lambda = \varphi^*(d\sigma_{\tilde{\psi},\lambda}) + d \sigma_{\tilde{\varphi},\lambda} = d ( \sigma_{\tilde{\psi},\lambda} \circ \varphi + \sigma_{\tilde{\varphi},\lambda} ),
\]
and
\[
\sigma_{\tilde{\psi}\circ \tilde{\varphi},\lambda}(z) = \int_{\{t\mapsto \{\psi_t\circ \varphi_t(z) \}} \lambda = \int_{\{t\mapsto \varphi_t(z)\}} \lambda + \int_{\{ t\mapsto \psi_t(\varphi(z))\}} \lambda = \sigma_{\tilde{\varphi},\lambda} (z) + \sigma_{\tilde{\psi},\lambda} (\varphi(z)),
\]
where we have used the fact that the path in the first integral is homotopic to the concatenation of the two paths in the last two integrals. The second identity in (ii) follows from
\[
(\psi\circ \varphi)^* \lambda - \lambda = \varphi^*(\psi^* \lambda) - \psi^* \lambda + \psi^* \lambda - \lambda = d( \sigma_{\tilde{\varphi},\psi^* \lambda} + \sigma_{\tilde{\psi},\lambda}),
\]
and
\[
\begin{split}
\sigma_{\tilde{\psi}\circ \tilde{\varphi},\lambda}(z)  &= \int_{\{t\mapsto \{\psi_t\circ \varphi_t(z) \}} \lambda = \int_{\{t\mapsto \psi_t(z)\}} \lambda + \int_{\{t\mapsto \psi(\varphi_t(z))\}} \lambda \\ &= \sigma_{\tilde{\psi},\lambda}(z) + \int_{\{t\mapsto \varphi_t(z)\}} \psi^* \lambda = \sigma_{\tilde{\psi},\lambda}(z) + \sigma_{\tilde{\varphi},\psi^* \lambda}(z).
\end{split}
\]
The two formulas in (iii) follows from those in (ii) applied to the case $\tilde{\psi}=\tilde{\varphi}^{-1}$, because $\sigma_{\mathrm{id},\lambda}=0$ for every $\lambda$.
\end{proof}

Let $\tilde{\varphi} = [\{\varphi_t\}]$ be an element of $\widetilde{\mathrm{Diff}}(\mathbb{D},\omega)$.  If $z_0$ is a boundary fixed point of $\varphi=\varphi_1$, then the path
\[
[0,1] \rightarrow \mathbb{D}, \quad t\mapsto \varphi_t(z_0), 
\]
is closed and by Stokes theorem the integral of $\lambda$ on it is the winding number of the above loop times the integral of $\omega$ on $\mathbb{D}$. In particular, the value of $\sigma_{\tilde{\varphi},\lambda}$ at $z_0$ does not depend on the primitive $\lambda$ of $\omega$. 

By Lemma \ref{formule} (i), the value of $\sigma_{\tilde{\varphi},\lambda}$ at an interior fixed point $z_0$ does not depend on $\lambda$ either. 

Thus, we can denote the action of a fixed point $z_0$ of $\varphi$ simply by
\[
\sigma_{\tilde{\varphi}}(z_0).
\]
We are now ready to define the Calabi invariant of an element of $\widetilde{\mathrm{Diff}}(\mathbb{D},\omega)$:

\begin{defn}
The Calabi invariant of $\tilde{\varphi}\in \widetilde{\mathrm{Diff}}(\mathbb{D},\omega)$ is the real number
\[
\mathrm{CAL}(\tilde{\varphi},\omega) := \int_{\mathbb{D}} \sigma\, \omega,
\]
where $\sigma=\sigma_{\tilde{\varphi},\lambda}$ is the action of $\tilde{\varphi}$ with respect to any smooth primitive $\lambda$ of $\omega$.
\end{defn}

The above integral does not depend on the choice of the primitive $\lambda$ because of Lemma \ref{formule} (i) and the fact that $\varphi:= \pi(\tilde{\varphi})$ preserves $\omega$.  Thus the Calabi invariant is well defined, giving us a real function
\[
\mathrm{CAL}(\cdot,\omega): \widetilde{\mathrm{Diff}}(\mathbb{D},\omega) \rightarrow \R.
\]
The action $\sigma_{\tilde{\varphi},\lambda}$ is zero when $\tilde{\varphi}$ is the identity of the group $\widetilde{\mathrm{Diff}}(\mathbb{D},\omega)$, and so is $\mathrm{CAL}(\tilde{\varphi},\omega)$. Moreover, Lemma \ref{formule} (ii) implies that if $\tilde{\varphi}$ and $\tilde{\psi}$ are two elements in $\widetilde{\mathrm{Diff}}(\mathbb{D},\omega)$ then
\[
\mathrm{CAL}(\tilde{\psi}\circ \tilde{\varphi},\omega) = \int_{\mathbb{D}} ( \sigma_{\tilde{\psi},\lambda} \circ \varphi + \sigma_{\tilde{\varphi},\lambda})\, \omega = \mathrm{CAL}(\tilde{\psi},\omega) +  \mathrm{CAL}(\tilde{\varphi},\omega),
\]
where $\varphi=\pi(\tilde{\varphi})$. We conclude that $\mathrm{CAL}(\cdot,\omega)$ is a group homomorphism. It is called the {\em Calabi homomorphism}.

We conclude this section by discussing the naturality of the concepts which we have introduced so far. Conjugacy by a diffeomorphism $h\in \mathrm{Diff}^+(\mathbb{D})$ defines a homomorphism
\[
\mathrm{Diff}(\mathbb{D},\omega) \rightarrow  \mathrm{Diff}(\mathbb{D},h^*\omega), \qquad \varphi \mapsto h^{-1} \circ \varphi \circ h.
\]
This homomorphism has a canonical lift to the respective covers:
\[
\widetilde{\mathrm{Diff}}(\mathbb{D},\omega) \rightarrow  \widetilde{\mathrm{Diff}}(\mathbb{D},h^*\omega), \qquad \tilde{\varphi} \mapsto h^{-1} \circ \tilde{\varphi} \circ h,
\]
where $h^{-1} \circ \tilde{\varphi} \circ h$ denotes the homotopy class $[\{h^{-1} \circ \varphi_t \circ h\}]$, for $\tilde{\varphi} = [\{\varphi_t\}]$.

\begin{prop}
\label{naturality}
For every $h\in \mathrm{Diff}^+(\mathbb{D})$, $\tilde{\varphi}\in \widetilde{\mathrm{Diff}}(\mathbb{D},\omega)$ and $\lambda$ smooth primitive of $\omega$, we have
\[
\sigma_{h^{-1} \circ \tilde{\varphi} \circ h,h^* \lambda} = \sigma_{\tilde{\varphi},\lambda}\circ h, \qquad \mathrm{CAL}(h^{-1} \circ \tilde{\varphi} \circ h, h^* \omega) = \mathrm{CAL}(\tilde{\varphi},\omega).
\]
\end{prop}

\begin{proof}
Set $\varphi:= \pi(\tilde{\varphi})$. The function $\sigma_{\tilde{\varphi},\lambda}\circ h$ satisfies
\[
d ( \sigma_{\tilde{\varphi},\lambda}\circ h ) = d ( h^* \sigma_{\tilde{\varphi},\lambda}) = h^*(d \sigma_{\tilde{\varphi},\lambda}) =h^* (\varphi^* \lambda - \lambda) = (h^{-1} \circ \varphi \circ h)^* h^* \lambda - h^* \lambda,
\]
and
\[
\sigma_{\tilde{\varphi},\lambda}\circ h(z) = \int_{\{ t \mapsto \varphi_t (h(z))\}} \lambda =
\int_{\{t\mapsto h^{-1} \circ \varphi_t \circ h(z)\}} h^* \lambda,
\]
for every $z\in \partial \mathbb{D}$. Therefore, it coincides with the action of $h^{-1} \circ \tilde{\varphi} \circ h$ with respect to $h^* \lambda$. This proves the first identity. The second one follows by integration:
\[
\mathrm{CAL}(h^{-1} \circ \tilde{\varphi} \circ h, h^* \omega) = \int_{\mathbb{D}} \sigma_{h^{-1} \circ \tilde{\varphi} \circ h,h^* \lambda} \, h^* \omega = \int_{\mathbb{D}}   \sigma_{\tilde{\varphi},\lambda}\circ h \, h^* \omega =  \int_{\mathbb{D}}  h^* ( \sigma_{\tilde{\varphi},\lambda}\, \omega) = \int_{\mathbb{D}} \sigma_{\tilde{\varphi},\lambda}\, \omega.
\]
\end{proof}

\begin{rem}
\label{compsuppdiff}
Denote by $\mathrm{Diff}_c(\D,\omega)$ the subgroup of $\mathrm{Diff}(\D,\omega)$ consisting of those $\omega$-preserving diffeomorphisms which have compact support in $\mathrm{int}(\D)$. It is well known that this group is simply connected, and that the action and the Calabi invariant descend to it. Indeed, the action of a compactly supported diffeomorphism can be normalised by asking it to be zero on $\partial \D$, and the Calabi invariant is defined as the integral of the action. This is actually the setting in which these notions are usually introduced: see for instance \cite{gg95}. When dealing with an element $\varphi$ of $\mathrm{Diff}_c(\D,\omega)$, we can therefore write $\sigma_{\varphi,\lambda}$ and $\mathrm{CAL}(\varphi,\omega)$
for the action with respect to the primitive $\lambda$ and for the Calabi invariant of $\varphi$.
\end{rem}

\subsection{Hamiltonian formalism}
\label{calsec2}

In this section we compute the action and the Calabi invariant of an element $\tilde{\varphi} $ of $\widetilde{\mathrm{Diff}}(\D,\omega)$ in terms of a generating Hamiltonian. In other words, we assume that there is a smooth function $H:[0,1]\times \D \rightarrow \R$, $H_t(z)=H(t,z)$, such that $\tilde{\varphi}$ is the equivalence class of the isotopy  $\{\varphi_t\}$ which solves the Cauchy problem
\[
\frac{d}{dt} \varphi_t = X_{H_t}\circ \varphi_t, \qquad \varphi_0 = \mathrm{id},
\]
where the Hamiltonian vector field $X_{H_t}$ is such that
\begin{equation}
\label{hamilt}
\imath_{X_{H_t}} \omega = dH_t.
\end{equation}
The fact that $\varphi_t$ leaves the boundary of $\D$ invariant implies that $X_{H_t}$ is tangent to the boundary. Then (\ref{hamilt})  implies that $dH_t$ vanishes on the tangent lines to the boundary of $\D$, and hence $H_t$ is constant on $\partial \D$. Up to the addition of a function of $t$, which does not affect the vector field $X_{H_t}$, we may assume that
\begin{equation}
\label{H_bc}
H_t = 0 \qquad \mbox{on } \partial \D, \qquad \forall t\in [0,1].
\end{equation}

\begin{prop}
\label{actham}
If $H$ satisfies (\ref{H_bc}) then for any smooth primitive $\lambda$ of $\omega$ there holds
\[
\sigma_{\tilde{\varphi},\lambda}(z) = \int_{\{t\mapsto \varphi_t(z)\}} \lambda + \int_0^1 H_t(\varphi_t(z))\, dt,
\]
for every $z\in \D$.
\end{prop}

\begin{proof}
Let $\gamma:[0,1]\rightarrow \D$ be a smooth path such that $\gamma(0)=z$ and $\gamma(1)\in \partial \D$. Define the smooth map
\[
\psi: [0,1]^2 \rightarrow \D, \qquad \psi(s,t) := \varphi_t(\gamma(s)).
\]
We compute the integral of $\psi^*\omega$ on $[0,1]^2$ in two different ways. The first computation gives us:
\[
\begin{split}
\int_{[0,1]^2} \psi^* \omega &= \int_{[0,1]^2} \omega(\varphi_t(\gamma(s)))\left[ \frac{\partial}{\partial s} \varphi_t(\gamma(s)), X_{H_t} (\varphi_t(\gamma(s))) \right] \, ds \, dt \\ &= - \int_{[0,1]^2} dH_t(\varphi_t(\gamma(s)))\left[ \frac{\partial}{\partial s} \varphi_t(\gamma(s)) \right]\, ds\, dt = - \int_0^1 \left( \int_0^1 \frac{\partial}{\partial s} H_t(\varphi_t(\gamma(s)))\, ds \right)\, dt \\ &= - \int_0^1 \left( H_t(\varphi_t(\gamma(1))) - H_t(\varphi_t(z)) \right)\, dt = \int_0^1 H_t(\varphi_t(z))\, dt.
\end{split}
\]
The second computation uses Stokes' theorem:
\[
\begin{split}
\int_{[0,1]^2} \psi^* \omega &= \int_{[0,1]^2} \psi^* d\lambda = \int_{[0,1]^2} d\psi^* \lambda = \int_{\partial [0,1]^2} \psi^* \lambda \\ &= \int_{\gamma} \lambda + \int_{\{t\mapsto \varphi_t(\gamma(1))\}} \lambda - \int_{\varphi\circ \gamma} \lambda - \int_{\{t\mapsto \varphi_t(z)\}}\lambda  \\ &= \sigma_{\tilde{\varphi},\lambda}(\gamma(1)) - \int_{\gamma} ( \varphi^* \lambda - \lambda) - \int_{\{t\mapsto \varphi_t(z)\}} \lambda\\ &= \sigma_{\tilde{\varphi},\lambda}(\gamma(1)) - \int_{\gamma} d\sigma_{\tilde{\varphi},\lambda} - \int_{\{t\mapsto \varphi_t(z)\}} \lambda\\ &= \sigma_{\tilde{\varphi},\lambda}(z) - \int_{\{t\mapsto \varphi_t(z)\}}\lambda.
\end{split}
\]
The desired formula for $\sigma_{\tilde{\varphi},\lambda}(z)$ follows by comparing the above two identities.
\end{proof}

We conclude this section by expressing the Calabi invariant in terms of the generating Hamiltonian. 

\begin{prop}
\label{calham}
If $H$ vanishes on the boundary as in (\ref{H_bc}), then the Calabi invariant of $\tilde{\varphi}$ is given by the formula
\[
\mathrm{CAL}(\tilde{\varphi},\omega) = 2 \int_{[0,1] \times \D} H(t,z) \, dt\wedge \omega.
\]
\end{prop}

\begin{proof}
We use the formula of Proposition \ref{actham} and compute the two integrals on the right-hand side of
\[
\mathrm{CAL}(\tilde{\varphi},\omega) = \int_{\D} \sigma_{\tilde{\varphi},\lambda} \, \omega = \int_{\D} \left(  \int_0^1 \lambda(\varphi_t(z))\Bigl[\frac{d}{dt} \varphi_t(z)\Bigr]\, dt\right)  \, \omega +
\int_{\D} \left(   \int_0^1 H_t(\varphi_t(z))\, dt \right) \, \omega
\]
separately. Since the map $\varphi_t$ is $\omega$-preserving for every $t\in [0,1]$, the second integrals equals
\[
\int_{\D} \left(   \int_0^1 H_t(\varphi_t(z))\, dt \right) \, \omega = \int_{[0,1] \times \mathbb{D}} H \, dt\wedge \omega
\]
by Fubini's theorem. Using again Fubini's theorem and the fact that $\varphi_t$ is area-preserving, the first integral can be rewritten as
\[
\begin{split}
\int_{\D} \left(  \int_0^1 \lambda(\varphi_t(z))\Bigl[\frac{d}{dt} \varphi_t(z)\Bigr]\, dt\right) \, \omega &= \int_{\D} \left(  \int_0^1 \lambda(\varphi_t(z))\bigl[X_{H_t}(\varphi_t(z))\bigr]\, dt\right)  \, \omega \\ &= \int_0^1 \left( \int_{\D} \lambda(X_{H_t}) \, \omega \right) dt 
 \end{split}
 \]
From the identity
\[
0 = \imath_{X_{H_t}} ( \lambda \wedge d\lambda) = (\imath_{X_{H_t}} \lambda) d\lambda - \lambda \wedge \imath_{X_{H_t}} d\lambda = \lambda(X_{H_t}) \, \omega - \lambda \wedge dH_t,
\]
we find
\[
\begin{split}
\int_0^1 \left( \int_{\D} \lambda(X_{H_t}) \, \omega \right) dt  &= \int_0^1 \left( \int_{\D} \lambda \wedge dH_t \right) dt = \int_0^1 \left( \int_{\D} \bigl( H_t \, d\lambda - d(H_t \lambda) \bigr) \right) dt \\&= \int_{[0,1] \times \mathbb{D}} H \, dt\wedge \omega - \int_0^1 \left( \int_{\partial \D} H_t \lambda \right)\, dt \\ &=  \int_{[0,1] \times \mathbb{D}} H \, dt\wedge \omega,
\end{split}
\]
where we have used the fact that $H_t$ vanishes on the boundary of $\D$. The desired formula for $\mathrm{CAL}(\tilde{\varphi},\omega)$ follows.
\end{proof}

\begin{rem} The formulas of Propositions \ref{actham} and \ref{calham} show that the conclusion of Theorem \ref{fixpointmon} from the Introduction holds for an element $\tilde{\varphi}$ which satisfies assumption (ii) and (iii) of that theorem and is generated by an autonomous Hamiltonian $H:\D \rightarrow \R$. Indeed, the fact that $\tilde{\varphi}$ is not the identity implies that $H$, which is normalized to be zero on the boundary, is not identically zero, so by Proposition \ref{calham} this Hamiltonian is somewhere negative and hence achieves its minimum at some interior critical point $z_0$. This point is a fixed point of $\varphi=\pi(\tilde{\varphi})$ of action
\[
\sigma_{\tilde{\varphi}}(z_0) = H(z_0) < \fint_{\D} H \omega = \frac{1}{2} \frac{\mathrm{CAL}(\tilde{\varphi},\omega)}{\int_{\D} \omega} \leq 0.
\]
The same argument works for elements $\tilde{\varphi}$ which are generated  by a Hamiltonian $H:[0,1]\times \D \rightarrow \R$ for which there exists $z_0\in \D$ such that
\[
H(t,z_0) = \min_{z\in \D} H(t,z) \qquad \forall t\in [0,1].
\]
In the case of the standard area form $\omega_0$ on $\D$, any $\varphi\in \mathrm{Diff}_c(\D,\omega_0)$ which is $C^1$-close enough to the identity can be generated by a Hamiltonian with the above property (see \cite{bp94}). This fact leads to an alternative proof of a special case of Corollary \ref{fixpoint} for compactly supported maps.
\end{rem}

We conclude this section by computing the action and the Calabi invariant of diffeomorphisms which are produced by a radial autonomous Hamiltonian and by the standard area form $dx\wedge dy$ on $\D$. Here it is convenient to work with the standard primitive
\[
\lambda_0 := \frac{1}{2} \bigl( x\, dy - y \, dx \bigr)
\]
of $dx\wedge dy$, which is radially symmetric and whose expression in polar coordinates is
\[
\lambda_0 = \frac{r^2}{2}\, d\theta.
\]

\begin{lem}
\label{radialham}
Consider the Hamiltonian
\[
H: \D \rightarrow \R, \qquad H(z) = h(|z|^2),
\]
where $h:[0,1]\rightarrow \R$ is a smooth function with $h(1)=0$. Let $\tilde{\varphi}$ be the element of $\widetilde{\mathrm{Diff}}(\D,dx\wedge dy)$ which is defined by the flow  of the Hamiltonian vector field $X_H$ and let $\varphi:= \pi(\tilde{\varphi})$ be the corresponding time-1 map. Then
\[
\varphi(z) = e^{-2 h'(|z|^2)i} z, \qquad \sigma_{\tilde{\varphi},\lambda_0}(z) = h(|z|^2) - |z|^2 h'(|z|^2) \qquad \forall z\in \D,
\]
and
\[
\mathrm{CAL}(\tilde{\varphi},dx\wedge dy) = 4\pi \int_0^1 r\, h(r^2) \, dr.
\]
\end{lem}

\begin{proof}
The Hamiltonian vector field of $H$ is easily seen to vanish at $0$ and to be given by
\[
X_H(z) = - 2 h'(|z|^2) \frac{\partial}{\partial \theta}
\]
for $z\in \mathbb{D}\setminus \{0\}$. Its flow is therefore
\[
\varphi_t(z) = e^{-2 h'(|z|^2)it} z, \qquad \forall z\in \D, \; \forall t\in \R,
\]
and in particular
\[
\varphi(z) = \pi(\tilde{\varphi})(z) = \varphi_1(z) = e^{-2 h'(|z|^2)i} z,
\]
as claimed.
Let $z\in \partial \D$. The integral of $\lambda_0$ over the boundary path 
\[
[0,1] \rightarrow \partial \D, \qquad t\mapsto \varphi_t(z) = e^{-2 h'(1)it} z,
\]
is $-h'(1)$. The function
\[
\D \rightarrow \R, \qquad z\mapsto h(|z|^2) - |z|^2 h'(|z|^2)
\]
takes the value $-h'(1)$ on $\partial \D$ and its differential is easily seen to agree with
\[
\varphi^* \lambda_0 - \lambda_0 = - 2 r^3 h''(r^2) \, dr.
\]
Therefore, this function is the action $\sigma_{\tilde{\varphi},\lambda_0}$, as claimed. The formula for the Calabi invariant of $\tilde{\varphi}$ follows by integrating the action, or directly from  Proposition \ref{calham}.
\end{proof}

\subsection{Polar coordinates and lift to the strip}

When dealing with disk maps fixing the origin, it is convenient to work with polar coordinates: in other words, we consider the smooth map
\[
p: S \rightarrow \mathbb{D}, \quad S := [0,1]\times \R, \qquad p(r,\theta) := r e^{i\theta},
\]
which restricts to a smooth covering map from $(0,1] \times \R$ to the punctured disk $\mathbb{D}\setminus \{0\}$. The map
\begin{equation}
\label{deckT}
T: S \rightarrow S, \qquad T(r,\theta) := (r,\theta+2\pi),
\end{equation}
restricts to a deck transformation of this cover.

\begin{lem}\label{L:extn_of_lifts}
Let $\varphi:\mathbb{D}\rightarrow\mathbb{D}$ be an orientation preserving diffeomorphism of class $C^k$ which fixes the origin, with $k\geq 1$.  Let 
\[
\ring{\Phi}: (0,1]\times \R \rightarrow(0,1] \times \R 
\]
be a lift of the restriction $\varphi|_{\mathbb{D}\backslash\{0\}}$ 
via the covering map $p$.   Then:
\begin{enumerate}[(i)]
\item $\ring{\Phi}$ extends uniquely to a homeomorphism $\Phi:S\rightarrow S$, which in the case $k\geq 2$ is a diffeomorphism of class $C^{k-1}$.  
\item If we set $\Phi(r,\theta) = (R(r,\theta),\Theta(r,\theta))$ we have for $k\geq 2$
\[
D_1 R(0,\theta) = \bigl|D\varphi(0)[e^{i\theta}]\bigr| \qquad \forall \theta\in \R. 
\]
\end{enumerate}
\end{lem}

\begin{proof}
In order to prove (i), it is enough to show that $\ring{\Phi}$ has an extension $\Phi: S \rightarrow S$ of class $C^{k-1}$: once this is proved, the same is true for $\ring{\Phi}^{-1}$ and hence both $\Phi$ and $\Phi^{-1}$ are maps of class $C^{k-1}$.

Observe that $p$ factors through maps $q$ and $h$ as in the diagram 
\[
 \xymatrix{(0,1] \times \R \ar[dr]^{p} \ar[d]_{q} &   	\\
		(0,1] \times \partial \mathbb{D}  \ar[r]_{\ \ \ h} & \mathbb{D}\backslash\{0\}  }
\]
where 
\[
		q(r,\theta):=\big(r,e^{i\theta}\big) \qquad\mbox{and}\qquad h(r,z):=rz.  
\]
Since $h$ is a homeomorphism and $q$ is a covering map, the lift $\ring{\Phi}$ of the restriction of $\varphi$ to $\mathbb{D}\setminus \{0\}$ can also be obtained by first conjugating $\varphi|_{\mathbb{D}\setminus \{0\}}$ via $h$ and then lifting 
via $q$.   That is, $\ring{\Phi}$ is a lift of the map 
\[
\ring{\phi}:(0,1] \times \partial \mathbb{D} \rightarrow(0,1]\times \partial \mathbb{D},
\qquad 	\ring{\phi} :=h^{-1}\circ \varphi|_{\mathbb{D}\setminus \{0\}} \circ h,
\]
via $p$. Now we claim that $\ring{\phi}$ extends to a $C^{k-1}$ map
\[
\phi:[0,1] \times \partial \mathbb{D}\rightarrow [0,1] \times \partial \mathbb{D}.  
\]
Before we prove this claim, observe that statement (i) will follow immediately, because $q$ extends to a smooth covering map from $[0,1]\times \R$ to $[0,1]\times \partial \mathbb{D}$,
and so lifting via $q$ one loses no regularity.   

Now we prove the claim. It is straightforward to verify that $\ring{\phi}$ has the form
\[
\ring{\phi}(r,z)=\left(|\varphi(rz)|,\frac{\varphi(rz)}{|\varphi(rz)|}\right)
\]
for all $r\in(0,1]$ and all $z\in\partial \mathbb{D}$.   Now, since $k\geq1$, we have for all $(x,y)\in\mathbb{D}$, where $x,y$ are the standard Euclidean coordinates, 
\begin{align*} 
					\varphi(x,y)&=\varphi(0)+\int_0^1D\varphi(tx,ty)[(x,y)]\,dt\\
							&=x \int_0^1 D_1 \varphi (tx,ty)\,dt\ +\ y \int_0^1D_2 \varphi (tx,ty)\,dt\\
							&=x \, \varphi_1(x,y)\ + y\, \varphi_2(x,y),
\end{align*}
where $\varphi_1,\varphi_2:\mathbb{D}\rightarrow\R$ are $C^{k-1}$ functions.  We observe that for $r\in (0,1]$ and $z=(x,y)\in \partial \mathbb{D}$ we have
\[
\frac{\varphi(rz)}{r}=x\, \varphi_1(rz)\ + y\, \varphi_2(rz),
\]
and the map
\[
\psi: [0,1] \times \partial \mathbb{D} \rightarrow [0,1] \times \partial \mathbb{D}, \qquad \psi(r,z) := x\, \varphi_1(rz)\ + y\, \varphi_2(rz) \quad \mbox{for } z = (x,y),
\]
is of class $C^{k-1}$ on the whole $[0,1]\times \partial \mathbb{D}$. We can express $\ring{\phi}$ in terms of $\psi$ as
\begin{equation}\label{E:extension}
\ring{\phi}(r,z)=\left(|\varphi(rz)|,\frac{\varphi(rz)}{r}\frac{r}{|\varphi(rz)|}\right)=\left(r |\psi(r,z)|, \frac{\psi(r,z)}{|\psi(r,z)|}\right) \qquad \forall (r,z)\in (0,1]\times \partial \D.
\end{equation}
But $\psi$ nowhere vanishes: it does not vanish for $0<r\leq1$ because in this case
\[
\psi(r,z)=\frac{\varphi(rz)}{r} \neq 0,
\] 
and it does not vanish for $r=0$ because
\begin{equation}
\label{for0}
\psi(0,z)=x\, \varphi_1(0)+y\, \varphi_2(0)=D\varphi(0)[z],
\end{equation}
and $D\varphi(0)$ is non-singular.   Thus the $C^{k-1}$ map $\psi$ nowhere vanishes, and the right-hand side of (\ref{E:extension}) defines a map of class $C^{k-1}$ on $[0,1] \times \partial \mathbb{D}$ which extends $\ring{\phi}$.   
This proves the claim and statement (i).
 
Now assume that $k\geq 2$ and let $(R,\Theta)$ be the two components of the $C^{k-1}$ map $\Phi$. From (\ref{E:extension}) we see that
\[
R(r,\theta) = |\varphi(re^{i\theta})| = r |\psi(r,e^{i\theta})|.
\]
Since $\psi$ is differentiable and does not vanish, we find by (\ref{for0})
\[
D_1 R(0,\theta) = |\psi(0,e^{i\theta})| =  \bigl| D\varphi(0)[e^{i\theta}] \bigr|,
\]
which proves (ii).
\end{proof} 

Now let $\varphi\in \mathrm{Diff}(\mathbb{D},\omega)$ be a diffeomorphism which fixes the origin. By the above lemma, any lift of $\varphi|_{\mathbb{D}\setminus \{0\}}$ via the covering map $p:(0,1]\times \R \rightarrow \mathbb{D}\setminus \{0\}$ extends to a smooth diffeomorphism $\Phi$ of $S = [0,1]\times \R$. With a little abuse of terminology, we call also this extension $\Phi$ a {\em lift} of $\varphi$. 

On $S$ we consider the 2-form 
\[
\Omega := p^* \omega,
\]
which is positive on the interior of $S$, because $p$ is orientation preserving, and satisfies the periodicity condition
\begin{equation}
\label{perT}
T^* \Omega = \Omega,
\end{equation}
where $T$ is the map which is defined in (\ref{deckT}). By construction, any lift $\Phi: S \rightarrow S$ of the diffeomorphism $\varphi\in \mathrm{Diff}(\mathbb{D},\omega)$  fixing the origin is a smooth diffeomorphism which satisfies the following conditions:
\begin{enumerate}[(a)]
\item $\Phi \circ T = T \circ \Phi$;
\item $\Phi$ maps each connected component of $\partial S$ into itself;
\item $\Phi^* \Omega = \Omega$.
\end{enumerate}

The punctured disk $\D\setminus \{0\}$ is foliated by circles centred at $0$ and rays emanating from $0$. These two foliations are orthogonal to each other.

\begin{defn}
\label{radmono}
A diffeomorphism $\varphi\in \mathrm{Diff}(\mathbb{D})$ fixing the origin is said to be {\em radially monotone} if the image through $\varphi$ of the foliation by rays  is transverse to the foliation by circles. 
\end{defn}

It is easy to translate the condition of radial monotonicity for $\varphi$ into a condition for the lift $\Phi$.

\begin{defn}
\label{mono}
The diffeomorphism $\Phi: S \rightarrow S$, $\Phi=(R,\Theta)$, is said to be monotone if $D_1 R(r,\theta)>0$ for every $(r,\theta)\in S$.
\end{defn}

Notice that two different lifts of $\varphi$ have the same first component $R$, so monotonicity does not depend on the choice of the lift.

\begin{lem}
\label{monlem}
Assume that $\varphi\in \mathrm{Diff}(\mathbb{D},\omega)$ fixes the origin. Then $\varphi$ is  radially monotone in the sense of Definition \ref{radmono} if and only if a lift $\Phi: S \rightarrow S$ of $\varphi$ (and hence every lift of $\varphi$) is monotone in the sense of Definition \ref{mono}.
\end{lem}

\begin{proof}
If we set $\Phi=(R,\Theta)$ we have
\[
\varphi(r e^{i\theta}) = R(r,\theta) e^{i\Theta(r,\theta)} \qquad \forall (r,\theta)\in S.
\]
By definition, the diffeomorphism $\varphi$ is radially monotone if and only if
\[
D_1 R(r,\theta) \neq 0 \qquad \forall (r,\theta)\in (0,1] \times \R.
\]
By Lemma \ref{L:extn_of_lifts} (ii) we always have
\[
D_1 R(0,\theta) > 0 \qquad \forall \theta\in \R,
\]
and the thesis follows from the continuity of $D_1 R$ on $S$.
\end{proof}

\subsection{Generalised generating functions}
\label{gensec}

Let $\Omega$ be a smooth 2-form on $S$ which is positive on $\mathrm{int}(S)$ and satisfies the periodicity condition (\ref{perT}). Let $\Phi: S \rightarrow S$ be a diffeomorphism of the strip which satisfies conditions (a), (b), (c) of the previous section and is monotone in the sense of Definition \ref{mono}. The aim of this section is to show that $\Phi$ can be represented in terms of a suitable generating function. This would be a standard generating function of the form $W=W(R,\theta)$ if $\Omega$ is the standard form $\omega = dr\wedge d\theta$. Indeed, in this case the assumptions on $\Phi$ imply that there exists a smooth function $W: S \rightarrow \R$ which is $2\pi$-periodic in the variable $\theta$ and unique up to the addition of a constant such that if $(R,\Theta)$ are the two components of $\Phi(r,\theta)$ then $\Phi$ is determined by the identities
\begin{equation}
\label{genefunc}
R = r + D_2 W(R,\theta), \qquad \Theta = \theta - D_1 W(R,\theta).
\end{equation}
See e.g.\ \cite[Chapter 9]{ms98}. Identities of this sort are quite desirable, for instance because they reduce the problem of finding fixed points of $\Phi$ to that of finding critical points of $W$.
The fact that the 2-form $\Omega$ is allowed to be more general will make the dependence of $\Phi$ on the generating function $W$ somehow more implicit, but the correspondence between fixed points of $\Phi$ and critical points of $W$ will be preserved.

The 2-form $\Omega$ can be written as
\[
\Omega(r,\theta) = F(r,\theta) \, dr\wedge d\theta,
\]
where $F$ is a smooth function on $S$, which is positive on $\mathrm{int}(S)$ and satisfies
\[
F(r,\theta+2\pi) = F(r,\theta) \qquad \forall (r,\theta)\in S.
\]
The form $\Omega$ has the following two privileged primitives
\begin{equation}
\label{ab}
\begin{array}{rclrl}
\Lambda_A(r,\theta) & := & A(r,\theta) \, d\theta, & \qquad \mbox{where} & \quad \displaystyle{A(r,\theta):= \int_0^r F(s,\theta)\, ds,} \\ \\
\Lambda_B(r,\theta) & := & -B(r,\theta) \, dr, & \qquad \mbox{where} & \quad \displaystyle{B(r,\theta):= \int_0^{\theta} F(r,\vartheta)\, d\vartheta.}
\end{array}
\end{equation}
The primitive $\Lambda_A$ is also periodic, while $\Lambda_B$ is not. Indeed, we have
\begin{equation}
\label{perio}
A(r,\theta+2\pi) = A(r,\theta), \qquad B(r,\theta+2\pi) = B(r,\theta) + B(r,2\pi) \qquad \forall (r,\theta)\in S.
\end{equation}
Actually, if $\Omega$ is the lift via $p$ of a smooth 2-form $\omega$ on $\mathbb{D}$ then  $\Lambda_A$ is the lift of a smooth primitive $\lambda_a$ of $\omega$:

\begin{lem}
\label{lambda_a}
Assume that $\Omega$ is the lift via $p$ of the smooth 2-form
\[
\omega(x,y)  = f(x,y)\, dx \wedge dy,
\]
where $f$ is a smooth function on $\D$. Then $\Lambda_A = p^*\lambda_a$, where $\lambda_a$ is the following smooth primitive of $\omega$ on $\mathbb{D}$:
\[
\lambda_a(x,y) := a(x,y) (x\, dy - y \, dx) , \qquad   a(x,y) := \int_0^1t\, f(tx,ty)\, dt.
\]
\end{lem}

\begin{proof}
Since
\[
\Omega = p^* \omega = f(r\cos \theta,r\sin\theta) p^* (dx\wedge dy) = f(r\cos \theta,r\sin\theta) r \, dr\wedge d\theta,
\]
the functions $F$ and $f$ are related by the identity
\begin{equation}
\label{Ff}
F(r,\theta) = r \, f(r\cos \theta,r\sin\theta).
\end{equation}
From the identity
\[
p^*(x\, dy - y\, dx ) = r^2 \, d\theta,
\]
we find
\[
p^* \lambda_a =  r^2 a(r\cos \theta,r\sin \theta) \, d\theta.
\]
By using the definition of $a$ and (\ref{Ff}), the coefficient of $d\theta$ in the above expression can be manipulated as follows:
\[
\begin{split}
r^2 a(r\cos \theta,r\sin \theta) &= r^2 \int_0^1 t \, f(tr \cos \theta,tr\sin\theta) \, dt = r \int_0^1 tr \, f(tr \cos \theta,tr\sin\theta) \, dt \\ &= r \int_0^1 F(tr,\theta)\, dt = \int_0^r F(s,\theta)\, ds = A(r,\theta).
\end{split}
\]
Therefore, 
\[
p^* \lambda_a = A(r,\theta)\, d\theta = \Lambda_A,
\]
as claimed.
\end{proof}

Set $\Phi=(R,\Theta)$. The fact that $\Phi$ is monotone implies that the function $r \mapsto R(r,\theta)$ is a diffeomorphism of $[0,1]$ onto itself, for every $\theta\in \R$, and this implies that the map
\begin{equation}
\label{lapsi}
\Psi: S \rightarrow S, \qquad \Psi(r,\theta) := (R(r,\theta),\theta),
\end{equation}
is a smooth diffeomorphism. This fact allows us to work with coordinates $(R,\theta)$ on $S$: in the following, we shall see $r=r(R,\theta)$ and $\Theta=\Theta(R,\theta)$ as smooth functions of $(R,\theta)$. Consider the 1-form
\[
\Xi (R,\theta) = \bigl( A(R,\theta) - A(r,\theta) \bigr) \, d\theta + \bigl( B(R,\theta) - B(R,\Theta) \bigr) \, dR.
\]
By (\ref{perio}), $\Xi$ is $2\pi$-periodic:
\[
\begin{split}
\Xi(R,\theta+2\pi) &= \bigl( A(R,\theta+2\pi) - A(r,\theta+2\pi) \bigr) \, d\theta + \bigl( B(R,\theta+2\pi) - B(R,\Theta+2\pi) \bigr) \, dR \\ &= \bigl( A(R,\theta) - A(r,\theta) \bigr) \, d\theta + \bigl( B(R,\theta) + B(R,2\pi) - B(R,\Theta) - B(R,2\pi) \bigr) \, dR \\ &= \Xi(R,\theta).
\end{split}
\]
Using the fact that $\Phi$ preserves $\Omega$, we compute
\[
\begin{split}
d\Xi &= D_1 A (R,\theta) \, dR\wedge d\theta - D_1 A(r,\theta)\, dr\wedge d\theta - D_2 B(R,\theta)\, dR\wedge d\theta + D_2 B(R,\Theta) \, dR \wedge d\Theta \\ &= F(R,\theta)  \, dR\wedge d\theta - F(r,\theta)\, dr\wedge d\theta -F(R,\theta)\, dR\wedge d\theta +F(R,\Theta) \, dR \wedge d\Theta \\ &= F(R,\Theta)\, dR\wedge d\Theta - F(r,\theta)\, dr\wedge d\theta = \Phi^*\Omega -  \Omega = 0.
\end{split}
\]
Therefore, the 1-form $\Xi$ is closed and hence exact on $S$. Let $W=W(R,\theta)$ be a primitive of $\Xi$. We claim that $W$ is periodic. Indeed, this follows from the fact that the integral of $\Xi$ on the path $\alpha$ on $\partial S$ from $(0,0)$ to $(0,2\pi)$ vanishes:
\[
\int_{\alpha} \Xi = \int_0^{2\pi} \bigl( A(0,\theta) - A(0,\theta) \bigr)\, d\theta = 0.
\]
The periodic function $W$ is called a {\em generating function} for the monotone diffeomorphism $\Phi$ with respect to the invariant 2-form $\Omega$. It is defined up to the addition of a constant and satisfies
\begin{eqnarray}
\label{W1}
D_1 W(R,\theta) &=& B(R,\theta) - B(R,\Theta), \\
\label{W2}
D_2 W(R,\theta) &=& A(R,\theta) - A(r,\theta),
\end{eqnarray}
where $\Phi(r,\theta) = (R,\Theta)$. Notice that (\ref{W2}) implies that $W$ is constant on each component of the boundary of $S$. Notice also that (\ref{W1}) and (\ref{W2}) reduce to (\ref{genefunc}) in the case $\Omega = dr\wedge d\theta$.

The function $r \mapsto A(r,\theta)$ is strictly increasing on $[0,1]$ for every $\theta\in \R$, because $D_1 A(r,\theta) = F(r,\theta) >0$ for $r\in (0,1)$. Similarly, the function $\theta \mapsto B(R,\theta)$ is strictly increasing  on $\R$ for every $R\in (0,1)$, because $D_1 B(R,\theta) = F(R,\theta) >0$ for $R\in (0,1)$. These two facts imply that $(R,\theta)$ is an interior critical point of $W$ if and only if $(R,\theta)$ is an interior fixed point of $\Phi$.
We summarise the above discussion into the following:

\begin{prop}
Let $\Phi: S \rightarrow S$ be a diffeomorphism satisfying conditions (a), (b) and (c). If $\Phi$ is monotone then there is a smooth periodic function $W: S \rightarrow \R$, which is unique up to the addition of a real constant, such that (\ref{W1}) and (\ref{W2}) hold. Moreover, $W$ is constant on each of the two boundary components of $S$. Finally, the interior critical points of $W$ are precisely the fixed points of $\Phi$.
\end{prop}

\subsection{The action in terms of the generating function}
\label{actionsec}

Throughout this section $\tilde{\varphi}=[\{\varphi_t\}]$ is an element of $\widetilde{\mathrm{Diff}}(\mathbb{D},\omega)$ which satisfies assumption (i) of Theorem \ref{fixpointmon} from the Introduction: $\varphi := \pi(\tilde{\varphi})$ fixes the origin and is radially monotone. 

The isotopy $\{\varphi_t\}$ lifts uniquely to a path of diffeomorphisms $\Phi_t: S \rightarrow S$ such that $\Phi_0 = \mathrm{id}$. We denote by $\Phi:= \Phi_1$ the resulting lift of $\varphi$, which does not depend on the choice of the representative $\{\varphi_t\}$ of $\tilde{\varphi}$. 

The diffeomorphism $\Phi$ satisfies conditions (a), (b), (c) of the previous section, where $\Omega= p^* \omega$, and is monotone. Therefore, $\Phi$ admits a generating function $W: S \rightarrow \R$, which we normalise by requiring that
\begin{equation}
\label{norm}
W(1,\theta) = 0 \qquad \forall \theta\in \R.
\end{equation}
The aim of this section is to relate the action of $\tilde{\varphi}$ to the generating function $W$. We begin by determining the value of $W$ on the boundary component $\{0\} \times \R$:

\begin{lem}
\label{Wsotto}
The value of $W$ on the boundary component $\{0\} \times \R$ coincides with the action of the fixed point $0\in \mathbb{D}$ with respect to $\tilde{\varphi}$: 
\[
W(0,\theta) = \sigma_{\tilde{\varphi}} (0) \qquad \forall \theta\in \R.
\]
\end{lem}

\begin{proof}
Fix an arbitrary $\theta\in \R$. Then we have by (\ref{W1})
\[
\begin{split}
W(0,\theta) &= W(1,\theta) - \int_0^1 D_1 W(R,\theta)\, dR = - \int_0^1 \bigl( B(R,\theta) - B(R,\Theta) \bigr)\, dR \\ &= - \int_0^1 B(r,\theta) \, dr + \int_0^1 B(R,\Theta)\, dR =  \int_{\gamma} \Lambda_B - \int_{\gamma} \Phi^*(\Lambda_B)  = \int_{\gamma} \Lambda_B - \int_{\Phi\circ \gamma} \Lambda_B,
\end{split}
\]
where $\gamma: [0,1] \rightarrow S$ is the path $\gamma(r) = (r,\theta)$. Denote by $\alpha_0:[0,1]\rightarrow \partial S$, respectively $\alpha_1 : [0,1] \rightarrow \partial S$, the affine path going from $\gamma(0)$, resp.\ $\gamma(1)$, to $\Phi(\gamma(0))$, resp.\ $\Phi(\gamma(1))$. Since the integral of $\Lambda_B$ on paths contained in $\partial S$ vanishes, we have
\[
W(0,\theta) = \int_{\gamma \# \alpha_1 \# (\Phi\circ \gamma)^{-1} \# \alpha_0^{-1}} \Lambda_B.
\]
Let $\lambda$ be any smooth primitive of $\omega$ on $\mathbb{D}$. Then $p^* \lambda$ is a primitive of $\Omega$ on $S$, and hence it differs from $\Lambda_B$ by an exact 1-form. Since the path $\gamma \# \alpha_1 \# (\Phi\circ \gamma)^{-1} \# \alpha_0^{-1}$ is closed, we can continue the computation by replacing $\Lambda_B$ by $p^* \lambda$: 
\[
\begin{split}
W(0,\theta) &= \int_{\gamma \# \alpha_1 \# (\Phi\circ \gamma)^{-1} \# \alpha_0^{-1}} p^*\lambda = \int_{p\circ \gamma} \lambda + \int_{p\circ \alpha_1} \lambda- \int_{p\circ \Phi\circ \gamma}\lambda - \int_{p\circ \alpha_0} \lambda \\ &= \int_{p\circ \gamma} \lambda + \int_{p\circ \alpha_1} \lambda- \int_{\varphi \circ p\circ \gamma}\lambda - 0 = \int_{p\circ \alpha_1} \lambda - \int_{p\circ \gamma} \bigl( \varphi^* \lambda - \lambda \bigr) .
\end{split}
\]
Using the definition of the action $\sigma_{\tilde{\varphi},\lambda}$ and the fact that the path $p\circ \alpha_1$ is homotopic to the path $t\mapsto \varphi_t(e^{i\theta})$ within $\partial \mathbb{D}$, because $\alpha_1$ is homotopic to $t\mapsto \Phi_t(1,\theta)$ within $\partial S$, we finally obtain
\[
W(0,\theta) = \sigma_{\tilde{\varphi},\lambda}(e^{i\theta}) - \int_{p\circ \gamma} d \sigma_{\tilde{\varphi},\lambda} = \sigma_{\tilde{\varphi},\lambda}(e^{i\theta}) - \sigma_{\tilde{\varphi},\lambda}(e^{i\theta}) + \sigma_{\tilde{\varphi},\lambda}(0)  = \sigma_{\tilde{\varphi},\lambda}(0),
\]
as we wished to prove.
\end{proof}

As we have seen in Lemma \ref{lambda_a}, $\Lambda_A$ is the lift via $p$ of a smooth primitive $\lambda_a$ of $\omega$ on $\mathbb{D}$. We shall express the action of $\tilde{\varphi}$ with respect to $\lambda_a$ in terms of the generating function $W$. 

Let $\Sigma:= p^* \sigma_{\tilde{\varphi},\lambda_a}: S \rightarrow \R$ be the lift of the action of $\tilde{\varphi}$ with respect to $\lambda_a$. This function is uniquely determined by the conditions
\begin{equation}
\label{detSigma}
d \Sigma = \Phi^* \Lambda_A - \Lambda_A, \qquad \Sigma(1,\theta) = \sigma_{\tilde{\varphi},\lambda_a} (e^{i\theta}) \quad \forall \theta\in \R.
\end{equation}
In order to express the function $\Sigma$ in terms of $W$,  we introduce the smooth function
\[
G: S \rightarrow \R, \qquad G(r,\theta) := \int_0^r \left( \int_0^{\theta} F(s,\vartheta)\, d\vartheta \right) \, ds,
\]
which satisfies
\begin{equation}
\label{derh}
D_1 G (r,\theta) = \int_0^\theta F(r,\vartheta)\, d\vartheta = B(r,\theta), \qquad D_2 G (r,\theta) = \int_0^r F(s,\theta)\, ds = A(r,\theta),
\end{equation}
and
\begin{equation}
\label{translH}
G\circ T(r,\theta) - G(r,\theta) = G(r,2\pi) \qquad \forall (r,\theta)\in S.
\end{equation}

\begin{prop}
\label{forac}
The lifted action $\Sigma$ has the following form
\begin{equation}
\label{forSigma}
\Sigma(r,\theta) = W(R,\theta) + G(\Phi(r,\theta)) - G(R,\theta),
\end{equation}
where $\Phi(r,\theta) = (R,\Theta)$.  In particular,
\[
\Sigma(r,\theta) = W(r,\theta)
\]
for every fixed point $(r,\theta)$ of $\Phi$.
\end{prop}

\begin{proof}
Using (\ref{W1}), (\ref{W2}) and (\ref{derh}) we compute
\[
\begin{split}
\Phi^* \Lambda_A &- \Lambda_A =   A(R,\Theta) \, d\Theta - A(r,\theta) \, d\theta = A(R,\Theta)\, d\Theta - A(R,\theta)\, d\theta + D_2 W(R,\theta)\, d\theta \\ &=  A(R,\Theta)\, d\Theta - A(R,\theta)\, d\theta + dW(R,\theta) - D_1 W(R,\theta)\, dR \\ &= dW(R,\theta) +  A(R,\Theta)\, d\Theta - A(R,\theta)\, d\theta + B(R,\Theta)\, dR - B(R,\theta)\, dR \\ &=  dW(R,\theta) + D_2 G(R,\Theta)\, d\Theta - D_2 G(R,\theta)\, d\theta + D_1 G(R,\Theta)\, dR - D_1 G(R,\theta)\, dR \\ &= dW(R,\theta) + dG(R,\Theta) - dG(R,\theta).
\end{split}
\]
This shows that the function which appears on the right hand side of (\ref{forSigma}) satisfies the first condition in (\ref{detSigma}). Moreover, by the normalisation condition (\ref{norm}) we find that this function takes the following value at $(1,\theta)$:
\[
W(1,\theta) + G(1,\Theta) - G(1,\theta) = \int_{\theta}^{\Theta} D_2 G(1,\vartheta)\, d\vartheta = \int_{\theta}^{\Theta} A(1,\vartheta)\, d\vartheta = \int_{\alpha_1} \Lambda_A,
\]
where $\alpha_1: [0,1] \rightarrow \partial S$ is the affine path going from $(1,\theta)$ to $(1,\Theta)$. Since $\alpha_1$ is homotopic to $t\mapsto \Phi_t(1,\theta)$ within $\partial S$, the projected path $p\circ \alpha_1$ is homotopic to the path $t\mapsto \varphi_t(e^{i\theta})$ within $\partial \D$, and hence
\[
\int_{\alpha_1} \Lambda_A = \int_{\alpha_1} p^* \lambda_a = \int_{p\circ \alpha_1} \lambda_a = \sigma_{\tilde{\varphi},\lambda_a} (e^{i\theta}).
\]
This shows that the function which appears on the right hand side of (\ref{forSigma}) satisfies also the second condition in (\ref{detSigma}), and hence it coincides with the function $\Sigma$.
\end{proof}

\subsection{The Calabi invariant in terms of the generating function}

Throughout this section we keep the same assumptions and notation of the last one. Our aim is to  compute the Calabi invariant of $\tilde{\varphi}$ in terms of the generating function $W$.
 
Denote by 
\[
Q:= [0,1]\times [0,2\pi]
\]
a fundamental domain in $S$. Then we have
\begin{equation}
\label{calfor}
\mathrm{CAL}(\tilde{\varphi},\omega) = \int_{\mathbb{D}} \sigma_{\tilde{\varphi},\lambda_a}\, \omega = \int_Q p^* \bigl( \sigma_{\tilde{\varphi},\lambda_a}\, \omega \bigr) =
\int_Q \Sigma\, \Omega.
\end{equation}
We shall compute the last integral by using formula (\ref{forSigma}) in Proposition \ref{forac}.  In the next two Lemmata, we compute the integral of the second and third term on the right hand side of (\ref{forSigma}). In order to simplify the formulas, we introduce the function
\[
\tilde{\Theta} = \Theta \circ \Psi^{-1},
\]
where $\Psi$ is the diffeomorphism which is defined in (\ref{lapsi}). In other words, $\tilde{\Theta}$ is the second component of $\Phi$ seen as a function of $(R,\theta)$.

\begin{lem}
\label{L1}
The following formula holds
\[
\int_Q G(\Phi(r,\theta))\, \Omega(r,\theta) = \int_0^1 G(R,2\pi) B(R,\tilde{\Theta}(R,2\pi))\, dR - \int_Q A(r,\theta)B(r,\theta)\, dr\wedge d\theta.
\]
\end{lem}

\begin{proof}
Let $\gamma:[0,1] \rightarrow S$ be the path $\gamma(r)=(r,0)$, and let $\alpha_0$ and $\alpha_1$ be the affine paths in $\partial S$ going from $\gamma(0)$ and $\gamma(1)$ to $\Phi(\gamma(0))$ and $\Phi(\gamma(1))$, respectively. Let $\Gamma$ be the 2-chain in $S$ whose boundary is $\gamma+\alpha_1-\Phi\circ \gamma-\alpha_0$. Then we have
\[
\Phi(Q) - Q = T(\Gamma) - \Gamma,
\]
where $T$ is the map defined in (\ref{deckT}).
Therefore, using the fact that $\Phi$ preserves the $T$-invariant 2-form $\Omega$ together with (\ref{translH}), we find
\begin{equation}
\label{lem1}
\begin{split}
\int_Q G\circ \Phi\, \Omega &=
\int_{\Phi(Q)} G \, \Omega = \int_{Q+T(\Gamma) - \Gamma} G\, \Omega = \int_Q  G \, \Omega + \int_{\Gamma} \bigl(T^* (G\, \Omega) - G \, \Omega\bigr) \\ &= 
\int_Q  G \, \Omega + \int_{\Gamma} (T^* G-G) \, \Omega =  \int_Q  G \, \Omega + \int_{\Gamma} G(r,2\pi) \, \Omega(r,\theta).
\end{split}
\end{equation}
We compute the last two integrals separately. The first one can be manipulated using (\ref{derh}) as follows:
\begin{equation}
\label{lem2}
\begin{split}
 \int_Q G(r,\theta) F(r,\theta) \, dr\wedge d\theta & = \int_0^1 \left( \int_0^{2\pi} G(r,\theta) F(r,\theta) \, d\theta \right) \, dr \\ &= \int_0^1 \left( \Bigl[ G(r,\theta) B(r,\theta) \Bigr]_{\theta=0}^{\theta=2\pi} - \int_0^{2\pi} A(r,\theta) B(r,\theta) \, d\theta \right) \, dr \\ &= \int_0^1 G(r,2\pi) B(r,2\pi)\, dr - \int_Q A(r,\theta) B(r,\theta)\, dr\wedge d\theta.
\end{split}
\end{equation}
Since
\[
d\bigl( G(r,2\pi)B(r,\theta)\, dr \bigr)= G(r,2\pi) F(r,\theta) \, d\theta\wedge dr = - G(r,2\pi) \, \Omega(r,\theta),
\]
the second integral can be computed using Stokes theorem:
\begin{equation}
\label{lem3}
\begin{split}
\int_{\Gamma} G(r,2\pi) \, \Omega(r,\theta) &= - \int_{\partial \Gamma}  G(r,2\pi)B(r,\theta)\, dr = -\int_{\gamma+\alpha_1-\Phi\circ \gamma-\alpha_0}  G(r,2\pi)B(r,\theta)\, dr \\ &= \int_{\Phi\circ \gamma} G(r,2\pi)B(r,\theta)\, dr \\&= \int_0^1 G(R(r,0),2\pi) B(R(r,0),\Theta(r,0)) D_1 R (r,0)\, dr \\ &= \int_0^1 G(R,2\pi) B(R,\tilde{\Theta}(R,0))\, dR.
\end{split}
\end{equation}
The desired formula follows from (\ref{lem1}), (\ref{lem2}) and (\ref{lem3}) together with the identity
\[
B(R,2\pi) + B(R,\tilde{\Theta}(R,0)) = B(R,\tilde{\Theta}(R,0)+2\pi) = B(R,\tilde{\Theta}(R,2\pi)),
\]
which follows from (\ref{perio}).
\end{proof}

\begin{lem}
\label{L2}
If the generating function $W$ is normalised by (\ref{norm}) then we have the formula
\[
\begin{split}
\int_Q G(R(r,\theta),\theta) \, \Omega(r,\theta)  &=  - \int_Q W(r,\theta)\, \Omega(r,\theta) - \int_Q  A(r,\theta) B(r,\theta)\, dr \wedge d\theta \\ &+ \int_0^1 G(R,2\pi) B(R,\tilde{\Theta}(R,2\pi)) \, dR.
\end{split}
\]
\end{lem}

\begin{proof}
Using the fact that $\Phi$ preserves $\Omega$, we can manipulate the integral of $G(R,\theta)\Omega(r,\theta)$ as follows:
\begin{equation}
\label{cal2}
\begin{split}
\int_Q G(R(r,\theta),\theta) \, \Omega(r,\theta) &=  \int_Q G(R(r,\theta),\theta) \, \Phi^* \Omega(r,\theta) \\ &= \int_Q G(R(r,\theta),\theta) F(R(r,\theta),\Theta(r,\theta)) \, dR(r,\theta) \wedge d\Theta(r,\theta) \\ &= \int_{\Psi(Q)} G(R,\theta) F(R,\tilde{\Theta}(R,\theta)) \, d R\wedge d\tilde{\Theta}(R,\theta) \\ &= \int_{Q}G(R,\theta) F(R,\tilde{\Theta}(R,\theta)) \, d R\wedge d\tilde{\Theta}(R,\theta).
\end{split}
\end{equation}
In the last line we have used the fact that the diffeomorphism $\Psi$ which is defined in (\ref{lapsi}) satisfies $\Psi(Q)=Q$. By plugging in the identity
\[
d\tilde{\Theta}(R,\theta) = D_1 \tilde{\Theta}(R,\theta) \, dR + D_2 \tilde{\Theta}(R,\theta) \, d\theta,
\]
we can manipulate the latter integral further and obtain
\begin{equation}
\label{cal3}
\begin{split}
\int_{Q} & G(R,\theta)  F(R,\tilde{\Theta}(R,\theta)) \, d R\wedge d\tilde{\Theta}(R,\theta) \\ &= \int_Q G(R,\theta)  F(R,\tilde{\Theta}(R,\theta)) D_2 \tilde{\Theta}(R,\theta) \, dR\wedge d\theta \\ &= \int_Q G(R,\theta)  \frac{\partial}{\partial \theta} B(R,\tilde{\Theta}(R,\theta))\, dR \wedge d\theta \\ &= \int_0^1 \left( \int_0^{2\pi} G(R,\theta)  \frac{\partial}{\partial \theta} B(R,\tilde{\Theta}(R,\theta))\, d\theta \right)\, dR \\ & =\int_0^1 \left( \Bigl[G(R,\theta) B(R,\tilde{\Theta}(R,\theta)) \Bigr]_{\theta=0}^{\theta=2\pi} - \int_0^{2\pi} A(R,\theta) B(R,\tilde{\Theta}(R,\theta)) \, d\theta \right) \, dR \\ &= \int_0^1 G(R,2\pi) B(R,\tilde{\Theta}(R,2\pi)) \, dR - \int_Q A(R,\theta) B(R,\tilde{\Theta}(R,\theta)) \, dR \wedge d\theta.
\end{split}
\end{equation}
The latter integral can be manipulated using (\ref{W1}):
\begin{equation}
\label{cal4}
\begin{split}
&\int_Q A(R,\theta) B(R,\tilde{\Theta}(R,\theta))\, dR \wedge d\theta \\&= \int_Q A(R,\theta) \bigl( B(R,\tilde{\Theta}(R,\theta)) - B(R,\theta) \bigr)\, dR \wedge d\theta + \int_Q  A(R,\theta) B(R,\theta)\, dR \wedge d\theta \\ &= -\int_Q A(R,\theta) D_1 W(R,\theta) \, dR \wedge d\theta + \int_Q  A(r,\theta) B(r,\theta)\, dr \wedge d\theta \\ &= -\int_0^{2\pi} \left( \int_0^1 A(R,\theta) D_1 W(R,\theta) \, dR \right) \, d\theta + \int_Q  A(r,\theta) B(r,\theta)\, dr \wedge d\theta  \\ &= - \int_0^{2\pi} \left( \Bigl[ A(R,\theta) W(R,\theta) \Bigr]_{R=0}^{R=1} - \int_0^1 F(R,\theta) W(R,\theta)\, dR \right) \, d\theta \\ &\qquad + \int_Q  A(r,\theta) B(r,\theta)\, dr \wedge d\theta  \\ &=  \int_Q W(r,\theta)\, \Omega(r,\theta) + \int_Q  A(r,\theta) B(r,\theta)\, dr \wedge d\theta.
\end{split}
\end{equation}
The desired formula follows from (\ref{cal2}), (\ref{cal3}) and (\ref{cal4}).
\end{proof}

From (\ref{calfor}), Proposition \ref{forac}, Lemma \ref{L1} and Lemma \ref{L2} we immediately deduce the following:

\begin{prop}
\label{forcal}
The Calabi invariant of $\tilde{\varphi}$ can be expressed by the formula.
\[
\mathrm{CAL}(\tilde{\varphi},\omega) = \int_Q (W + W\circ \Psi) \, \Omega = \int_Q W \, (\Omega + (\Psi^{-1})^*\Omega),
\]
where $\Psi: Q \rightarrow Q$ is the diffeomorphism defined by $\Psi(r,\theta) = (R(r,\theta),\theta)$, with
$\Phi=(R,\Theta)$.
\end{prop}

\subsection{The proof of the fixed point theorem and a counterexample}
\label{proofsec}

We can finally prove Theorem \ref{fixpointmon} from the introduction.

\begin{proof}[Proof of Theorem \ref{fixpointmon}] 
Let $\Phi: S \rightarrow S$ be the lift of $\varphi$ which is determined at the beginning of Section \ref{actionsec}. By Lemma \ref{monlem}, $\Phi$ is monotone. Let
$W$ be the generating function of $\Phi$, normalised by (\ref{norm}), i.e.\ $W=0$ on $\{1\}\times \R$. 

From the fact that $\tilde\varphi$ is not the identity we deduce that $\Phi$ is not the identity and hence $W$ is not identically zero. Then the condition $\mathrm{CAL}(\tilde{\varphi},\omega)\leq 0$ and the formula of Proposition \ref{forcal} imply that $W$ is somewhere negative. Being periodic, $W$ achieves its minimum at a point $(R_0,\theta_0)\in Q$, where $W(R_0,\theta_0)<0$. Since $W$ vanishes on $\{1\}\times \R$, $R_0$ belongs to $[0,1)$. Since $W$ is not constant, its value at its minimum is strictly smaller than its average over $Q$ with respect to the area form $\Omega + (\Psi^{-1})^* \Omega$, and from the formula of Proposition \ref{forcal} we obtain
\begin{equation}
\label{sotto}
W(R_0,\theta_0) < \frac{\int_Q W (\Omega + (\Psi^{-1})^* \Omega)}{\int_Q (\Omega + (\Psi^{-1})^* \Omega)} = \frac{\mathrm{CAL}(\tilde{\varphi},\omega)}{2 \int_Q \Omega} = \frac{\mathrm{CAL}(\tilde{\varphi},\omega)}{2 \int_{\D} \omega},
\end{equation}
where we have used the fact that 
\[
\int_Q (\Psi^{-1})^* \Omega = \int_{\Psi^{-1}(Q)} \Omega = \int_Q \Omega.
\]
Set 
\[
z_0:= p(R_0,\theta_0) = R_0 e^{i\theta_0} \in \mathrm{int}(\mathbb{D}).
\]
If $R_0=0$, then $z_0=0$ and, by Lemma \ref{Wsotto}, $W(R_0,\theta_0)=W(0,\theta_0)$ coincides with $\sigma_{\tilde{\varphi}}(z_0)$, the action of the fixed point at the origin. If $R_0>0$, then $(R_0,\theta_0)$ is a critical point of $W$, and hence a fixed point of $\Phi$. It follows that $z_0$ is a fixed point of $\varphi$, which by Proposition \ref{forac} has action
\[
\sigma_{\tilde{\varphi}}(z_0) = \Sigma(R_0,\theta_0) = W(R_0,\theta_0).
\]
In both cases, $z_0$ is an interior fixed point of $\varphi$ of action $W(R_0,\theta_0)$, and by (\ref{sotto}) we have 
\[
\sigma_{\varphi}(z_0) = W(R_0,\theta_0) < \frac{1}{2} \frac{\mathrm{CAL}(\tilde{\varphi},\omega)}{\int_{\D} \omega} \leq 0.
\]
This concludes the proof of Theorem \ref{fixpointmon}. 
\end{proof}

\begin{rem}
\label{optimality}
The coefficient $1/2$ is optimal in the upper bound for the action given by Theorem \ref{fixpointmon}, meaning that for every $\eta>1/2$ we can find $\tilde{\varphi}\in \widetilde{\mathrm{Diff}}(\D,\omega)$ such that $\varphi:= \pi(\tilde{\varphi})$ is radially monotone, $\mathrm{CAL}(\tilde{\varphi},\omega)<0$, and for all fixed points $z_0$ of $\varphi$ there holds
\begin{equation}
\label{dadimo}
\sigma_{\tilde{\varphi}}(z_0) \geq \eta \frac{\mathrm{CAL}(\tilde{\varphi},\omega)}{\int_{\D} \omega} \quad \mbox{or, equivalently, } \quad \frac{\sigma_{\tilde{\varphi}}(z_0)}{\mathrm{CAL}(\tilde{\varphi},\omega)} \int_{\D} \omega \leq \eta.
\end{equation}
In order to prove this fact, fix some $\epsilon\in (0,1)$ and notice that the non-positive function
\[
\hat{h}:[0,1] \rightarrow \R, \qquad \hat{h}(s) := \max \{-\epsilon, \pi(s-1)\}
\]
satisfies
\[
\int_0^1 r \hat{h}(r^2)\, dr < \int_0^{1-\epsilon/\pi} r \hat{h}(r^2)\, dr = - \frac{\epsilon}{2} \left( 1 - \frac{\epsilon}{\pi} \right)^2.
\]
By approximating $\hat{h}$ from above, we can find a smooth function $h:[0,1]\rightarrow \R$ such that 
\[
h(0)=-\epsilon, \quad h(1) = 0, \quad 0 < h' < \pi,
\]
and
\[
\int_0^1 r h(r^2)\, dr \leq - \frac{\epsilon}{2} \left( 1 - \frac{\epsilon}{\pi} \right)^2.
\]
Let $\tilde{\varphi}$ be the element of $\widetilde{\mathrm{Diff}}(\D,\omega_0)$ which is generated by the radial autonomous Hamiltonian $H(z):= h(|z|^2)$ and by the standard area form $\omega:=dx\wedge dy$ on $\D$. By Lemma \ref{radialham}, the time-one map $\varphi:= \pi(\tilde{\varphi})$ is
\[
\varphi(z) = e^{-2h'(|z|^2)i} z,
\]
and from the bounds on $h'$ we deduce that the only fixed point of $\varphi$ is the origin. By the same Lemma, this fixed point has action
\[
\sigma_{\tilde{\varphi}}(0) =h(0)= -\epsilon,
\]
while the Calabi invariant of $\tilde{\varphi}$ has the upper bound
\[
\mathrm{CAL}(\tilde{\varphi},\omega) = 4 \pi \int_0^1 r h(r^2)\, dr \leq - 2 \pi \epsilon \left( 1 - \frac{\epsilon}{\pi} \right)^2.
\]
Therefore, for the unique fixed point $z_0=0$ of $\varphi$ we have
\[
\frac{\sigma_{\tilde{\varphi}}(z_0)}{\mathrm{CAL}(\tilde{\varphi},\omega)}\int_{\D} \omega = - \frac{\epsilon}{\mathrm{CAL}(\tilde{\varphi},\omega)} \pi \leq \frac{\epsilon}{2\pi \epsilon (1-\epsilon/\pi)^2} \pi = \frac{1}{2} \left(1 - \frac{\epsilon}{\pi}\right)^{-2}.
\]
This shows that $\tilde\varphi$ satisfies (\ref{dadimo}) with $\eta:= 1/2 (1-\epsilon/\pi)^{-2}$, and the claim follows from the fact that by choosing $\epsilon$ small the number $\eta$ can be made arbitrarily close to $1/2$.
\end{rem}

We conclude this section by constructing an example which shows that the conclusion of Theorem \ref{fixpointmon} does not hold if we remove the radial monotonicity assumption. In this example, the map $\varphi= \pi(\tilde{\varphi})$ fixes the origin and is $C^0$-close but not $C^1$-close to the identity and not radially monotone.

We work with the standard area form $\omega := dx\wedge dy$ on $\D$. We fix a natural number $n\geq 2$ and consider the smooth radial autonomous Hamiltonian
\[
H_+: \D \rightarrow \R, \qquad H_+(z) = \frac{\pi}{n} ( 1 - |z|^2).
\]
By Lemma \ref{radialham}, the corresponding Hamiltonian flow is the path of rotations
\[
\varphi^+_t(z) = e^{2\pi i t/n} z.
\]
Let $\tilde{\varphi}^+$ be the element of $\widetilde{\mathrm{Diff}}(\D,dx\wedge dy)$ which is represented by the path $\{\varphi_t^+\}$, so that
\[
\varphi^+(z) := \pi(\tilde{\varphi}^+)(z) = e^{2\pi i /n} z
\]
is the counterclockwise rotation by the angle $2\pi/n$. Since $H_+$ vanishes on the boundary of $\D$, Lemma \ref{radialham} implies that the Calabi invariant of $\tilde{\varphi}^+$ is
\[
\mathrm{CAL}(\tilde{\varphi}^+,dx\wedge dy) = 4\pi \int_0^1 r \frac{\pi}{n} (1-r^2)\, dr = \frac{\pi^2}{n}.
\]
For each $k=1,\dots,n$, let $D_k$ be a closed disk which is contained in the sector
\begin{equation}
\label{sector0}
\left\{r e^{i\theta}  \mid 0< r< 1, \; \frac{2\pi (k-1)}{n} < \theta < \frac{2\pi k}{n} \right\}.
\end{equation}
Now let $H_-: \D \rightarrow \R$ be a smooth autonomous Hamiltonian whose support  is contained in the union of the disks $D_k$ and such that
\begin{equation}
\label{condH^-}
\int_{\D} H_-(z) \, dx\wedge dy < - \frac{\pi^2}{2n}.
\end{equation}
Let $\varphi_t^-$ be the Hamiltonian flow of $H_-$, $\tilde{\varphi}^-=[\{\varphi_t^-\}]$ the corresponding element of $\widetilde{\mathrm{Diff}}(\D,dx\wedge dy)$, and $\varphi^- = \pi(\tilde{\varphi}^-) = \varphi^-_1$ the corresponding area-preserving diffeomorphism. By Proposition \ref{calham} we have
\[
\mathrm{CAL}(\tilde{\varphi}^-,dx\wedge dy) = 2 \int_{\D} H_-(z) \, dx \wedge dy < -\frac{\pi^2}{n}.
\]
Let $\tilde{\varphi} := \tilde{\varphi}^+ \circ \tilde{\varphi}^-$. Since the Calabi invariant is a homomorphism, we find
\[
\mathrm{CAL}(\tilde{\varphi},dx\wedge dy) = \mathrm{CAL}(\tilde{\varphi}^+,dx\wedge dy) + \mathrm{CAL}(\tilde{\varphi}^-,dx\wedge dy) < 0.
\]
It is easy to see that $\varphi:= \pi(\tilde{\varphi}) = \varphi^+ \circ \varphi^-$ has a unique fixed point in the origin. Indeed, this follows from the fact that $\varphi^-$ leaves each sector (\ref{sector0}) invariant, while the $2\pi/n$-rotation $\varphi^+$ permutes these sectors cyclically. By Lemma \ref{formule} (ii) and Proposition \ref{actham}, the action of the origin with respect to $\tilde{\varphi}$ is
\[
\sigma_{\tilde{\varphi},\lambda}(0) = \sigma_{\tilde{\varphi}^+,\lambda}(\varphi^-(0)) + \sigma_{\tilde{\varphi}^-,\lambda}(0) = \sigma_{\tilde{\varphi}^+,\lambda}(0) + \sigma_{\tilde{\varphi}^-,\lambda}(0) = H^+(0) + H^-(0) = \frac{\pi}{n},
\]
where we have used the fact that both paths $\varphi_t^+$ and $\varphi_t^-$ fix the origin.
The maps $\varphi^+$ and $\varphi^-$ satisfy
\[
|\varphi^+(z) - z| \leq \frac{2\pi}{n}, \qquad |\varphi^-(z) - z| \leq \max_{k=1,\dots,n}\mathrm{diam}(D_k) \leq \frac{2\pi}{n},
\]
for every $z\in \D$. It follows that
\begin{equation}
\label{closetoid}
\|\varphi-\mathrm{id}\|_{\infty} \leq \frac{4\pi}{n},
\end{equation}
so $\varphi$ can be made arbitrarily $C^0$-close to the identity by choosing the natural number $n$ to be large.

We conclude that $\tilde{\varphi}$ is an element of $\widetilde{\mathrm{Diff}}(\D,dx\wedge dy)$ with negative Calabi invariant such that $\varphi=\pi(\tilde{\varphi})$ satisfies (\ref{closetoid}) and has a  unique fixed point with positive action. It satisfies all the assumptions of Theorem \ref{fixpointmon} except for the radial monotonicity, and all the assumptions of Corollary \ref{fixpoint} except for the $C^1$-closeness to the identity.

\begin{rem}
It is easy to modify this class of examples so that $\varphi$ is compactly supported in the interior of $\D$. Indeed, it is enough to slow down the rotation $\varphi^+$ outside from the support of $\varphi^-$ and make it agree with the identity on circles which are close to the boundary (see also Section \ref{controsec}). By embedding $\Z \times \D$ periodically into the strip $\R \times [0,\pi]$, one finds an area preserving diffeomorphism of the strip with zero flux, negative Calabi invariant, but no fixed points with negative action. Up to the modification of the area form by a suitable conjugacy, this shows that the conclusion of Theorem 1.12 in \cite{abhs17} does not hold if one removes the monotonicity assumption.
\end{rem} 

\subsection{Moving a fixed point into the origin}
\label{movingsec}

In order to deduce Corollary \ref{fixpoint} from Theorem \ref{fixpointmon}, we have to show that any diffeomorphism of the disk which is $C^1$-close to the identity and has an interior fixed point can be conjugated to a radially monotone diffeomorphism. This section is devoted to the proof of this result.

Let $\mathrm{ang}(L,L')\in [0,\pi/2]$ denote the angle between two lines $L,L'$ in the complex plane $\C$. We shall make use of the following elementary result, of which we give an analytic proof.

\begin{lem}
\label{cerchi}
For every $\epsilon>0$ there exists $\delta>0$ with the following property: if $C$ and $C'$ are circles in the plane with radius $r$ and $r'$, respectively, which are either disjoint or somewhere tangent, then for every $z\in C$ and $z'\in C'$ with
\[
|z-z'| < \delta \min\{ r,r'\}
\]
there holds
\[
\mathrm{ang}(T_z C, T_{z'} C') < \epsilon.
\]
\end{lem}

\begin{proof}
Assume without loss of generality that $r'\leq r$.  Since the problem is invariant with respect to translations and homotheties, we can assume that $C'$ is the unit circle centered at $0$, i.e. $C'=\partial \mathbb{D}$, and that $z'=1$. Since $r\geq r'=1$, the circle $C$ is contained in $\C \setminus \mathrm{int}(\mathbb{D})$.  Given $\delta>0$, let $z\in C$ be such that
\[
|z-1| = |z-z'| < \delta \min\{ r,r'\} = \delta.
\]
Let $w\in \C$ be the center of $C$. Then $C$ has the arc-length parametrisation
\[
\gamma: \R \rightarrow \C, \qquad \gamma(s) := w + (z-w) e^{is/r},
\]
which satisfies
\[
\gamma(0) = z, \qquad \gamma'(0) = i \frac{z-w}{r}, \qquad |\gamma''(s)| = \frac{1}{r} \qquad \forall s\in \R.
\]
The fact that $C$ is contained in $\C \setminus  \mathrm{int}(\mathbb{D})$ implies that the function
\[
f:\R \rightarrow \R, \qquad f(s) := \frac{1}{2} \bigl( |\gamma(s)|^2 - 1 \bigr),
\]
is everywhere non-negative. Notice that
\begin{equation}
\label{cer1}
f(0) = \frac{1}{2} ( |z|^2 - 1) = \frac{1}{2} (|z|+1)(|z|-1) \leq \frac{1}{2} (|z-1|+2)|z-1| < \frac{1}{2} (2+\delta) \delta.
\end{equation}
The first two derivatives of this functions are
\[
f'(s) = \re \bigl( \gamma'(s) \overline{\gamma(s)} \bigr), \qquad f''(s) = \re \bigl( \gamma''(s) \overline{\gamma(s)} + |\gamma'(s)|^2 \bigr) = \re \bigl( \gamma''(s) \overline{\gamma(s)} \bigr) + 1.
\]
Since the tangent line $T_z C$ is $i(z-w)\R$ we have
\[
\begin{split}
|f'(0)| &= \left| \re\Bigl( i \frac{z-w}{r} \overline{z} \Bigr)  \right| = |z| \, \cos \mathrm{ang} ( i(z-w)\R, z\R) \\ &= |z| \, \cos \mathrm{ang} ( T_z C, z\R) = |z| \, \sin \mathrm{ang} ( T_z C, iz\R),
\end{split}
\]
from which we get
\begin{equation}
\label{cer2}
\sin \mathrm{ang} ( T_z C, iz\R) = \frac{|f'(0)|}{|z|} \leq |f'(0)|.
\end{equation}
From the expression of $f''$ we find
\[
\begin{split}
f''(s) &\leq 1 + |\gamma''(s)| |\gamma(s)| = 1 + \frac{|\gamma(s)|}{r} \leq 1 + \frac{|\gamma(s)-w| + |w-z| + |z-1| + 1}{r} \\ &= 3 + \frac{|z-1|+1}{r} < 3 + \frac{\delta+1}{r} < 4 + \delta. 
\end{split}
\]
Let $s\in \R$. From Taylor's formula we find
\[
f(s) = f(0) + f'(0)s + \frac{1}{2} f''(\theta s) s^2 \leq f(0) + f'(0) s + \left( 2 + \frac{\delta}{2} \right) s^2,
\]
where $\theta\in (0,1)$. Since the function $f$ is everywhere non-negative, so is the polynomial on the righthand side of the above expression. Therefore, the discriminant of this polynomial is non-positive and hence
\begin{equation}
\label{cer3}
f'(0)^2 \leq 4  \left( 2 + \frac{\delta}{2} \right) f(0) < 4  \left( 2 + \frac{\delta}{2} \right) \frac{1}{2} (2+\delta) \delta = (4+\delta)(2+\delta) \delta < (4+\delta)^2 \delta,
\end{equation}
where we have used also (\ref{cer1}). Since
\[
|\arcsin t| \leq \frac{\pi}{2} |t| \qquad \forall t\in [-1,1],
\]
(\ref{cer2}) and (\ref{cer3}) imply that
\[
\mathrm{ang} ( T_z C, iz\R) < \frac{\pi}{2} (4+\delta) \sqrt{\delta}.
\]
Moreover, using the expression for the length of a chord in a unit circle in terms of the underlying angle, we have
\[
\begin{split}
\mathrm{ang}( iz\R, T_{z'} C') &= \mathrm{ang}( iz\R, T_{1} \partial \mathbb{D}) = \mathrm{ang}( iz\R, i \R) = \mathrm{ang}( z\R, \R) \\ &= 2 \arcsin \left( \frac{1}{2} \Bigl| \frac{z}{|z|} - 1 \Bigr| \right) \leq 2 \arcsin \frac{|z-1|}{2} \leq \frac{\pi}{2} |z-1| < \frac{\pi}{2} \delta.
\end{split}
\]
It follows that
\[
\mathrm{ang}(T_z C, T_{z'} C') \leq \mathrm{ang} ( T_z C, iz\R) + \mathrm{ang}( iz\R, T_{z'} C') < \frac{\pi}{2} \bigl( (4+\delta) \sqrt{\delta} + \delta \bigr).
\]
We have shown that if $z\in C$ and $z'\in C'$ are such that
\[
|z-z'| < \delta \min \{r,r'\}
\]
then 
\[
\mathrm{ang}(T_z C, T_{z'} C') < \frac{\pi}{2} \bigl( (4+\delta) \sqrt{\delta} + \delta \bigr),
\]
and the thesis follows from the fact that the latter function is infinitesimal for $\delta\rightarrow 0$.
\end{proof}

We recall that if $z_0$ is an interior point of the disk $\mathbb{D}$ then the M\"obius transformation
\[
h(z) = \frac{z+z_0}{\overline{z_0} z + 1}
\]
is a smooth diffeomorphism of the disk $\mathbb{D}$ onto itself mapping $0$ into $z_0$. Here is the precise statement of the result which is mentioned at the beginning of this section.

\begin{prop}
\label{moebius}
There exists a $C^1$-neighborhood $\mathscr{U}$ of the identity in the space of smooth diffeomorphisms of $\mathbb{D}$ with the following property: if $\varphi\in \mathscr{U}$ has an interior fixed point $z_0$ and $h\in \mathrm{Diff}(\mathbb{D})$ is a M\"obius transformation mapping $0$ into $z_0$, then the diffeomorphism $h^{-1} \circ \varphi \circ h$ is radially monotone.
\end{prop}

\begin{proof}
The M\"obius transformation $h$ maps the foliation by circles centred at $0$ into a foliation $\{C_z\}_{z\in \mathbb{D}\setminus \{z_0\}}$ of $\mathbb{D} \setminus \{z_0\}$ which consists of circles which enclose $z_0$. Being a conformal mapping, $h$ maps the foliation by rays
emanating from $0$ into a foliation $\{R_z\}_{z\in \mathbb{D}\setminus \{z_0\}}$ of $\mathbb{D} \setminus \{z_0\}$ consisting of circle arcs or segments which are orthogonal to the foliation $\{C_z\}_{z\in \mathbb{D}\setminus \{z_0\}}$. By the conformality of $h$, it is enough to show that if the diffeomorphism $\varphi$ belongs to a suitable $C^1$-neighborhood of the identity then
\[
\mathrm{ang}( D \varphi(z) T_z R_z , T_{\varphi(z)} C_{\varphi(z)}) >0 \qquad \forall z\in \mathbb{D} \setminus \{z_0\}.
\]
Equivalently, we must show that
\begin{equation}
\label{toprove}
\mathrm{ang}( i D \varphi(z) T_z R_z , T_{\varphi(z)} C_{\varphi(z)}) < \frac{\pi}{2} \qquad \forall z\in \mathbb{D} \setminus \{z_0\}.
\end{equation}
From the identity
\[
\begin{split}
\varphi(z) &= \varphi(z_0) + \int_0^1 D\varphi(z_0 + t(z-z_0))[z-z_0]\, dt \\ &= z_0 + \int_0^1 \bigl( D\varphi(z_0 + t(z-z_0)) - I \bigr) [z-z_0]\, dt + z - z_0 \\ &= z + \int_0^1 \bigl( D\varphi(z_0 + t(z-z_0)) - I \bigr) [z-z_0]\, dt,
\end{split}
\]
we deduce the estimate
\begin{equation}
\label{vic1}
|\varphi(z) - z| \leq \|D\varphi - I\|_{\infty} |z-z_0| \qquad \forall z\in \mathbb{D}.
\end{equation}
Similarly,
\[
|\varphi^{-1}(w) - w| \leq \|D\varphi^{-1} - I\|_{\infty} |w-z_0| \qquad \forall w\in \mathbb{D},
\]
and applying this to $w=\varphi(z)$ we find
\begin{equation}
\label{vic2}
|z-\varphi(z)| \leq \|D\varphi^{-1} - I\|_{\infty} |\varphi(z)-z_0| \qquad \forall z\in \mathbb{D}.
\end{equation}
Putting together (\ref{vic1}) and (\ref{vic2}) we obtain
\begin{equation}
\label{vicini}
|\varphi(z) - z| \leq \max\{ \|D\varphi - I\|_{\infty}, \|D\varphi^{-1} - I\|_{\infty} \} \min \{ |z-z_0|, |\varphi(z)-z_0| \}  \qquad \forall z\in \mathbb{D}.
\end{equation}
Let $\delta$ be the positive number which is given by Lemma \ref{cerchi} with $\epsilon = \pi/4$: if $C$ and $C'$ are circles in the plane which are either disjoint or somewhere tangent and have radius $r$ and $r'$, respectively, then there holds
\begin{equation}
\label{vicini2}
z\in C, \; z'\in C' , \;|z-z'| < \delta \min\{ r,r'\} \qquad \Rightarrow \qquad
\mathrm{ang}(T_z C, T_z C') < \frac{\pi}{4}.
\end{equation}
Let $\delta'>0$ be such that
\begin{equation}
\label{nonruota}
T\in GL(2,\R) , \; \|T-I\| < \delta' \qquad \Rightarrow \qquad \mathrm{ang}(TL,L) < \frac{\pi}{4}
\end{equation}
for every 1-dimensional subspace $L \subset \R^2$. We consider the following $C^1$-neighborhood of the identity in $\mathrm{Diff}(\mathbb{D})$:
\[
\mathscr{U} := \Bigl\{ \varphi\in \mathrm{Diff}(\mathbb{D}) \mid \|D\varphi - I\|_{\infty} < \frac{\delta}{2}, \;  \|D\varphi^{-1} - I\|_{\infty} < \frac{\delta}{2}, \; \|D\varphi - I\|_{\infty} < \delta' \Bigr\}.
\]
Now let $\varphi\in \mathscr{U}$ be a diffeomorphism with interior fixed point $z_0$. Let $z\in \mathbb{D} \setminus \{z_0\}$. Then we have
\begin{equation}
\label{ang1}
\begin{split}
\mathrm{ang}( i D \varphi(z) T_z R_z, T_{\varphi(z)} C_{\varphi(z)} ) &\leq \mathrm{ang}( i D \varphi(z) T_z R_z, i T_z R_z) + \mathrm{ang} (  i T_z R_z, T_{\varphi(z)} C_{\varphi(z)} ) \\ &= \mathrm{ang}( D \varphi(z) T_z R_z, T_z R_z) + \mathrm{ang} (  T_z C_z, T_{\varphi(z)} C_{\varphi(z)} ).
\end{split}
\end{equation}
Since $\|D\varphi(z) - I\|< \delta'$, (\ref{nonruota}) implies that
\begin{equation}
\label{ang2}
\mathrm{ang}( D \varphi(z) T_z R_z, T_z R_z) < \frac{\pi}{4}.
\end{equation}
The circles $C_z$ and $C_{\varphi(z)}$ are either disjoint or they coincide, and in the latter case they are in particular somewhere tangent. Since both $C_z$ and $C_{\varphi(z)}$ enclose $z_0$, their radii $r$ and $r'$ are such that
\[
r > \frac{1}{2} |z-z_0|, \qquad r' > \frac{1}{2} |\varphi(z)-z_0|.
\]
By (\ref{vicini}) we have
\[
|\varphi(z) - z| < \frac{\delta}{2} \min\{ |z-z_0|,|\varphi(z)-z_0| \} < \delta \min \{r,r'\}.
\]
Then (\ref{vicini2}) implies that
\begin{equation}
\label{ang3}
\mathrm{ang}(T_z C_z, T_{\varphi(z)} C_{\varphi(z)}) < \frac{\pi}{4}.
\end{equation}
By (\ref{ang1}), (\ref{ang2}) and (\ref{ang3}) we conclude that
\[
\mathrm{ang}( i D \varphi(z) T_z R_z, T_{\varphi(z)} C_{\varphi(z)} ) < \frac{\pi}{2},
\]
which proves (\ref{toprove}).
\end{proof}

Now we can easily deduce Corollary \ref{fixpoint} from Theorem \ref{fixpointmon}:

\begin{proof}[Proof of Corollary \ref{fixpoint}.]
Let $\tilde{\varphi}\in \widetilde{\mathrm{Diff}}(\mathbb{D},\omega)$ and $\varphi=\pi(\tilde{\varphi})$ be as in the assumptions of Corollary \ref{fixpoint}, with the $C^1$-neighborhood $\mathscr{U}$ being given by Proposition \ref{moebius}. Since $\varphi$ preserves an area form on $\mathrm{int}(\mathbb{D})$, by Brouwer's translation theorem it has an interior fixed point $z_0$. Let $h\in \mathrm{Diff}^+(\mathbb{D})$ be a M\"obius transformation mapping $0$ into $z_0$. By Proposition \ref{moebius} the diffeomorphism $h^{-1} \circ \varphi \circ h$ is radially monotone.  

The element $h^{-1} \circ \tilde{\varphi} \circ h$ belongs to $\widetilde{\mathrm{Diff}} (\mathbb{D},h^* \omega)$ and satisfies assumptions (i) and (ii) of Theorem \ref{fixpointmon}. By Proposition \ref{naturality}, its Calabi invariant is
\[
\mathrm{CAL}(h^{-1} \circ \tilde{\varphi} \circ h, h^* \omega) = \mathrm{CAL}(\tilde{\varphi},\omega) \leq 0,
\]
so $h^{-1} \circ \tilde{\varphi} \circ h$ satisfies also assumption (iii).  Therefore, Theorem \ref{fixpointmon} implies that $h^{-1} \circ \tilde{\varphi} \circ h$ has an interior fixed point $z_0$ with negative action, and more precisely
\[
\sigma_{h^{-1} \circ \tilde{\varphi} \circ h} (z_0) < \frac{1}{2} \frac{\mathrm{CAL}(h^{-1} \circ \tilde{\varphi} \circ h, h^* \omega)}{\int_{\D} h^* \omega} = \frac{1}{2} \frac{\mathrm{CAL}(\tilde{\varphi} , \omega)}{\int_{\D} \omega}.
\]
Then $h(z_0)$ is an interior fixed point of $\varphi$ with the same action:
\[
\sigma_{\tilde{\varphi}}(h(z_0))=\sigma_{\tilde{\varphi},\lambda}(h(z_0)) = \sigma_{h^{-1} \circ \tilde{\varphi} \circ h,h^* \lambda} (z_0)= \sigma_{h^{-1} \circ \tilde{\varphi} \circ h} (z_0),
\]
see again Proposition \ref{naturality}. The proof is complete.
\end{proof}

\begin{rem}
If one replaces condition (iii) in either Theorem \ref{fixpointmon} or Corollary \ref{fixpoint} by the assumption that the Calabi invariant is non-negative, then one gets the existence of an interior fixed point with positive action, and more precisely with action strictly larger than half of the Calabi invariant averaged by the total area of the disk.  This follows from  Theorem \ref{fixpointmon} or Corollary \ref{fixpoint} applied to $\tilde{\varphi}^{-1}$.
\end{rem}

\subsection{A refined example}
\label{controsec}

The aim of this section is to modify the example of Section \ref{proofsec}, in order to make it compactly supported and to have good bounds on the action and on the Calabi invariant. The first requirement is easy to achieve, the latter requires a more careful analysis. We will actually construct two examples, which are the building stones in the proof of Theorem \ref{main2}.

Throughout the whole section, we work with the standard area form $\omega=dx\wedge dy$ and with its standard primitive
\[
\lambda_0 := \frac{1}{2} ( x\, dy - y \, dx),
\]
whose expression in polar coordinates $(r,\theta)$ is
\[
\lambda_0 = \frac{r^2}{2} \, d\theta.
\]
The area-preserving diffeomorphisms  that we will construct in this section are compactly supported in the interior of $\D$ and hence, as explained in Remark \ref{compsuppdiff}, the action and the Calabi invariant of these maps are well defined.

Given $\delta\in (0,1/2)$, let $\chi_{\delta}: [0,+\infty) \rightarrow \R$ be a smooth convex function which is supported in $[0,1)$ and satisfies
\begin{equation}
\label{chi}
\chi_{\delta}(s) = 1 - \delta - s \qquad \forall s\in [0,1-2\delta].
\end{equation}
It follows that $\chi_{\delta}$ is monotonically decreasing, non-negative, and satisfies
\begin{eqnarray}
\label{chi1}
& \max\{1-\delta-s,0\} \leq \chi_{\delta}(s) \leq  \max\{ (1 - \delta)(1 - s),0\}, &\\
\label{chi2}
&-1 \leq \chi_{\delta}'(s) \leq 0, &\\
\label{chi3}
& 0\leq \chi_{\delta}(s) - s \chi_{\delta}'(s) \leq 1-\delta,&
\end{eqnarray}
for every $s\in [0,+\infty)$,
where the last inequalities follow from the fact that the function $\chi_{\delta}(s) - s \chi_{\delta}'(s)$ is monotonically decreasing, having derivative $-s\chi''_{\delta}(s)\leq 0$, and takes the value $1-\delta$ for $s=0$ and $0$ for $s\geq 1$. 

Let $n$ be a large positive integer and let $\delta\in (0,1/2)$ be a small real number. The sizes of $n$ and $\delta$ will be determined in due time.
The autonomous Hamiltonian 
\[
H_+: \D \rightarrow \R, \qquad H_+(z) = \frac{\pi}{n} \chi_{\delta}(|z|^2),
\]
is supported in $\mathrm{int}(\D)$. We denote by $\varphi^+_t$ the flow of $X_{H_+}$ and by $\varphi^+:= \varphi^+_1$ the time-1 map. Each map $\varphi^+_t$ belongs to $\mathrm{Diff}_c(\D,dx\wedge dy)$. The diffeomorphism $\varphi^+_t$ restricts to the counterclockwise rotation of angle $2\pi t/n$ on the disk of radius $\sqrt{1-2\delta}$ about the origin. Outside from this disk, the map $\varphi^+_t$ rotates each circle about the origin counterclockwise by an angle which does not exceed $2\pi t/n $ (because of (\ref{chi2})), and which becomes zero near the boundary. In particular, we have
\begin{equation}
\label{vicid1}
\|\varphi^+_t - \mathrm{id}\|_{\infty} \leq \frac{2\pi}{n} \qquad \forall t\in [0,1].
\end{equation} 
By Lemma \ref{radialham}, the action of $\varphi^+$ with respect to $\lambda_0$ is the radial function
\[
\sigma_{\varphi^+,\lambda_0}(z) = \frac{\pi}{n} \bigl( \chi_{\delta}(|z|^2) -  |z|^2\chi_{\delta}'(|z|^2) \bigr).
\]
By (\ref{chi}) and (\ref{chi3}), this function satisfies
\begin{equation}
\label{azphi+}
\begin{split}
\sigma_{\varphi^+,\lambda_0}(z) &= \frac{\pi}{n} (1-\delta) \qquad \mbox{for } |z| \leq \sqrt{1-2\delta}, \\
0 \leq \sigma_{\varphi^+,\lambda_0}(z) &\leq \frac{\pi}{n} (1-\delta)  \leq \frac{\pi}{n}  \qquad   \mbox{for }  z\in \D,
\end{split}
\end{equation}
for every $t\in [0,1]$. Using also Lemma \ref{radialham} and (\ref{chi1}), we find that the Calabi invariant of $\varphi^+$ has the upper bound
\begin{equation}
\label{calH+}
\begin{split}
\mathrm{CAL}(\varphi^+,dx\wedge dy ) &= 4\pi \int_0^1 r\, \frac{\pi}{n} \chi_{\delta}(r^2)\, dr \leq \frac{4\pi^2}{n} \int_0^1 r (1-\delta)(1 - r^2) \, dr \\ &= \frac{\pi^2}{n} (1- \delta) \leq \frac{\pi^2}{n}.
\end{split}
\end{equation}

Given any $\rho\in (0,1)$, we can find a finite number of pairwise non-overlapping closed disks in the sector
\[
\left\{ r e^{i\theta} \mid 0 < r < 1, \; 0 < \theta < \frac{2\pi}{n} \right\}
\]
such that the ratio between their total area and the area of the sector is at least $\rho$. The existence of such a disk packing is a direct consequence of Vitali's covering theorem, see e.g.\ \cite[Section 3.2, Lemma 3.9]{ss05}. By acting on these disks by the rotation of angle $2\pi/n$, we obtain a finite family 
\[
\{D_j \mid j\in J\}
\]
of pairwise non overlapping disks in $\D$, each of which is contained in some sector
\begin{equation}
\label{sector}
\left\{ r e^{i\theta} \mid 0 < r < 1, \; \frac{2\pi (k-1)}{n} < \theta < \frac{2\pi k}{n} \right\},
\end{equation}
for some integer $k$, and such that
\begin{equation}
\label{denso}
\sum_{j\in J} \mathrm{area}(D_j) \geq \rho \, \mathrm{area} (\D) = \rho \pi.
\end{equation}
The constant $\rho$ can be chosen to be arbitrarily close to 1 and will be determined later. We denote by $z_j$ and $r_j$ the center and the radius of the disk $D_j$. By construction,
\[
r_j \leq \frac{\pi}{n} \qquad \forall j\in J.
\]
Fix some small number $\epsilon\in (0,1/2)$, whose size will be determined along the way. For every $j\in J$ we consider the autonomous Hamiltonian
\[
K_j: \D \rightarrow \R, \qquad K_j(z) = - c \, \chi_{\epsilon} \left( \frac{|z-z_j|^2}{r_j^2} \right),
\]
where the positive number $c$ is to be fixed. This function is supported in the interior part of $D_j$. The flow of the Hamiltonian vector field $X_{K_j}$ is denoted by $\psi_{j,t}$ and the time 1-map by $\psi_j := \psi_{j,1}$. The maps $\psi_{j,t}$ are supported in the interior of  $D_j$. Inside the disk of radius $\sqrt{1-\epsilon} r_j$ about $z_j$, the map $\psi_j$ is the clockwise rotation of angle $2c/r_j^2$.
In the next lemma, we estimate the action and the Calabi invariant of $\psi_j$. 

\begin{lem}
\label{psij}
The following properties hold:
\begin{enumerate}[(i)]
\item $\sigma_{\psi_j,\lambda_0}$ is supported in the interior of $D_j$ for every $t\in [0,1]$;
\item $- c  - \pi/n  \leq \sigma_{\psi_j,\lambda_0} \leq \pi/n$ on $\D$ for every $t\in [0,1]$;
\item $\mathrm{CAL}(\psi_j,dx\wedge dy) \leq - c r_j^2 (1-\epsilon)^2 \pi$.
\end{enumerate}
\end{lem}

\begin{proof}
Since $\psi_j$ is supported in the interior of the disk $D_j$, so is its action, and we just have to prove the bounds (ii) and (iii).
Let $\tilde{K}_j$ be the autonomous Hamiltonian
\[
\tilde{K}_j: \D \rightarrow \R, \qquad \tilde{K}_j(z) = - c\,  \chi_{\epsilon} \left( \frac{|z|^2}{r_j^2} \right),
\]
and let $\tilde{\psi}_j$ be the time-one map of the vector field $X_{\tilde{K}_j}$.  By Lemma \ref{radialham}, the action of $\tilde{\psi}_j$ with respect to $\lambda_0$ is the function
\[
\sigma_{\tilde{\psi}_j,\lambda_0}(z) = - c \left( \chi_{\epsilon} \Bigl( \frac{|z|^2}{r_j^2} \Bigr) - \frac{|z|^2}{r_j^2} \chi_{\epsilon}' \Bigl( \frac{|z|^2}{r_j^2} \Bigr) \right).
\]
By (\ref{chi}), this function satisfies
\begin{equation}
\label{dentro}
\sigma_{\tilde{\psi}_{j},\lambda_0} = - c (1-\epsilon) \qquad \mbox{on } r_j \sqrt{1-\epsilon} \, \D,
\end{equation}
is supported in the interior of $r_j \D$ and, by (\ref{chi3}), has the bounds
\begin{equation}
\label{tutto}
-c \leq - c (1-\epsilon) \leq \sigma_{\tilde{\psi}_j,\lambda_0}(z) \leq 0 \qquad \mbox{on } \D.
\end{equation}
The diffeomorphisms $\psi_j$ and $\tilde{\psi}_j$ are conjugated by the translation $z\mapsto z+z_j$. In order to work on the disk $\D$, it is convenient to fix an area-preserving diffeomorphism $\tau_j: \D \rightarrow \D$ which on the disk $r_j \D$ coincides with the translation $z\mapsto z+z_j$. Then $\tilde{\psi}_j = \tau_j^{-1} \circ \psi_j \circ \tau_j$ and by Proposition \ref{naturality} we have
\[
\sigma_{\psi_j, \lambda_0} \circ \tau_j = \sigma_{ \tau_j^{-1} \circ \psi_j \circ \tau_j, \tau_j^* \lambda_0} = \sigma_{\tilde{\psi}_j,  \tau_j^* \lambda_0},
\]
and
\[
\mathrm{CAL}(\psi_j,dx\wedge dy) = \mathrm{CAL}(\tilde{\psi}_j,dx\wedge dy).
\]
By (\ref{dentro}) and (\ref{tutto}) we find
\[
\mathrm{CAL}(\tilde{\psi}_j,dx\wedge dy) = \int_{\D} \sigma_{\tilde{\psi}_j,\lambda_0} \, dx\wedge dy \leq - \int_{r_j \sqrt{1-\epsilon} \D} c (1-\epsilon) \, dx\wedge dy = - c r_j^2 (1-\epsilon)^2 \pi,
\]
and the bound (iii) on the Calabi invariant of $\psi_j$ follows.

Let $(x_j,y_j)$ be the coordinates of the center $z_j$ of the disk $D_j$. Since on $r_j \D$ we have
\[
\tau_j^* \lambda_0 - \lambda_0 = \frac{1}{2} \bigl( (x+x_j)\, dy - (y+y_j)\, dx - x\, dy + y\, dx \bigr) = \frac{1}{2} (x_j\, dy - y_j \, dx),
\]
the 1-form $\tau_j^* \lambda_0$ differs from $\lambda_0$ by the differential of a smooth function $u: \D \rightarrow \R$ such that
\[
u(x,y) = \frac{1}{2} ( x_j y - y_j x) \qquad \forall (x,y)\in r_j \D.
\]
By Lemma \ref{formule} (i), we find
\[
\sigma_{\tilde{\psi}_j,\tau^*_j \lambda_0} = \sigma_{\tilde{\psi}_j, \lambda_0} + u \circ \tilde{\psi}_j - u.
\]
Since the diffeomorphism $\tilde{\psi}_j$ is supported in the interior of $r_j \D$, so is its action $\sigma_{\tilde{\psi}_j,\tau^*_j \lambda_0}$, which by the above identity satisfies
\[
| \sigma_{\tilde{\psi}_j,\tau^*_j \lambda_0} - \sigma_{\tilde{\psi}_j, \lambda_0}| \leq \sup_{r_j \D} u - \inf_{r_j \D} u = r_j |z_j|\leq r_j \leq \frac{\pi}{n}.
\]
The estimate (ii) follows from the above inequality, together with (\ref{tutto}).
\end{proof}

Fix some $t\in [0,1]$. Since the diffeomorphisms $\psi_{j,t}$ have pairwise disjoint support, they pairwise commute. Their composition is denoted by $\varphi^-_t$: $\{\varphi^-_t\}_{t\in [0,1]}$ is a smooth path of area-preserving compactly supported diffeomorphisms starting at the identity and ending at $\varphi^-:= \varphi_1^-$. This path of diffeomorphisms is the flow of the autonomous Hamiltonian
\[
H_- := \sum_{j\in J} K_j.
\]
The diffeomorphism $\varphi^-_t$ is supported in the union of the $D_j$'s and maps each $D_j$ into itself. Since the diameter of these disks is at most $2\pi/n$, we have
\begin{equation}
\label{vicid2}
\|\varphi^-_t - \mathrm{id}\|_{\infty} \leq \frac{2\pi}{n} \qquad \forall t\in [0,1].
\end{equation}
By Lemma \ref{formule} (ii), we get
\[
\sigma_{\varphi^-,\lambda_0} = \sum_{j\in J} \sigma_{\psi_j,\lambda_0},
\]
where the summands have pairwise disjoint support by Lemma \ref{psij} (i). By Lemma \ref{psij} (ii) we get
\begin{equation}
\label{tutto2}
-c  - \frac{\pi}{n} \leq \sigma_{\varphi^-,\lambda_0} \leq \frac{\pi}{n} \qquad \mbox{on } \D.
\end{equation}
By Lemma \ref{psij} (iii) and (\ref{denso}) we obtain
\begin{equation}
\label{calest2}
\begin{split}
\mathrm{CAL}(\varphi^-,dx\wedge dy) &= \sum_{j\in J} \mathrm{CAL}(\psi_j,dx\wedge dy) \leq - c (1-\epsilon)^2 \sum_{j\in J} \pi r_j^2 \\ &= - c (1-\epsilon)^2 \sum_{j\in J} \mathrm{area} (D_j) \leq - c (1-\epsilon)^2 \rho \pi.
\end{split}
\end{equation}

From now on, we assume that $\delta$ is so small that the union of the disks $D_j$ is contained in the disk of radius $\sqrt{1-2\delta}$ about the origin, where $\varphi^+$ is a $2\pi/n$-rotation. We consider the smooth path of diffeomorphisms
\[
\varphi_t:= \varphi^+_t \circ \varphi^-_t \in \mathrm{Diff}_c(\D,dx\wedge dy)
\]
which connects the identity to $\varphi:= \varphi_1 = \varphi^+ \circ \varphi^-$.
From (\ref{vicid1}) and (\ref{vicid2}) we obtain
\begin{equation}
\label{vicid3}
\|\varphi_t - \mathrm{id}\|_{\infty} \leq \frac{4\pi}{n} \qquad \forall t\in [0,1].
\end{equation}
Since $\varphi^+$ is a rotation by the angle $2\pi/n$ on the disk of radius $\sqrt{1-2\delta}$ about the origin and $\varphi^-$ leaves each of the sectors (\ref{sector}) invariant, the only fixed points of $\varphi$ are the origin and points forming a neighborhood of $\partial \D$ outside from this disk. By using the formula
\begin{equation}
\label{laformula}
\sigma_{\varphi,\lambda_0} = \sigma_{\varphi^+,\lambda_0} \circ \varphi^- + \sigma_{\varphi^-,\lambda_0}
\end{equation}
from Lemma \ref{formule} (ii), together with the fact that the fixed points of $\varphi$ are outside of the support of $\varphi^-$ and $\sigma_{\varphi^-,\lambda_0}$ vanishes on them, (\ref{laformula}) implies that the origin has action
\[
\sigma_{\varphi,\lambda_0}(0) = \sigma_{\varphi^+,\lambda_0}(0) = \frac{\pi}{n} (1-\delta),
\]
and all other fixed points $z$ of $\varphi$ have action
\[
\sigma_{\varphi,\lambda_0}(z) =  \sigma_{\varphi^+,\lambda_0}(z) \geq 0.
\]
Therefore, all the fixed points of $\varphi$ have non-negative action.
Furthermore, all the other periodic points of $\varphi$ have period at least $n$. Indeed, periodic points which belong to the support of $\varphi^-$ have period which is a multiple of $n$, again by the invariance under $\varphi^-$ of the sectors (\ref{sector}) and by the fact that the rotation of angle $2\pi/n$ permutes these sectors cyclically. Periodic points outside of the disk of radius $\sqrt{1-2\delta}$ about the origin which are not fixed points have period at least $n$, because on the corresponding circle $\varphi=\varphi^+$ is a rotation of an angle not exceeding $2\pi/n$.

By using again (\ref{laformula}), the inequalities (\ref{azphi+}) and (\ref{tutto2}) imply the bound
\begin{equation}
\label{tutto3}
-c  - \frac{\pi}{n} \leq \sigma_{\varphi,\lambda_0} \leq \frac{2\pi}{n} \qquad \mbox{on } \D.
\end{equation}
On the other hand, (\ref{calH+}) and (\ref{calest2}) give us the following upper bound on the Calabi invariant of $\varphi$
\begin{equation}
\label{calest3}
\mathrm{CAL}(\varphi,dx\wedge dy) = \mathrm{CAL}(\varphi^+,dx\wedge dy) + \mathrm{CAL}(\varphi^-,dx\wedge dy) \leq \frac{\pi^2}{n} - c(1-\epsilon)^2 \rho \pi.
\end{equation}

We can finally fix the free parameters $n$, $\epsilon$, and $c$. In our first example, we fix an arbitrary positive number $L$ and we take
\[
c:= L - \frac{L+\pi}{n}.
\]
With this choice, the first estimate in (\ref{tutto3}) gives us
\[
\sigma_{\varphi,\lambda_0} \geq  - L + \frac{L}{n} \qquad \mbox{on } \D.
\]
On the other hand, (\ref{calest3}) implies 
\[
\begin{split}
\mathrm{CAL}(\varphi,dx\wedge dy) &\leq \frac{\pi^2}{n} - \left( L - \frac{L+\pi}{n}  \right) (1-\epsilon)^2 \rho \pi \\ & = - L\pi  (1-\epsilon)^2 \rho  + \frac{\pi}{n} \bigl( \pi + (L+\pi)(1-\epsilon)^2 \rho \bigr).
\end{split}
\]
By choosing $n$ large, $\epsilon$ small, and $\rho$ sufficiently close to 1, the latter quantity can be made as close as we wish to $-L\pi$. Therefore, taking into account also (\ref{vicid3}) and our previous discussion about the fixed and the periodic points of $\varphi$, we have proved the following fact.

\begin{prop}
\label{sysgra}
For every $\epsilon>0$ and $L>0$ there exists a natural number $n$, which can be chosen to be arbitrarily large, and a smooth isotopy $\varphi_t\in \mathrm{Diff}_c(\D,dx\wedge dy)$, $t\in [0,1]$, with $\varphi_0=\mathrm{id}$ such that, setting $\varphi:= \varphi_1$, the following properties hold:
\begin{enumerate}[(i)]
\item $\|\varphi_t - \mathrm{id}\|_{\infty} < \epsilon$ for every $t\in [0,1]$;
\item $\sigma_{\varphi,\lambda_0} \geq - L + L/n$ on $\D$;
\item $\mathrm{CAL}(\varphi,dx\wedge dy)\leq - L\pi + \epsilon$;
\item all the fixed points of $\varphi$ have non-negative action;
\item all the periodic points of $\varphi$ which are not fixed points have period at least $n$.
\end{enumerate}
\end{prop}

In our second example, we choose
\[
c:= \frac{2\pi}{n}.
\]
Now (\ref{tutto3}) becomes
\[
- \frac{3\pi}{n}  \leq \sigma_{\varphi,\lambda_0} \leq  \frac{2\pi}{n} \qquad \mbox{on } \D,
\]
which shows that the action $\sigma_{\varphi,\lambda_0}$ can be chosen to be arbitrarily close to the zero function in the uniform norm, by choosing a large integer $n$. The bound (\ref{calest3}) is now
\[
\mathrm{CAL}(\varphi,dx\wedge dy) \leq \frac{\pi^2}{n} - \frac{2\pi}{n} (1-\epsilon)^2 \rho \pi = - \frac{\pi^2}{n} ( 2\rho - 1) + \epsilon (2-\epsilon) \frac{2\pi^2}{n} \rho.
\]
If $\rho>1/2$ and $\epsilon$ is small enough, the latter number is negative. Therefore, we have proved the following fact.

\begin{prop}
\label{calpic}
For every $\epsilon>0$ there exists a smooth isotopy $\varphi_t\in \mathrm{Diff}_c(\D,dx\wedge dy)$, $t\in [0,1]$, with $\varphi_0=\mathrm{id}$ such that, setting $\varphi:= \varphi_1$, the following properties hold:
\begin{enumerate}[(i)]
\item $\|\varphi_t - \mathrm{id}\|_{\infty} < \epsilon$ for every $t\in [0,1]$;
\item $|\sigma_{\varphi,\lambda_0}| \leq \epsilon$ on $\D$ for every $t\in [0,1]$;
\item $\mathrm{CAL}(\varphi,dx\wedge dy)<0$;
\item all the fixed points of $\varphi$ have non-negative action.
\end{enumerate}
\end{prop}

\begin{rem} 
\label{newrem1}
An inspection to the proofs of the two propositions above shows that we could strengthen statement (ii) in both propositions by stating that $\sigma_{\varphi_t,\lambda_0} \geq - L + L/n$ for all $t\in [0,1]$ (Proposition \ref{sysgra}) and $|\sigma_{\varphi_t,\lambda_0}| \leq \epsilon$ for every $t\in [0,1]$ (Proposition \ref{calpic}). We will not make use of this here, but see Remark \ref{newrem2} below.
\end{rem}

\section{Global surfaces of section for Reeb flows} 

\subsection{Lift of disk maps to Reeb flows}
\label{liftsec}

The strategy of the proof of Theorem \ref{main2} is to construct a suitable contact form on $S^3$ whose Reeb flow has a disk-like global surface of section, such that the first return map is conjugated to one of the diffeomorphisms $\varphi\in \mathrm{Diff}_c(\D,dx\wedge dy)$ which are constructed in Section \ref{controsec}.

In this section, we prove two general statements about lifting compactly supported area-preserving diffeomorphisms of the disk to Reeb flows on the solid torus and on the three-sphere. The following result is a variation of \cite[Proposition 4.2]{bra08}.

\begin{prop}
\label{barney}
Let $\varphi\in \mathrm{Diff}_c(\D,dx\wedge dy)$, let $\lambda$ be a primitive of $dx\wedge dy$ on $\D$, and let $L$ be a positive number such that the function
\[
\tau:= \sigma_{\varphi,\lambda} + L
\]
is positive on $\D$. Then there exists a smooth contact form $\beta$ on the solid torus $\D \times \R/L \Z$ with the following properties:
\begin{enumerate}[(i)]
\item $\beta = \lambda + ds$ in a neighborhood of $\partial \D \times \R/L \Z$, where $s$ denotes the coordinate on $\R/L \Z$; in particular, the Reeb vector field $R_{\beta}$ of $\beta$ coincides with $\partial/\partial s$ near the boundary of $\D \times \R/L \Z$, and its flow is globally well-defined;
\item for all $s\in \R/L \Z$ we have $(\imath_s)^* d\beta = dx\wedge dy$, where $\imath_s: \D \rightarrow \D \times \R/L\Z$ denotes the inclusion $z\mapsto (z,s)$;
\item each surface $\D \times \{s\}$ is transverse to the flow of $R_{\beta}$, and the orbit of every point in $\D \times \R/L \Z$ intersects $\D \times \{s\}$ both in the future and in the past;
\item the first return map and the first return time of the flow of $R_{\beta}$ associated to the surface $\D\times \{0\} \cong \D$ are the map $\varphi$ and the function $\tau$;
\item $\mathrm{vol}(\D \times \R/L \Z, \beta\wedge d\beta) = L \pi + \mathrm{CAL}(\varphi,dx\wedge dy)$.
\end{enumerate}
Furthermore:
\begin{enumerate}[(i)]
\setcounter{enumi}{5}
\item if $\sigma_{\varphi,\lambda}$ is uniformly close to the zero function and $\varphi$ is isotopic to the identity through a path taking values in a small $C^0$-neighborhood of the identity in $\mathrm{Diff}_c(\D,dx\wedge dy)$, then the flow of $R_{\beta}$ is close to the flow $(t,(z,s)) \mapsto (z,s+t)$ in the $C^0_{\mathrm{loc}}$-topology of maps from $\R \times \D \times \R/L\Z$ to $\D \times \R/L\Z$;
\item if $\lambda(\partial/\partial \theta)> 0$ on $\partial \D$, then $\beta$ is smoothly isotopic to $\lambda + ds$ by a path of contact forms on $\D \times \R/L\Z$ which satisfy (i), (ii) and (iii).
\end{enumerate}
\end{prop}

\begin{proof}
The fact that the minimum of $\tau$ is positive implies that the diffeomorphism
\[
f: \D \times \R \rightarrow \D \times \R, \qquad f(z,s) := \bigl(\varphi(z),s - \tau(z) \bigr), 
\]
defines a free action of $\Z$ on $\D \times \R$. We denote by $M$ the quotient by this action and by 
\[
p: \D \times \R \rightarrow M
\]
the quotient projection. Then $M$ is a smooth 3-manifold with boundary. Since
\[
f^*(\lambda + ds) = \varphi^* \lambda + f^* ds = \lambda + d\sigma_{\varphi,\lambda} + d (s\circ f) =  \lambda + d \tau + ds - d\tau = \lambda + ds,
\]
the contact form $\lambda + ds$ on $\D \times \R$ induces a contact form on the quotient manifold $M$, which we denote by $\eta$. Thus we have $p^* \eta = \lambda + ds$.

Since $f(z,\tau(z)) = (\varphi(z),0)$, $f$ maps the graph of $\tau$ onto $\D \times \{0\}$, so the region
\[
M_0 := \{(z,s) \in \D \times \R \mid 0 \leq s \leq \tau(z) \}
\]
is a fundamental domain for the quotient projection $p$, and $M$ can be seen as the space which is obtained from $M_0$ by identifying $(z,\tau(z))$ with $(\varphi(z),0)$. Therefore,
\[
\begin{split}
\mathrm{vol}(M,\eta\wedge d\eta) &= \mathrm{vol}(M_0,(\lambda + ds) \wedge d(\lambda + ds)) =  \mathrm{vol}(M_0,ds \wedge d\lambda)  \\ & = \mathrm{vol}(M_0,dx\wedge dy \wedge ds) = \int_{\D} \tau \, dx\wedge dy = L \pi + \int_{\D} \sigma_{\varphi,\lambda} \, dx\wedge dy \\ & = L \pi + \mathrm{CAL}(\varphi,dx\wedge dy).
\end{split}
\]
If $U$ is a small neighbourhood of $\partial \D$ in $\D$, then $\varphi|_U = \mathrm{id}$ and $\tau|_U = L$, and hence
\[
f(z,s) = (z,s-L) \qquad \forall (z,s)\in U \times \R.
\]
Therefore, $p(U \times \R)$ can be identified with $U \times \R/L\Z$, and the contact form $\eta$ equals $\lambda + ds$ on this neighbourhood of the boundary of $M$.

The Reeb vector field of the contact form $\lambda + ds$ is $\partial/\partial s$, with flow
\begin{equation}
\label{flow}
(t,(z,s)) \mapsto (z,s+t).
\end{equation}
The surfaces $\D \times \{s\}$ are transverse to this flow, and the orbit of the point $(z,0)$ reaches the graph of $\tau$ at time $\tau(z)$, and precisely at the point $(z,\tau(z))$, which is identified with $(\varphi(z),0)$ by the quotient projection. Instead, $p(z,t)$ does not belong to $p(\D \times \{0\})$ when $0<t<\tau(z)$. 
Since $p_*(\partial/\partial s) = R_{\eta}$, the above facts imply that each $p(\D \times \{s\})$ is a global surface of section for the flow of $R_{\eta}$, and that $\varphi$ and $\tau$ are the first return map and the first return time associated to the surface of section $p(\D \times \{0\})$.

The final step is to pull back the contact form $\eta$ to the solid torus $\D \times \R/L\Z$ by a suitable diffeomorphism onto $M$. In order to do this, we consider the diffeomorphism
\[
g: \D \times \R \rightarrow \D \times \R, \qquad g(z,s):= (z,s-L),
\]
and we claim that there exists a diffeomorphism $h: \D \times \R \rightarrow \D \times \R$ with the following properties:
\begin{enumerate}[(a)]
\item the diagram
\[
\begin{CD}
 \D \times \R @>{h}>> \D \times \R \\
 @V{g}VV @VV{f}V \\
 \D\times \R @>{h}>> \D \times \R
 \end{CD}
 \]
 commutes;
 \item $h(z,s)=(z,s)$ for every $(z,s)\in U \times \R$, where $U$ is a neighbourhood of $\partial \D$ in $\D$;
 \item $h(z,0)=(z,0)$ for every $z\in \D$;
 \item for all $s\in \R$ we have $(\jmath_s)^* dh^*(\lambda+ds) = dx\wedge dy$, where $\jmath_s: \D \rightarrow \D \times \R$ denotes the inclusion $z\mapsto (z,s)$;
 \end{enumerate}

Postponing for a moment the proof of the existence of a diffeomorphism $h$ with the above properties, we show how this implies the conclusions (i)-(v) of the proposition. By (a), the map $h$ induces a diffeomorphism
\[
\tilde{h} : \D \times \R/L\Z \rightarrow M
\]
which by (b) and (c) restricts to the identity on $U \times \R/L\Z$ and on $\D\times \{0\}$. By the properties of $\eta$ which are discussed above, the contact form $\beta:= \tilde{h}^* \eta$ satisfies (i), (iii), (iv) and (v). Since the contact form $\beta$ is obtained from $h^*(\lambda+ds)$ by passing to the quotient induced by the map $g$, property (ii) follows from (d).

There remains to construct a diffeomorphism $h: \D \times \R \rightarrow \D \times \R$ satisfying (a), (b), (c) and (d). Up to the replacement of $U$ with a smaller neighborhood of $\partial \D$, we can see the diffeomorphism $\varphi$ as the time-1 map of a Hamiltonian flow $\varphi_t$ on $\D$, which is defined by a Hamiltonian which has support in $\D \setminus U$ and depends 1-periodically on time. Therefore, $\varphi=\varphi_1$, $\varphi_t|_U = \mathrm{id}$ for every $t\in \R$, and
\begin{equation}
\label{gruppo}
\varphi_{t+1} = \varphi_t \circ \varphi \qquad \forall t\in \R,
\end{equation}
where the above identity is a property of the flow of any time-periodic vector field. We shall construct a map $h$ of the form
\begin{equation}
\label{ansatz}
h(z,s) = \bigl( \varphi_{s/L}^{-1}(z), S(\varphi_{s/L}^{-1}(z),s) \bigr) \qquad \forall (z,s)\in \D\times \R,
\end{equation}
where the smooth function $S: \D \times \R \rightarrow \R$ is to be determined. With this Ansatz, we compute:
\[
\begin{split}
f\circ h(z,s) &= \bigl( \varphi \circ \varphi_{s/L}^{-1}(z), S( \varphi_{s/L}^{-1}(z),s) - \tau(\varphi_{s/L}^{-1}(z)) \bigr), \\
h\circ g (z,s) &= \bigl( \varphi_{s/L -1}^{-1} (z) , S( \varphi_{s/L -1}^{-1} (z), s-L) \bigr).
\end{split}
\]
By inverting the identity (\ref{gruppo})  we find $\varphi\circ \varphi_{t+1}^{-1} = \varphi_t^{-1}$, which for $t=s/L-1$ implies the equality of the first components of $f\circ h$ and $h\circ g$. Therefore, condition (a) is equivalent to the equality of the second components of these two compositions, which is equivalent to
\begin{equation}
\label{cociclo}
S(z,s) - \tau(z) = S(\varphi(z),s-L) \qquad \forall (z,s)\in \D \times \R.
\end{equation}
We now show how to construct a solution $S$ of the above functional equation. Since the minimum of $\tau$ is positive, we can find a number $\delta\in (0,1)$ such that
\begin{equation}
\label{boundontau}
\tau \geq \delta L \qquad \mbox{on } \D.
\end{equation}
Let $\chi: [0,L] \rightarrow \R$ be a smooth function which takes the value $0$ in a neighbourhood of $0$, the value $1$ in a neighbourhood of $L$, and satisfies
\begin{equation}
\label{boundonchi}
0 \leq \chi' \leq \frac{1}{L\sqrt{1-\delta}} \qquad \mbox{on } [0,L].
\end{equation}
Such a function exists because $1/(L\sqrt{1-\delta})>1/L$. We define
\[
\tilde{S} :\D \times [0,L] \rightarrow \R, \qquad \tilde{S}(z,s) := s + \chi(s) (\tau(z) - L),
\]
and extend $\tilde{S}$ to the whole of $\D \times \R$ by setting
\[
S(z,s) := \left\{ \begin{array}{ll} \tilde{S}(\varphi^k(z),s-kL) - \sum_{j=k}^{-1} \tau(\varphi^j(z)) & \mbox{if } kL \leq s < (k+1)L, \; k\leq -1, \\
\tilde{S}(z,s) & \mbox{if } 0 \leq s < L, \\
\tilde{S}(\varphi^k(z),s-kL) + \sum_{j=0}^{k-1} \tau(\varphi^j(z)) & \mbox{if } kL \leq s < (k+1) L, \; k \geq 1, \end{array} \right.
\]
where the $k$'s vary among all integers. By construction, $S$ satisfies the functional equation (\ref{cociclo}). Moreover, $S$ is smooth. Indeed, it suffices to check the smoothness of $S$ near each slice $\D \times \{kL\}$, $k\in \Z$, and by (\ref{cociclo}) it suffices to check this near $\D \times \{L\}$. If $s<L$ and $L-s$ is small enough, then $\chi(s)=1$ and hence
\[
S(z,s) = \tilde{S}(z,s) = s + \tau(z) - L.
\]
On the other hand, if $s\geq L$ and $s-L$ is small enough, then $\chi(s-L)=0$ and hence
\[
S(z,s) = \tilde{S}(\varphi(z),s-L) + \tau(z) = s - L + \tau(z).
\]
The equality of the above two expressions shows that $S$ is smooth. Finally, notice that by (\ref{boundontau}) and (\ref{boundonchi}) we have
\[
\begin{split}
D_2 \tilde{S}(z,s) &= 1 + \chi'(s) (\tau(z) - L) \geq 1 + \chi'(s) (\delta L - L) = 1 - L\, \chi'(s) (1-\delta)\\ &  \geq 1 - \frac{1-\delta}{\sqrt{1-\delta}} = 1 - \sqrt{1-\delta} =: \epsilon > 0,
\end{split}
\] 
and hence
\begin{equation}
\label{derS}
D_2 S(z,s) \geq \epsilon >0 \qquad \forall (z,s) \in \D \times \R,
\end{equation}
by the definition of $S$. The inequality (\ref{derS}) implies that $h: \D \times \R \rightarrow \D \times \R$ is a diffeomorphism. Indeed, $h$ factorises as $h = h_2 \circ h_1$, where
\[
h_1: \D \times \R \rightarrow \D \times \R \qquad h_1(z,s) := \bigl( \varphi_{s/L}^{-1}(z),s \bigr)
\]
is clearly a diffeomorphism and
\[
h_2: \D \times \R \rightarrow \D \times \R \qquad h_2(z,s) := \bigl( z,S(z,s) \bigr) 
\]
is a diffeomorphism because of (\ref{derS}). 

Therefore, $h$ is a smooth diffeomorphism and satisfies (a). If $z$ belongs to the neighbourhood $U$ of $\partial \D$, then $\varphi(z)=z$ and $\tau(z)=L$, so, for every $s\in [kL,(k+1)L)$ with $k\in \Z$ we find
\[
S(z,s) = \tilde{S}(z,s-kL) + k \tau(z) = s-kL + kL = s.
\]
Thus, $S(z,s)=s$ on $U \times \R$. Since $\varphi_t|_U = \mathrm{id}$ for every $t\in \R$, we deduce that $h$ satisfies (b). Moreover, $h$ satisfies (c) because $\varphi_0 = \mathrm{id}$ and $S(\cdot,0)=0$ on $\D$. Finally, the form (\ref{ansatz}) of $h$ and the fact that the diffeomorphisms $\varphi_t$ preserve the area form $dx\wedge dy$ imply that (d) holds:
\[
(\jmath_s)^* dh^*(\lambda+ds) = (\jmath_s)^* h^* (d\lambda) = (\jmath_s)^* h^* (dx\wedge dy) = (\varphi_{s/L}^{-1})^*(dx\wedge dy) = dx\wedge dy.
\]

We now check that the flow of $R_{\beta}$ is $C^0_{\mathrm{loc}}$-close to the flow (\ref{flow}) on $\D \times \R/L\Z$ when the assumptions of (vi) are fulfilled. In this case, the isotopy $\varphi_t$ is uniformly close to the identity. Using also the fact that $\tau= \sigma_{\varphi,\lambda} + L$ is uniformly close to the constant function $L$, the function $S$ constructed above is easily seen to be $C^0_{\mathrm{loc}}$-close to the function $(z,s) \mapsto s$. Therefore, the map $h$ is $C^0_{\mathrm{loc}}$-close to the identity on $\D \times \R$. Since $h$ conjugates the Reeb flow of $h^* (\lambda + ds)$ and the flow (\ref{flow}) on $\D \times \R$, we deduce that these two flows are $C^0_{\mathrm{loc}}$-close. Being obtained from the Reeb flow of $h^* (\lambda + ds)$ by the quotient projection
\[
\D \times \R \rightarrow \D \times \R/L\Z,
\] 
the flow of $R_{\beta}$ is therefore $C^0_{\mathrm{loc}}$-close to the flow (\ref{flow}) on $\D \times \R/L\Z$, and (vi) is proved.

Statement (vii) follows from Lemma \ref{newlem} below.
\end{proof}

The next lemma is a version with boundary of a result of Giroux \cite[Proposition 2]{gir02} about the uniqueness up to isotopy of contact structures supported by an open book decomposition (see also \cite[Proposition 3.18]{etn06}).

\begin{lem}
\label{newlem}
Let $\lambda$ be a smooth 1-form on $\D$ such that $d\lambda>0$ on $\D$ and $\lambda(\partial/\partial \theta)> 0$ on $\partial \D$. Suppose that $\beta$ is a smooth contact form on $\D\times \R/\Z$ satisfying:
\begin{enumerate}[(i)]
\item $\beta=\lambda+ds$ near $\partial \D\times \R/\Z$;
\item the Reeb vector field $R_{\beta}$ is positively transverse to the disks $\D\times \{s\}$ for all $s\in \R/\Z$.
\end{enumerate}
Then $\beta$ is smoothly isotopic to $\lambda+ds$ by a path of contact forms which satisfy (i) and (ii). If moreover 
\begin{enumerate}[(i)]
\setcounter{enumi}{2}
\item $(\imath_s)^*d\beta=d\lambda$ for all $s\in \R/\Z$, where $\imath_s: \D \rightarrow \D\times \R/\Z$ denotes the  inclusion $z\mapsto (z,s)$, 
\end{enumerate}
then the contact forms of the isotopy satisfy (iii).
\end{lem}

\begin{proof}
By assumption (i) and by the positiveness of $d\lambda$, the non-vanishing 3-form $\beta\wedge d\beta$ is positive on $\D\times \R/\Z$. Assumption (ii) is equivalent to the fact that $d\beta$ restricts to a positive 2-form on each disk $\D\times \{s\}$. In particular, $d\beta \wedge ds$ is positive on $\D\times \R/\Z$.

By the assumption on $\lambda|_{\partial \D}$ and by (i) there exists $\epsilon>0$ such that
\begin{equation}
\label{nl1}
\lambda \left( \frac{\partial}{\partial \theta} \right) > 0 \qquad \mbox{on } \D \setminus (1-\epsilon) \D,
\end{equation}
and
\begin{equation}
\label{nl2}
\beta = \lambda + ds \qquad \mbox{on } \D \setminus (1-\epsilon) \D.
\end{equation}
Choose a smooth cut-off function $\chi: [0,1]\rightarrow [0,1]$ such that $\chi=1$ on $[0,1-\epsilon]$, $\chi=0$ near 1 and $\chi'\leq 0$ everywhere.

\medskip

\noindent{\em Claim 1. For every $R\geq 0$ the 1-forms
\[
\beta_R := \beta + R\chi(r)\, ds
\]
are contact forms on $\D\times \R/\Z$ satisfying (i) and (ii). They satisfy also (iii) if $\beta$ does.}

\medskip

By differentiating $\beta_R$ we find
\[
d\beta_R = d\beta + R\chi'(r) dr\wedge ds,
\]
and hence
\[
\beta_R \wedge d\beta_R = \beta\wedge d\beta + R \chi(r) \,d\beta \wedge ds- R \chi'(r)  \,dr \wedge \beta \wedge ds.
\]
Let us determine the sign of the three 3-forms in the latter expression. The first one is positive everywhere, as shown above. The function $R\chi$ is non negative and the 3-form $d\beta \wedge ds$ is positive, so the second term is non-negative. The last 3-form is supported in $\D \setminus (1-\epsilon)\D$, and there it is non-negative because $\chi'\leq 0$ and by (\ref{nl2})  the 3-form $dr \wedge \beta \wedge ds$ coincides with $dr \wedge \lambda \wedge ds$ on $\D \setminus (1-\epsilon)\D$, which is positive thanks to (\ref{nl1}). We conclude that $\beta_R \wedge d\beta_R$ is positive, and hence $\beta_R$ is a contact form.

The 1-form $\beta_R$ satisfies (i) by the corresponding property of $\beta$ and because $\chi$ vanishes near 1. It satisfies also (ii) because the restriction of $d\beta_R$ to each disk $\D\times \{s\}$ coincides with the restriction of $d\beta$. For the same reason, it satisfies (iii) if $\beta$ does.
This concludes the proof of Claim 1.

\medskip

We denote by $\alpha$ the contact form
\[
\alpha := \lambda + ds.
\]

\noindent{\em Claim 2. For every $R\geq 0$ the 1-forms
\[
\alpha_R := \alpha + R\chi(r)\, ds
\]
are contact forms on $\D\times \R/\Z$ satisfying (i), (ii) and (iii).}

\medskip

This is actually a particular case of Claim 1, with a slightly simpler proof, due to the special form of $\alpha$.

\medskip

\noindent{\em Claim 3. There exists $R_0>0$ such that for all $R\geq R_0$ the path of 1-forms
\begin{equation}
\label{nl3}
t \beta_R + (1-t) \alpha_R, \qquad t\in [0,1],
\end{equation}
consists of contact forms on $\D\times \R/\Z$ satisfying (i) and (ii). These contact forms satisfy also (iii) if $\beta$ does.}

\medskip

As $d\lambda\wedge ds>0$ and $\D\times \R/\Z$ is compact, there exists $R_0>0$ so that
\begin{equation}
\label{nl4}
R \,d\lambda \wedge ds + \beta \wedge d\alpha \geq 0 
\qquad \forall R\geq R_0.
\end{equation}
Similarly, as $d\beta\wedge ds>0$, possibly by taking $R_0$ larger we may also assume that
\begin{equation}
\label{nl5}
R \,d\beta \wedge ds + \alpha \wedge d\beta \geq 0\qquad \forall R\geq R_0.
\end{equation}
Now we compute the wedge product of the form (\ref{nl3}) by its differential:
\begin{equation}
\label{nl6}
\begin{split} 
\bigl(t\beta_R &+ (1-t) \alpha_R \bigr) \wedge d \bigl(t\beta_R + (1-t) \alpha_R \bigr) \\ &= t^2 \beta_R \wedge d\beta_R + (1-t)^2 \alpha_R \wedge d\alpha_R + t(1-t) \beta_R \wedge d\alpha_R + t(1-t) \alpha_R \wedge d\beta_R.
\end{split}
\end{equation}
By Claims 1 and 2, the sum of the first two terms in the last expression defines a positive 3-form for every $R>0$. Let us check that the other two 3-forms are non-negative when $R\geq R_0$. We compute:
\[
\beta_R \wedge d\alpha_R = \beta \wedge d\alpha - R\chi'(r) \, dr\wedge \beta \wedge ds + R\chi(r)\, d\lambda \wedge ds.
\]
The middle term is non-negative, as shown in the proof of Claim 1. Therefore,
\[
\beta_R \wedge d\alpha_R \geq \beta \wedge d\alpha + R\chi(r)\, d\lambda \wedge ds.
\]
The sum of the two 3-forms on the right-hand side coincides with the form (\ref{nl4}) on $(1-\epsilon)\D$, and hence it is non-negative there. On the complement of this set it coincides with
\[
\alpha\wedge d\alpha + R \chi(r) d\lambda \wedge ds,
\]
and since $\alpha\wedge d\alpha$ and $d\lambda \wedge ds$ are positive, the above form is also positive. We conclude that the third term in (\ref{nl6}) is non-negative for every $R\geq R_0$.

The last mixed term in (\ref{nl6}) is
\[
\alpha_R \wedge d\beta_R = \alpha\wedge d\beta - R \chi'(r) \,dr\wedge \alpha \wedge ds + R\chi(r)\, d\beta\wedge ds.
\]
The middle term is non-negative because of (\ref{nl1}) and (\ref{nl2}), as in the proof of Claim 1. Therefore,
\[
\alpha_R \wedge d\beta_R \geq \alpha\wedge d\beta+ R\chi(r)\, d\beta\wedge ds.
\]
The 3-form on the right-hand side coincides with the form (\ref{nl5}) on $(1-\epsilon)\D$, and hence it is non-negative there. On the complement of this set it coincides with
\[
\alpha\wedge d\alpha + R\chi(r) d\lambda\wedge ds,
\]
which is positive. We conclude that the last term in (\ref{nl6}) is non-negative for every $R\geq R_0$. This shows that if $R\geq R_0$ then the 1-forms defined in (\ref{nl3}) are contact forms.

These 1-forms satisfy properties (i) and (ii), being convex combinations of 1-forms satisfying the same properties. For the same reason, they satisfy (iii) if $\beta$, and hence $\beta_R$, does. This concludes the proof of Claim 3.

\medskip 

We can finally prove the Lemma: Claim 1 gives us a path of contact forms from $\beta=\beta_0$ to $\beta_{R_0}$, Claim 3 a path of contact forms from $\beta_{R_0}$ to $\alpha_{R_0}$ and Claim 2 a path of contact forms from $\alpha_{R_0}$ to $\alpha_0=\alpha$. All these contact forms satisfy (i) and (ii), and also (iii) if $\beta$ does.
\end{proof}

\begin{rem}
\label{newrem2}
The conclusion of statement (vii) in Proposition \ref{barney} holds also if we replace the assumption on $\lambda|_{\partial \D}$ by the assumption that $\varphi$ is isotopic to the identity through a path $\{\varphi_t\}_{t\in[0,1]} \subset \mathrm{Diff}_c (\D, dx \wedge dy)$ such that the function $\sigma_{\varphi_t,\lambda} + L$ is positive on $\D$ for every $t \in [0,1]$. The proof is actually more direct than the proof of the lemma above. This alternative form of (vii) could also be used in our arguments below, thanks to Remark \ref{newrem1}.
\end{rem}

Thanks to Proposition \ref{barney}, it is now easy to show that compactly supported area-preserving diffeomorphisms of the disk can be lifted to Reeb flows on $S^3$. We consider the smooth family of closed disks in $S^3$:
\begin{equation}
\label{open_book}
\Sigma_{\varsigma} := \{(z_1,z_2)\in S^3 \mid \mbox{either } z_2 = 0 \mbox{ or } z_2 \neq 0 \mbox{ and } \arg z_2=\varsigma\}, \qquad \varsigma\in \R/2\pi \Z.
\end{equation}
These disks have the same boundary: $\partial \Sigma_{\varsigma} = \Gamma$ for all $\varsigma\in \R/2\pi \Z$, where
\begin{equation}
\label{great_circle}
\Gamma:= S^3 \cap (\C \times \{0\}) = \{(z_1,z_2)\in S^3 \mid |z_1|=1, \; z_2 = 0\}
\end{equation}
is the image of a Reeb orbit of $R_{\alpha_0}$. Their interiors $\Sigma_{\varsigma}\setminus \Gamma$ define a smooth foliation of $S^3\setminus \Gamma$ by open disks. We single out one of these discs
\[
\Sigma:= \Sigma_0 = \{ (x_1,y_1,x_2,y_2)\in S^3 \mid x_2 \geq 0, \; y_2=0\}
\]
and we parametrise it by the map
\[
u: \D \rightarrow \Sigma, \qquad u(x,y) := (x,y,\sqrt{1-x^2-y^2},0).
\]
This map is a homeomorphism, and its restriction to the interior of $\D$ is a smooth embedding into $S^3$.

\begin{prop}
\label{liftS3}
Let $\lambda$ be a smooth primitive of $dx\wedge dy$ on $\D$ which coincides with $\lambda_0$ on a neighborhood of $\partial \D$.
Let $\varphi\in \mathrm{Diff}_c(\D,dx\wedge dy)$ be such that the function
\[
\tau:= \sigma_{\varphi,\lambda} + \pi
\]
is positive on $\D$. Then there exists a smooth contact form $\alpha$ on $S^3$ with the following properties:
\begin{enumerate}[(i)]
\item $\alpha$ coincides with $\alpha_0$ in a neighbourhood of $\Gamma$ in $S^3$, and in particular $\Gamma$ is a closed orbit of $R_{\alpha}$;
\item for every $\varsigma \in \R/2\pi \Z$, the restrictions of $d\alpha$ and of $d\alpha_0$ to $\Sigma_{\varsigma}$ coincide;
\item the flow of $R_{\alpha}$ is transverse to the interior of each $\Sigma_{\varsigma}$, and the orbit of every point in $S^3\setminus \Gamma$ intersects the interior of $\Sigma_{\varsigma}$ both in the future and in the past;
\item the first return map and the first return time associated to $\Sigma$ are the map $u\circ \varphi \circ u^{-1}$ and the function $\tau\circ u^{-1}$;
\item $\mathrm{vol}(S^3, \alpha\wedge d\alpha) = \pi^2 + \mathrm{CAL}(\varphi,dx\wedge dy)$;
\end{enumerate}
Furthermore:
\begin{enumerate}[(i)]
\setcounter{enumi}{5}
\item if $\sigma_{\varphi,\lambda}$ is uniformly close to the zero function and $\varphi$ is isotopic to the identity through a path taking values in a small $C^0$-neighborhood of the identity in $\mathrm{Diff}_c(\D,dx\wedge dy)$, then the flow of $R_{\alpha}$ is close to the flow of $R_{\alpha_0}$ in the $C^0_{\mathrm{loc}}$-topology of maps from $\R \times S^3$ to $S^3$;
\item the contact form $\alpha$ is smoothly isotopic to $\alpha_0$ on $S^3$ through a path of contact forms which satisfy (i), (ii) and (iii).
\end{enumerate}
\end{prop}

\begin{proof}
Consider the map
\[
f: \D \times \R/\pi \Z \rightarrow S^3
\]
whose expression in polar coordinates $(r,\theta)$ on $\D$  is
\[
f(r,\theta,s) := \bigl(r e^{i(\theta+2s)},\sqrt{1-r^2} e^{2is} \bigr).
\]
This map is continuous, maps $\mathrm{int}(\D)\times \R/\pi \Z$ diffeomorphically onto $S^3\setminus \Gamma$ and $\partial \D \times \R/\pi \Z$ onto $\Gamma$.  Moreover, the image of $\D \times \{s\}$ by $f$ is the closed disk $\Sigma_{2s}$, and the restriction of $f$ to $\D\times \{0\}$ coincides with the map $u$.

The standard contact form $\alpha_0$ can be written in polar coordinates $(r_1,\theta_1,r_2,\theta_2)$ on $S^3\subset \R^2 \times \R^2$ as
\[
\alpha_0 = \frac{1}{2} (r_1^2 \,d\theta_1 + r_2^2 \, d\theta_2 ) ,
\]
so its pull-back by $f$ is the contact form
\[
f^* \alpha_0 = \frac{1}{2} \bigl( r^2 (d\theta + 2 \,ds) + 2 (1-r^2)\, ds \bigr) = \frac{1}{2} r^2\, d\theta + ds = \lambda_0 + ds
\]
on $\D \times \R/\pi \Z$.

Let $\beta$ be the contact form on $\D \times \R/\pi \Z$ which is induced by $\varphi$ thanks to Proposition \ref{barney} with  $L=\pi$. Since $\beta$ coincides with $\lambda_0 + ds= f^* \alpha_0$ on a neighbourhood of $\partial \D \times \R/\pi \Z$, there is a unique smooth 1-form $\alpha$ on $S^3$ such that $f^* \alpha = \beta$. This 1-form coincides with $\alpha_0$ in a neighbourhood of $\Gamma$ and is a contact form such that
\[
\mathrm{vol}(S^3,\alpha\wedge d\alpha) = \mathrm{vol}(\D \times \R/\pi \Z, \beta\wedge d\beta) = \pi^2 + \mathrm{CAL}(\varphi,dx\wedge dy).
\]
Therefore, $\alpha$ satisfies conditions (i) and (v). It satisfies also condition (ii), thanks to the corresponding assertion in Proposition \ref{barney}. From the corresponding properties of the Reeb flow of $\beta$ and since 
\[
f|_{\D \times \{0\}}: \D \times \{0\} \rightarrow \Sigma
\]
agrees with $u$, we conclude that also (iii) and (iv) hold. Assertions (vi) and (vii) follow immediately from the corresponding assertions in Proposition \ref{barney}, since 
\[
\lambda \left( \frac{\partial}{\partial \theta} \right) = \lambda_0 \left( \frac{\partial}{\partial \theta} \right) = \frac{1}{2} 
\]
on $\partial \D$.
\end{proof}

\subsection{Proof of Theorem \ref{main2}}

We are now ready to prove Theorem \ref{main2}. We begin with the construction of a contact form $\alpha$ on $S^3$ whose Reeb flow belongs to a given $C^0_{\mathrm{loc}}$-neighborhood $\mathscr{R}$ of the flow of $R_{\alpha_0}$ and such that $T_{\min}(\alpha)^2$ is larger than $\mathrm{vol}(S^3,\alpha\wedge d\alpha)$.

Let $\varphi\in \mathrm{Diff}_c(\D,dx\wedge dy)$ be a diffeomorphism which satisfies the conditions of Proposition \ref{calpic} with $\epsilon\leq \pi/2$. In particular, the action of $\varphi$ with respect to $\lambda_0$ is not smaller than $-\pi/2$, so
\begin{equation}
\label{taubig}
\tau:= \sigma_{\varphi,\lambda_0} + \pi \geq  \frac{\pi}{2} \qquad \mbox{on } \D,
\end{equation}
and hence $\varphi$ induces a contact form on $S^3$ which satisfies the conditions of Proposition \ref{liftS3} with $\lambda=\lambda_0$. By reducing if necessary the size of $\epsilon$, properties (i) and (ii) of Proposition \ref{calpic} and Proposition \ref{liftS3} (vi) guarantee that the Reeb flow of $\alpha$ belongs to $\mathscr{R}$. Moreover, by Proposition \ref{calpic} (ii) and Proposition \ref{liftS3} (vii) the contact form $\alpha$ is smoothly isotopic to $\alpha_0$.

Since $\varphi$ has negative Calabi invariant (Proposition \ref{calpic} (iii)), we have by Proposition \ref{liftS3} (v)
\begin{equation}
\label{volpic}
\mathrm{vol}(S^3,\alpha\wedge d\alpha) = \pi^2 + \mathrm{CAL}(\varphi,dx\wedge dy) < \pi^2.
\end{equation}
Since $R_{\alpha}$ coincides with $R_{\alpha_0}$ in a neighbourhood of $\partial \Sigma$,  the circle $\partial \Sigma$ is a closed orbit of $R_{\alpha}$ with period $\pi$. By Proposition \ref{liftS3} (iii) and (iv), all other orbits of $R_{\alpha}$ correspond to interior periodic points of $\varphi$, and if $z\in \mathrm{int}(\D)$ is a $k$-periodic point of $\varphi$ then the corresponding closed orbit of $R_{\alpha}$ has period
\[
T(z):= \sum_{j=0}^{k-1} \tau(\varphi^j(z)).
\]
Consider first the case $k=1$, that is, the case in which $z$ is a fixed point of $\varphi$. By condition  (iv) in Proposition \ref{calpic}, the action of $z$ is non-negative, and hence the period of the corresponding closed orbit of $R_{\alpha}$ is not smaller than $\pi$:
\[
T(z) = \tau(z) = \sigma_{\varphi,\lambda_0}(z) + \pi \geq \pi.
\] 
Consider now the case $k\geq 2$. By (\ref{taubig}),  we find that the period of the closed orbit of $R_{\alpha}$ which corresponds to the $k$-periodic point $z$ satisfies
\[
T(z) = \sum_{j=0}^{k-1} \tau(\varphi^j(z)) \geq k \frac{\pi}{2} \geq \pi.
\]
Therefore, all the closed orbits of $R_{\alpha}$ have period at least $\pi$ and, since $\partial \Sigma$ is a closed orbit of period $\pi$, we obtain
\[
T_{\min}(\alpha) = \pi.
\]
By (\ref{volpic}) we conclude that $\alpha$ is a contact form on $S^3$ which satisfies
\[
T_{\min}(\alpha)^2 > \mathrm{vol}(S^3,\alpha\wedge d\alpha),
\]
and the first part of Theorem \ref{main2} is proved. 

The construction of a contact form $\alpha$ on $S^3$ isotopic to $\alpha_0$ and such that
\begin{equation}
\label{conclusion}
T_{\min}(\alpha)^2 \geq c \, \mathrm{vol}(S^3,\alpha\wedge d\alpha)
\end{equation}
uses the map $\varphi$ of Proposition \ref{sysgra} instead of that of Proposition \ref{calpic}. Indeed, let $\varphi\in \mathrm{Diff}_c(\D,dx\wedge dy)$ and $n$ be a diffeomorphism and a natural number satisfying the conditions of Proposition \ref{sysgra} with $\epsilon=\pi^2/c$ and $L=\pi$. By condition (ii) in this proposition, the function $\tau:= \sigma_{\varphi,\lambda_0} + \pi$ satisfies
\begin{equation}
\label{taugra}
\tau \geq \frac{\pi}{n} \qquad \mbox{on } \D,
\end{equation}
and hence $\varphi$ induces a contact form $\alpha$ on $S^3$ which satisfies the conditions of Proposition \ref{liftS3}. This contact form is isotopic to $\alpha_0$ because of Proposition \ref{liftS3} (vii). By Proposition \ref{sysgra} (iii), the contact volume of $S^3$ with respect to $\alpha$ satisfies
\begin{equation}
\label{ilvolume}
\mathrm{vol}(S^3,\alpha\wedge d\alpha) = \pi^2 + \mathrm{CAL}(\varphi,dx\wedge dy) \leq \epsilon = \pi^2/c.
\end{equation}
The orbit $\partial \Sigma$ has period $\pi$, and by Proposition \ref{sysgra} (iv) all the fixed points $z$ of $\varphi$ correspond to periodic orbits of $R_{\alpha}$ of period not smaller than $\pi$:
\[
\tau(z) = \pi + \sigma_{\varphi,\lambda_0}(z) \geq \pi.
\]
Now let $z$ be a $k$-periodic point of $\varphi$ with $k\geq 2$. By Proposition \ref{sysgra} (v), $k$ is at least $n$ and by (\ref{taugra}) the period $T(z)$ of the corresponding closed orbit of $R_{\alpha}$ satisfies
\[
T(z) = \sum_{j=0}^{k-1} \tau(\varphi^j(z)) \geq k \frac{\pi}{n} \geq \pi.
\]
We conclude that all the closed orbits of $R_{\alpha}$ have period at least $\pi$, and hence $T_{\min}(\alpha) = \pi$. Therefore, (\ref{conclusion}) follows from (\ref{ilvolume}).

\begin{rem}
\label{howclose}
The first return map $\varphi$ of the Reeb flow given by the contact form $\alpha$ which satisfies (\ref{conclusion}) can be chosen to be $C^0$-close to the identity, because of condition (i) in Proposition \ref{sysgra}, and the same is true for the isotopy which connects $\varphi$ to the identity. This implies that the flow of $R_{\alpha}$ becomes $C^0_{\mathrm{loc}}$-close to the flow of $R_{\alpha_0}$ after a suitable time reparametrization of its orbits. This reparametrization is clearly necessary, because the first return time $\tau$ is not uniformly close to $\pi$ (its integral on $\D$ is small). Therefore, one can find tight contact forms on $S^3$ with arbitrarily high systolic ratio and which are close to the standard one in the following weak sense: after a time reparametrization, their Reeb flow becomes $C^0_{\mathrm{loc}}$-close to the standard flow.
\end{rem}

\subsection{From Reeb flows on the three-sphere to diffeomorphisms of the disk}
\label{surfsecsec}

In Section \ref{liftsec}, we have seen how compactly supported area-preserving disk diffeomorphisms can be lifted to Reeb flows first on the solid torus $\D \times \R/\pi \Z$ and then on the sphere $S^3$. On $S^3$, the lifted flow coincides with the Reeb flow of the standard contact form $\alpha_0$ in a neighborhood of the great circle $\Gamma$, which is defined in (\ref{great_circle}).

In this section, we would like to perform the opposite construction, and go from a Reeb flow on $S^3$ having the great circle $\Gamma$ as closed Reeb orbit first to a Reeb flow on the solid torus, and then to an area-preserving diffeomorphism of the disk.  Since we do not want to assume that the Reeb flow coincides with the standard one on a whole neighborhood of $\Gamma$, we will not obtain a compactly supported diffeomorphism of the disk, and special care is needed to control the boundary behaviour. Because of this fact, it is convenient to work with a map from the solid torus to the 3-sphere which is smooth up to the boundary, instead of the map which is used in the proof of Proposition \ref{liftS3}. Indeed, we shall use the map
\[
f: \D \times \R/\pi \Z \rightarrow S^3, \qquad f(re^{i\theta},s) := \left( \sin \Bigl( \frac{\pi}{2} r \Bigr) e^{i(\theta + 2s)}, \cos \Bigl( \frac{\pi}{2} r \Bigr) e^{2is} \right),
\]
where $(r,\theta)$ are polar coordinates on the disk $\D$.
It is easy to check that $f$ is smooth up to the boundary and that its restriction to $\mathrm{int}(\D)\times \R/\pi \Z$ is a diffeomorphism onto $S^3\setminus \Gamma$. Moreover, $f$ maps $\partial \D \times \R/\pi \Z$ onto $\Gamma$, and for every $s\in \R/\pi \Z$ the restriction
\[
f(\cdot,s) : \D \rightarrow S^3
\]
is a smooth embedding with image the closed disk $\Sigma_{2s}$, which is defined in (\ref{open_book}).

The pull-back of the standard contact form $\alpha_0$ by $f$ is the 1-form
\[
\beta_0 := f^*(\alpha_0) = \frac{1}{2} \left( \sin^2 \Bigl( \frac{\pi}{2} r \Bigr) \, d(\theta+2s) +  \cos^2 \Bigl( \frac{\pi}{2} r \Bigr) \, d(2s) \right) = \frac{1}{2} \sin^2 \left( \frac{\pi}{2} r \right) \, d\theta + ds
\]
on $\D \times \R/\pi \Z$. Its differential is
\[
d\beta_0 = \frac{\pi}{2} \sin \left( \frac{\pi}{2} r \right) \cos \left( \frac{\pi}{2} r \right) \, dr\wedge d\theta = \frac{\pi}{4} \sin (\pi r) \, dr\wedge d\theta = \frac{ \pi\sin (\pi r)}{4r} \, dx\wedge dy.
\]
Notice that $\beta_0$ is a contact form on $\mathrm{int}(\D)\times \R/\pi \Z$, but its differential vanishes on $\partial \D \times \R/\pi \Z$. The Reeb vector field of $\beta_0$ on $\mathrm{int}(\D) \times \R/\pi \Z$ is
\[
R_{\beta_0} = \frac{\partial}{\partial s},
\]
which trivially extends smoothly to the closed solid torus $\D \times \R/\pi \Z$.

\begin{prop} 
\label{solidtorus}
Let $\alpha$ be a smooth contact form on $S^3$ such that
\[
R_{\alpha}= R_{\alpha_0}  \qquad \mbox{on } \Gamma,
\]
and consider the smooth 1-form  $\beta := f^* \alpha$ on $\D \times \R/\pi \Z$,
which is a contact form on $\mathrm{int}(\D) \times \R/\pi \Z$.  Then:
\begin{enumerate}[(i)]
\item $\beta(z,s)[ \partial/\partial \theta] = 1/2$, $\beta(z,s)[ \partial/\partial s] = 1$, and $d\beta(z,s)=0$ for every $(z,s)\in \partial \D\times \R/\pi \Z$;
\item the Reeb vector field $R_{\beta}$ of the contact form $\beta$ on $\mathrm{int}(\D) \times \R/\pi \Z$ extends smoothly to the closed solid torus $\D \times \R/\pi \Z$ to a vector field which is tangent to the boundary; 
\item for every integer $k\geq 0$ and every real number $\epsilon>0$ there exists a positive number $\sigma=\sigma(k,\epsilon)$ such that if $\|R_{\alpha}-R_{\alpha_0}\|_{C^{k+1}(S^3)} < \sigma$ then $\|R_{\beta} - \partial/\partial s\|_{C^k(\D\times \R/\pi \Z)} < \epsilon$.
\end{enumerate}
\end{prop}

\begin{proof}
Let $(z,s)\in \partial \D \times \R/\pi \Z$ and set $u:= f(z,s)\in \Gamma$. From the definition of $f$ we obtain that
\[
df(z,s)\left[ \frac{\partial}{\partial \theta} \right] = \frac{1}{2} R_{\alpha_0} (u), \qquad df(z,s) \left[ \frac{\partial}{\partial s} \right]  = R_{\alpha_0} (u).
\]
Together with the fact that $R_{\alpha_0}(u) = R_{\alpha}(u)$ is in the kernel of $d\alpha(u)$, we deduce that $d\beta=f^* (d\alpha)$ vanishes at $(z,s)$. Moreover,
\[
\begin{split}
\beta (z,s)\left[ \frac{\partial}{\partial \theta} \right] &= \alpha(u) \circ df(z,s) \left[ \frac{\partial}{\partial \theta} \right] = \frac{1}{2} \alpha(u) \bigl[R_{\alpha_0}(u)\bigr] = \frac{1}{2} \alpha(u) \bigl[R_{\alpha}(u)\bigr] = \frac{1}{2}, \\ 
\beta (z,s)\left[ \frac{\partial}{\partial s} \right] &= \alpha(u) \circ df(z,s) \left[ \frac{\partial}{\partial s} \right] = \alpha(u) \bigl[R_{\alpha_0}(u)\bigr] = \alpha(u) \bigl[R_{\alpha}(u)\bigr] = 1,
\end{split}
\]
and (i) is proved.

The smooth map
\[
g: \mathrm{int}(\D) \times \R/\pi \Z \rightarrow S^3, \qquad g(z,s) := \bigl( \sqrt{1-|z|^2} e^{2is}, e^{2is} z \bigr),
\]
is a diffeomorphism onto a neighborhood of $\Gamma$, maps $\{0\}\times \R/\pi \Z$ onto $\Gamma$ and satisfies
\[
g^* \alpha_0 = \frac{1}{2} (x\, dy - y \, dx) + ds, 
\]
where $(x,y,s)$ are the standard coordinates on $\D \times \R/\pi \Z$. Therefore, 
\[
g^* R_{\alpha_0} = R_{g^* \alpha_0} = \frac{\partial}{\partial s}.
\]
By the assumptions on $R_{\alpha}$, we can write
\[
g^* R_{\alpha} = (1+a_0) \frac{\partial}{\partial s} + a_1 \frac{\partial}{\partial x} + a_2 \frac{\partial}{\partial y},
\]
where $a_0$, $a_1$, and $a_2$ are smooth real functions on $\mathrm{int}(\D) \times \R/\pi \Z$ such that
\[
a_0(0,s) = a_1(0,s) = a_2(0,s) = 0 \qquad \forall s\in \R/\pi \Z.
\]
Moreover, for every non-negative integer $k$ we have the bounds
\begin{equation}
\label{pippo0}
\|a_j\|_{C^k(\sqrt{2}/2 \D \times \R/\pi \Z)} \leq c_k \left\|R_{\alpha} - R_{\alpha_0} \right\|_{C^k(S^3)}, \qquad j=0,1,2,
\end{equation}
where $c_k$ is a suitable positive number. Thanks to the identities
\[
\frac{\partial}{\partial x} = \frac{x}{r} \frac{\partial}{\partial r} - \frac{y}{r^2} \frac{\partial}{\partial \theta}, \qquad \frac{\partial}{\partial y} = \frac{y}{r} \frac{\partial}{\partial r} + \frac{x}{r^2} \frac{\partial}{\partial \theta},
\]
we find the formula
\begin{equation}
\label{formula}
g^* R_{\alpha} = (1+a_0) \frac{\partial}{\partial s} + \frac{x a_1 + y a_2}{r} \frac{\partial}{\partial r} + \frac{x a_2 - y a_1}{r^2} \frac{\partial}{\partial \theta} \qquad \mbox{on } \bigl( \mathrm{int}(\D) \setminus \{0\} \bigr) \times \R/\pi \Z.
\end{equation}

The map
\[
h: \bigl( \D \setminus \{0\} \bigr) \times \R/\pi \Z \rightarrow \mathrm{int}(\D) \times \R/\pi \Z, \qquad h(re^{i\theta},s) := \left( \cos \Bigl( \frac{\pi}{2} r \Bigr) e^{-i\theta}, \frac{\theta}{2} + s \right),
\]
is smooth, maps $\partial \D \times \R/\pi \Z$ onto $\{0\} \times \R/\pi \Z$ and fits into the commutative diagram
\[
\xymatrix{ \bigl( \D \setminus \{0\} \bigr)\times \R/\pi \Z \ar[dd]_h \ar[rd]^f & \\ & S^3 \\ \mathrm{int}(\D) \times \R/\pi \Z \ar[ur]_g & }
\]
The map $h$ is bijective, and its inverse is the map
\[
h^{-1}(re^{i\theta},s) = \left( \frac{2}{\pi} \arccos (r) e^{-i\theta}, \frac{\theta}{2} + s \right),
\]
which is smooth on $(\mathrm{int}(\D) \setminus \{0\}) \times \R/\pi \Z$. Moreover,
\[
\begin{split}
dh^{-1}(re^{i\theta},s) \left[ \frac{\partial}{\partial s} \right] &= \frac{\partial}{\partial s}, \\
dh^{-1}(re^{i\theta},s) \left[ \frac{\partial}{\partial r} \right] &= - \frac{2}{\pi} \frac{1}{\sqrt{1-r^2}} \frac{\partial}{\partial r}, \\
dh^{-1}(re^{i\theta},s) \left[ \frac{\partial}{\partial \theta} \right] &= -  \frac{\partial}{\partial \theta} + \frac{1}{2}  \frac{\partial}{\partial s}.
\end{split}
\]
By these identities and by the commutativity of the above diagram, (\ref{formula}) implies the following formula for the Reeb vector field $R_{\beta}$ in the point  $(re^{i\theta},s)\in (\mathrm{int}(\D) \setminus \{0\}) \times \R/\pi \Z$:
\begin{equation}
\label{pippo}
\begin{split}
R_{\beta} &= f^* R_{\alpha} = (g\circ h)^* R_{\alpha} = h^* (g^* R_{\alpha}) = dh^{-1}(h)\bigl[ g^* R_{\alpha} (h) \bigr] =
\\ &= (1+a_0 \circ h) dh^{-1}(h) \left[ \frac{\partial}{\partial s} \right] + \bigl( a_1 \circ h \cos \theta - a_2 \circ h \sin \theta \bigr)  dh^{-1}(h) \left[ \frac{\partial}{\partial r} \right] \\ & \quad +  \frac{a_2 \circ h \cos \theta + a_1 \circ h \sin \theta}{\cos \left( \frac{\pi}{2} r \right)} dh^{-1}(h) \left[ \frac{\partial}{\partial \theta} \right] 
\\ &=  (1+a_0\circ h) \frac{\partial}{\partial s} - \frac{2}{\pi} \frac{a_1 \circ h \, \cos \theta - a_2 \circ h \, \sin \theta}{\sin \left( \frac{\pi}{2} r\right)} \frac{\partial}{\partial r}  \\ &\quad + \frac{a_2 \circ h \, \cos \theta + a_1 \circ h \, \sin \theta}{\cos \left( \frac{\pi}{2} r \right)}  \left( - \frac{\partial}{\partial \theta} + \frac{1}{2} \frac{\partial}{\partial s} \right) \\ &= \frac{2}{\pi} \frac{a_2 \circ h \, \sin \theta - a_1 \circ h \, \cos \theta}{\sin \left( \frac{\pi}{2} r\right)} \frac{\partial}{\partial r} - \frac{a_1 \circ h \, \sin \theta + a_2 \circ h\, \cos \theta}{\cos \left( \frac{\pi}{2} r \right)}  \frac{\partial}{\partial \theta} \\ &\quad + \left( 1 + a_0 \circ h + \frac{a_1 \circ h \, \sin \theta + a_2 \circ h\, \cos \theta}{2 \cos \left( \frac{\pi}{2} r \right)} \right) \frac{\partial}{\partial s}.
\end{split}
\end{equation}
From the fact that $a_j$ vanishes on $\{0\} \times \R/\pi \Z$ we deduce that $a_j\circ h$ vanishes on the boundary of the solid torus $\D \times \R/\pi \Z$. Using the fact that the function $r\mapsto \cos(\pi r/2)$ vanishes with order $1$ at $r=1$, we can write
\[
a_j \circ h(re^{i\theta},s) = \cos \left( \frac{\pi}{2} r \right) \tilde{a}_j (r e^{i\theta},s), \qquad \mbox{on } \bigl( \D \setminus \{0\} \bigr) \times \R/\pi \Z, \; j=0,1,2,
\]
where the functions $\tilde{a}_j$ are smooth on $(\D \setminus \{0\}) \times \R/\pi \Z$. An explicit expression for $\tilde{a}_j$ is
\[
\tilde{a}_j(re^{i\theta},s) = \frac{r-1}{\cos \left( \frac{\pi}{2} r \right)} \int_0^1 \partial_r a_j\circ h \bigl( (r+\rho(1-r))e^{i\theta},s \bigr) \, d\rho,
\]
where $\partial_r$ denotes the radial derivative and the function $(r-1)/\cos(\pi r/2)$ has a smooth extension at $r=1$. From the above formula and (\ref{pippo0}) we obtain the bounds
\begin{equation}
\label{pippo1}
\|\tilde{a}_j\|_{C^k((\D \setminus 1/2 \D)\times \R/\pi \Z)} \leq \tilde{c}_k \left\|R_{\alpha} - R_{\alpha_0} \right\|_{C^{k+1}(S^3)}, \qquad j=0,1,2,
\end{equation}
for suitable numbers $\tilde{c}_k$. Then (\ref{pippo}) can be rewritten as
\[
\begin{split}
R_{\beta} = &\frac{2}{\pi} \frac{a_2 \circ h \, \sin \theta - a_1 \circ h \, \cos \theta}{\sin \left( \frac{\pi}{2} r\right)} \frac{\partial}{\partial r} - \bigl( \tilde{a}_1  \sin \theta + \tilde{a}_2 \cos \theta \bigr) \frac{\partial}{\partial \theta} \\ &+ \left( 1 + a_0 \circ h + \frac{1}{2} \tilde{a}_1   \sin \theta + \frac{1}{2} \tilde{a}_2  \cos \theta \right) \frac{\partial}{\partial s},
\end{split}
\]
on $(\mathrm{int}(\D) \setminus \{0\}) \times \R/\pi \Z$. The above formula shows that $R_{\beta}$ has a smooth extension to the closed solid torus $\D \times \R/\pi \Z$, which is tangent to the boundary. This proves (ii).

Together with (\ref{pippo0}) and (\ref{pippo1}), the above formula implies that $R_{\beta}$ is $C^k$-close to $\partial/\partial s$ on $ (\D \setminus 1/2 \D) \times \R/\pi \Z$ when $R_{\alpha}$ is $C^{k+1}$-close to $R_{\alpha_0}$.
On the other hand, since $f$ is a diffeomorphism in the interior of $\D\times \R/\pi \Z$, the restriction of $R_{\beta} = f^* R_{\alpha}$ to $1/2 \D \times \R/\pi \Z$ is $C^k$-close to $\partial/\partial s = f^* R_{\alpha_0}$ when $R_{\alpha}$ is $C^k$-close to $R_{\alpha_0}$ on $S^3$. We conclude that if $R_{\alpha}$ is $C^{k+1}$-close $R_{\alpha_0}$, then $R_{\beta}$ is $C^k$-close to $\partial/\partial s$. This proves (iii). 
\end{proof}

Let $\alpha$ be a contact form on $S^3$ as in the above proposition. If  $R_{\alpha}$ is $C^1$-close to $R_{\alpha_0}$ then the smooth extension of the vector field $R_{\beta} = R_{f^* \alpha}$ to $\D \times \R/\pi \Z$ - which is still denoted by $R_{\beta}$ - is $C^0$-close to $\partial/\partial s$, and in particular it is transverse to the foliation $\{\D \times \{s\} \}_{s\in \R/\pi \Z}$.

Denote by $\phi_t$ the flow of $R_{\beta}$ on $\D\times \R/\pi \Z$. From the fact that the flow of $R_{\beta}$ is transverse to the foliation $\{\D \times \{s\} \}_{s\in \R/\pi \Z}$, we get the existence and smoothness of the first return time 
\[
\tau: \D \rightarrow \R, \qquad \tau(z) := \inf\{ t>0 \mid \phi_t(z,0) \in \D \times \{0\} \},
\]
and of the first return map
\[
\varphi: \D \rightarrow \D, \qquad \varphi(z) = p_1 ( \phi_{\tau(z)} (z,0)),
\]
where $p_1: \D \times \R/\pi \Z \rightarrow \D$ denotes the projection onto the first factor. The diffeomorphism $\varphi$ is isotopic to the identity through the smooth path of diffeomorphisms
\[
\varphi_s : \D \rightarrow \D, \qquad \varphi(z) = p_1 ( \phi_{\tau_s(z)} (z,0)),
\]
where $\tau_0$ is identically zero and
\[
\tau_s(z) := \inf\{ t>0 \mid \phi_t(z,0) \in \D \times \{\pi s\} \} \qquad \forall s\in (0,1].
\]

We denote by $\lambda$ and $\omega=d\lambda$ the smooth 1-form and 2-form on $\D$ which are obtained as pull-backs of $\beta$ and $d\beta$ by the smooth embedding
\[
\D \hookrightarrow \D \times \R/\pi \Z, \qquad z \mapsto (z,0).
\]
The fact that $\beta$ is a contact form in the interior of $\D\times \R/\pi \Z$ and the fact that $R_{\beta}$ is transverse to $\D \times \{0\}$ imply that $\omega$ does not vanish in the interior of $\D$. However, $\omega$ vanishes on the boundary of $\D$ because $d\beta$ vanishes on $\partial \D \times \R/\pi \Z$. From Proposition \ref{solidtorus} (i), we 
deduce that
\[
\int_{\D} \omega = \int_{\partial \D} \lambda =   \int_{\partial \D \times \{0\}} \beta = \pi.
\]
From the fact that the flow of $R_{\beta}$ preserves $\beta$ we deduce the identity
\begin{equation}
\label{taulambda}
\varphi^* \lambda = \phi_{\tau}^* \lambda + \lambda(R_{\beta}) \, d\tau = \lambda + d\tau,
\end{equation}
which by differentiation implies that $\varphi^* \omega = \omega$. Therefore, $\varphi$ belongs to $\mathrm{Diff}(\D,\omega)$ and $\tilde{\varphi}:= [\{\varphi_s\}]$ is a lift of $\varphi$ to $\widetilde{\mathrm{Diff}}(\D,\omega)$.

\begin{lem}
\label{volume}
Assume, as above, that $\alpha$ coincides with $\alpha_0$ on $\Gamma$ and that $R_{\alpha}$ is $C^1$-close enough to $R_{\alpha_0}$.
Then the contact volume of $(S^3,\alpha)$ and the first return time $\tau$ are related by the identity
\[
\mathrm{vol}(S^3,\alpha \wedge d\alpha) = \mathrm{vol}(\D \times \R/\pi \Z, \beta\wedge d\beta) = \int_{\D} \tau\, \omega.
\]
\end{lem}

\begin{proof}
Since $\beta =f^* \alpha$, the first identity follows from the fact that $f$ is a diffeomorphism from the interior of the solid torus $\D\times \R/\pi \Z$ onto $S^3\setminus \Gamma$. The smooth map
\[
\psi: [0,1] \times \D \rightarrow \D \times \R/\pi \Z, \qquad \psi(s,z) = \phi_{s\tau(z)}(z,0),
\]
is bijective from $[0,1) \times \D$ onto $\D \times \R/\pi \Z$. Therefore,
\[
\mathrm{vol}(\D \times \R/\pi \Z, \beta\wedge d\beta) = \int_{[0,1]\times \D} \psi^*(\beta\wedge d\beta).
\]
Since
\[
\psi^* \beta(s,z) \left[ \frac{\partial}{\partial s} \right] = \beta(\phi_{s\tau(z)}(z,0))\bigl[ \tau(z) R_{\beta}(\phi_{s\tau(z)}(z,0)) \bigr] = \tau(z),
\]
and, for $\zeta \in T_z \D$,
\[
\begin{split}
\psi^* \beta(s,z)[(0,\zeta)] &= \beta( \phi_{s\tau(z)}(z,0)) \bigl[ D \phi_{s\tau(z)}(z,0)[(\zeta,0)] + s R_{\beta} (\phi_{s\tau(z)}(z,0)) d\tau(z) [\zeta] \bigr] \\ &= \beta(z,0) [(\zeta,0)] + s \, d\tau(z)[\zeta],
\end{split}
\]
there holds
\[
\psi^* \beta = \tau\, ds + \lambda + s\, d\tau = d (s\tau) + \lambda.
\]
Therefore,
\[
\psi^*(d\beta) = d\lambda,
\]
and
\[
\psi^*(\beta\wedge d\beta) = \bigl( d (s\tau) + \lambda ) \wedge d\lambda = ( \tau\, ds + s \, d\tau ) \wedge d\lambda = \tau\, ds\wedge d\lambda.
\]
Integration over $[0,1]\times \D$ gives
\[
\int_{[0,1]\times \D} \psi^*(\beta\wedge d\beta) = \int_{[0,1]\times \D} \tau\, ds\wedge d\lambda = \int_{\D} \tau \, d\lambda = \int_{\D} \tau \, \omega.
\]
which proves the second identity.
\end{proof}

Let $z\in \partial \D$ and let 
\[
\gamma_z : [0,\tau(z)] \rightarrow \partial \D \times \R/\pi \Z, \qquad \gamma_z(t) := \phi_t(z,0),
\]
be the portion of the orbit of $(z,0)$ up to the first return time. By Proposition \ref{solidtorus} (i), the pull-back of $\beta$ to $\partial \D \times \R/\pi \Z$ by the inclusion is $d\theta/2 + ds$, and we find
\[
\tau(z) = \int_{\gamma_z} \beta = \int_{\gamma_z} \left( \frac{1}{2} d\theta + ds \right) = \int_{\gamma_z} \frac{1}{2} d\theta + \pi = \int_{p_1\circ \gamma_z} \lambda + \pi = \int_{\{s\mapsto \varphi_s(z)\}} \lambda + \pi,
\]
where we have used the fact that the path $s\mapsto \varphi_s(z)$ is a reparametrization of $p_1 \circ \gamma_z$. Together with (\ref{taulambda}), this implies the following  identity relating the first return time $\tau$ and the action of $\tilde{\varphi}$ with respect to the primitive $\lambda$ of $\omega$:
\begin{equation}
\label{tausigma}
\tau = \sigma_{\tilde{\varphi},\lambda} + \pi.
\end{equation}
In particular, a fixed point $z\in \mathrm{int}(\D)$ of $\varphi$ corresponds to a periodic orbit $t\mapsto f\circ \phi_t(z,0)$ of $R_{\alpha}$ of period
\[
\tau(z) = \sigma_{\tilde{\varphi},\lambda}(z) + \pi.
\]
By integrating (\ref{tausigma}) over $\D$ with respect to $\omega$ we find, thanks to Lemma \ref{volume},
\[
\mathrm{vol}(S^3,\alpha\wedge d\alpha) = \mathrm{CAL}(\tilde{\varphi},\omega) + \pi^2.
\]

When the Reeb vector field $R_{\alpha}$ is $C^{k+1}$-close to $R_{\alpha_0}$ on $S^3$, $k\geq 0$, Proposition \ref{solidtorus} implies that the vector field $R_{\beta}$ is $C^k$-close to $\partial/\partial s$, and hence its flow $\phi_t$ is close to the flow $(t,(z,s)) \mapsto (z,s+t)$ in the $C^k_{\mathrm{loc}}$-topology of maps from $\R \times \D\times \R/\pi \Z$ to $\D \times \R/\Z$. It follows that $\tau$ is $C^k$-close to the constant function $\pi$ and $\varphi$ is $C^k$-close to the identity mapping. Moreover, the maps $\varphi_s$ are $C^k$-close to the identity mapping, uniformly for $s\in [0,1]$. 

We can summarize the above discussion in the following:

\begin{prop}
\label{surfsec}
Let $\alpha$ be a smooth contact form on $S^3$ such that
\[
R_{\alpha} = R_{\alpha_0}  \qquad \mbox{on } \Gamma.
\]
If $R_{\alpha}$ is sufficiently $C^1$-close to $R_{\alpha_0}$, then there are a smooth 2-form $\omega$ on $\D$ which is positive on $\mathrm{int}(\D)$, a smooth primitive $\lambda$ of $\omega$ on $\D$, and an element $\tilde{\varphi}\in \widetilde{\mathrm{Diff}}(\D,\omega)$ such that:
\begin{enumerate}[(i)]
\item $\mathrm{vol}(S^3,\alpha\wedge d\alpha) = \pi^2 + \mathrm{CAL}(\tilde{\varphi},\omega)$;
\item every interior fixed point $z$ of $\varphi:= \pi(\tilde{\varphi})$ corresponds to a closed orbit of $R_{\alpha}$ of period
\[
\tau(z) = \pi + \sigma_{\tilde{\varphi},\lambda}(z);
\]
\item for every integer $k\geq 0$ and every real number $\epsilon>0$ there exists a positive number $\rho=\rho(k,\epsilon)$ such that if $\|R_{\alpha}-R_{\alpha_0}\|_{C^{k+1}(S^3)} < \rho$ then $\|\varphi - \mathrm{id}\|_{C^k(\D,\D)} < \epsilon$.
\end{enumerate}
\end{prop}

\subsection{Proof of Theorem \ref{main1}}

The proof of Theorem \ref{main1} consists in applying Corollary \ref{fixpoint} from the Introduction to the element $\tilde{\varphi}\in \widetilde{\mathrm{Diff}}(\D,\omega)$ which is produced by Proposition \ref{surfsec}. In order to use this proposition, we first have to reduce the case of an arbitrary contact form close to a Zoll one to the case in which the great circle $\Gamma$ defined by (\ref{great_circle}) is a $\pi$-periodic Reeb orbit.

The first step of this reduction is the fact that, up to rescaling, all Zoll contact forms are strictly contactomorphic to the standard one. 

\begin{prop}
\label{zoll}
Let $\alpha$ be a Zoll contact form on $S^3$ and let $T$ be the common period of the orbits of the corresponding Reeb flow. Then there exists a smooth diffeomorphism $\varphi: S^3 \rightarrow S^3$ such that 
\[
\varphi^* \alpha = \frac{T}{\pi} \alpha_0.
\]
In particular,
\[
\mathrm{vol}(S^3,\alpha\wedge d\alpha) = T^2.
\]
\end{prop}

The proof of the above proposition is a simple modification of the proof of Theorem B.2 in \cite{abhs17}, which deals with Zoll contact forms on $SO(3)$ instead of $S^3$.

The second step of the reduction is the following result.  

\begin{prop}
\label{reduction}
For every integer $k\geq 0$ and every $\epsilon>0$ there exists $\delta=\delta(k,\epsilon)>0$ such that if $\alpha$ is a smooth contact form whose Reeb vector field $R_{\alpha}$ admits  a closed orbit of period $\pi$ and satisfies $\|R_{\alpha}-R_{\alpha_0}\|_{C^k} < \delta$, then there exists a smooth diffeomorphism $\psi: S^3 \rightarrow S^3$ such that the Reeb vector field of the contact form $\psi^* \alpha$ satisfies
\[
R_{\psi^* \alpha} = R_{\alpha_0} \qquad \mbox{on } \Gamma,
\]
and 
\[
\|R_{\psi^* \alpha} - R_{\alpha_0}\|_{C^k} < \epsilon.
\]
\end{prop}

\begin{proof}
We denote by
\[
\gamma_0: \R/\pi \Z \rightarrow S^3, \qquad \gamma_0(t) = (e^{2it},0),
\]
the closed Reeb orbit of $\alpha_0$ whose image is the great circle $\Gamma$ and by $\gamma: \R/\pi \Z \rightarrow S^3$ the closed orbit of period $\pi$ of $R_{\alpha}$. 
We denote by $\mathrm{U}(2)$ the unitary group  of $\C^2$, which acts on $S^3$ and leaves the standard contact form $\alpha_0$ invariant. Let $U\in \mathrm{U}(2)$ be such that
\[
\gamma(0) = U \gamma_0(0).
\]
The orbit of $R_{\alpha_0}$ starting at $\gamma(0)$ is $U \gamma_0$.
From the fact that $\|R_{\alpha} - R_{\alpha_0}\|_{C^{k}}$ is small, a standard argument involving Gronwall's Lemma implies that $\gamma$ is $C^{k+1}$-close to $U\gamma_0$.
Then it is easy to find a diffeomorphism $\tilde{\psi}: S^3 \rightarrow S^3$ which is $C^{k+1}$-close to the identity and satisfies $\tilde{\psi}\circ U \gamma_0 = \gamma$. 
Consider the diffeomorphism $\psi:= \tilde{\psi} \circ U$. Notice that the $C^{k+1}$-closeness of $\tilde{\psi}$ to the identity  and the compactness of $U(2)$ imply that $\psi$ and $\psi^{-1}$ are $C^{k+1}$-uniformly bounded. Since $\gamma$ is a $\pi$-periodic orbit of $R_{\alpha}$, $\gamma_0 = \psi^{-1}\circ \gamma$ is a $\pi$-periodic orbit of $R_{\psi^* \alpha}$, and hence
\[
R_{\psi^* \alpha} = R_{\alpha_0} \qquad \mbox{on } \Gamma.
\] 
Moreover, the $C^k$-distance of $R_{\psi^* \alpha}$ from $R_{\alpha_0}$ can be estimated as follows:
\begin{equation}
\label{erre}
\begin{split}
\|R_{\psi^* \alpha} - R_{\alpha_0}\|_{C^k} & \leq  \|R_{\psi^* \alpha} - R_{\psi^* \alpha_0}\|_{C^k} + \| R_{\psi^* \alpha_0} - R_{\alpha_0}\|_{C^k} \\ & = \|\psi^* R_{\alpha} - \psi^* R_{\alpha_0}\|_{C^k} + \|R_{U^* \tilde{\psi}^* \alpha_0} - R_{U^* \alpha_0}\|_{C^k} \\ &  = \|\psi^* (R_{\alpha} - R_{\alpha_0}) \|_{C^k} + \|U^* R_{\tilde\psi^* \alpha_0} - U^* R_{\alpha_0}\|_{C^k}\\ & = \|\psi^* (R_{\alpha} - R_{\alpha_0}) \|_{C^k}+ \|R_{\tilde\psi^* \alpha_0} -  R_{\alpha_0} \|_{C^k}.
\end{split}
\end{equation}
The uniform $C^{k+1}$-bound on $\psi$ and $\psi^{-1}$ and the $C^k$-smallness of $R_{\alpha}-R_{\alpha_0}$ imply that
\begin{equation}
\label{piccolo}
\|\psi^* (R_{\alpha} - R_{\alpha_0}) \|_{C^k}
\end{equation}
is small. The fact that $\tilde{\psi}$ is $C^{k+1}$-close to the identity implies that the forms
\[
\tilde{\psi}^* \alpha_0-\alpha_0 \qquad \mbox{and} \qquad d\tilde{\psi}^* \alpha_0 - d\alpha_0= \tilde{\psi}^* d\alpha_0 - d\alpha_0
\]
are $C^k$-small. This in turn implies that $R_{\tilde{\psi}^*\alpha_0}$ is $C^k$-close to $R_{\alpha_0}$. Together with the smallness of (\ref{piccolo}), this fact and 
(\ref{erre}) imply that $R_{\psi^* \alpha}$ is $C^k$-close to $R_{\alpha_0}$.
\end{proof}

The last ingredient of the reduction argument is the following result.

\begin{prop}
\label{contTmin} 
There is a $C^3$-neighborhood of $\alpha_0$ in the space of smooth contact forms on $S^3$ on which the function $T_{\min}$ is $C^3$-continuous.
\end{prop}

\begin{proof}
Since the contact condition is $C^1$-open, we can find a $C^1$-neighborhood $\mathscr{A}_0$ of $\alpha_0$ such that for every $\alpha\in \mathscr{A}_0$ the path
\[
\alpha_t := \alpha_0 + t(\alpha-\alpha_0), \qquad t\in [0,1],
\]
consists of contact forms. Then Gray's stability theorem allows us to associate to any $\alpha$ in $\mathscr{A}_0$ a diffeomorphism $\psi:S^3 \rightarrow S^3$ such that $\psi^* \alpha =f \alpha_0$, where $f$ is a smooth positive function on $S^3$. The proof of Gray's stability theorem, see e.g.\ \cite[Theorem 2.2.2]{gei08}, shows that the map $\alpha\mapsto f$ is continuous from the $C^{k+1}$-topology on the space of 1-forms to the $C^k$-topology on the space of real functions, for every integer $k\geq 0$. Since $T_{\min}(\psi^* \alpha) = T_{\min}(\alpha)$, we are reduced to showing that the function
\[
f \mapsto T_{\min}(f \alpha_0)
\]
is $C^2$-continuous on some $C^2$-neighborhood of the constant function $1$ on $S^3$. This fact follows from the more general Theorem 3.6 in \cite{apb14}. 
\end{proof}

\begin{rem}
Actually, one can show that $T_{\min}$ is $C^1$-continuous on a suitable $C^3$-neighborhood of $\alpha_0$. Indeed, by arguing as in the above proof, it is enough to show that the function $f\mapsto T_{\min}(f \alpha_0)$ is $C^0$-continuous on some $C^2$-neighborhood of the constant function $1$ on $S^3$. If the function $f$ is $C^2$-close enough to 1, then the flow of $R_{f \alpha_0}$ is conjugated to the characteristic flow on the boundary of a smooth convex bounded domain in $\R^4$, and the $C^0$-continuity of $f\mapsto T_{\min}(f \alpha_0)$ follows from Proposition \ref{cont} below.
\end{rem}

We are finally ready to prove Theorem \ref{main1}.

\begin{proof}[Proof of Theorem \ref{main1}]
It is enough to show that the standard contact form $\alpha_0$ has a $C^3$-neighborhood $\mathscr{A}_0$ such that for every $\alpha\in \mathscr{A}_0$ there holds
\[
T_{\min}(\alpha)^2 \leq \mathrm{vol}(S^3,\alpha\wedge d\alpha),
\]
with equality holding if and only if $\alpha$ is Zoll. Indeed, in this case Proposition \ref{zoll} implies that the set
\[
\mathscr{A} := \{c \, \psi^* \alpha \mid c\in \R\setminus \{0\}, \alpha\in \mathscr{A}_0, \; \psi\in \mathrm{Diff}(S^3)\}
\]
is a $C^3$-neighborhood of the space of smooth Zoll contact forms on $S^3$ which has the required property.

Let $\mathscr{U}$ be the $C^1$-neighborhood of the identity in $\mathrm{Diff}^+(\D)$ which is given by Theorem \ref{fixpoint} in the introduction, and let $\epsilon>0$ be such that any $\varphi\in \mathrm{Diff}^+(\D)$ with $\|\varphi - \mathrm{id}\|_{C^1} < \epsilon$ belongs to $\mathscr{U}$. Let $\rho=\rho(1,\epsilon)$ be the positive number which is given by Proposition \ref{surfsec}. Let $\delta=\delta(2,\rho)$ be the positive number which is given by Proposition \ref{reduction}. 

If $\alpha$ is an arbitrary contact form on $S^3$, the rescaled contact form $\alpha_1 := (\pi/T_{\min}(\alpha)) \alpha$ satisfies $T_{\min}(\alpha_1)=\pi$. From the bound
\[
\|\alpha_1 - \alpha_0\|_{C^3} \leq \frac{\pi}{T_{\min}(\alpha)} \|\alpha - \alpha_0\|_{C^3} + \left| \frac{\pi}{T_{\min}(\alpha)} - 1 \right| \|\alpha_0\|_{C^3}
\]
and from the continuity of $T_{\min}$ stated in Proposition \ref{contTmin}, we deduce that $\alpha_1$ is $C^3$-close to $\alpha_0$ when $\alpha$ is $C^3$-close to $\alpha_0$. Using also the fact that $R_{\alpha_1}$ is $C^2$-close to $R_{\alpha_0}$ when $\alpha_1$ is $C^3$-close to $\alpha_0$, we deduce the existence of a $C^3$-neighborhood $\mathscr{A}_0$ of $\alpha_0$ such that for every $\alpha\in \mathscr{A}_0$ the rescaled contact form $\alpha_1 = (\pi/T_{\min}(\alpha)) \alpha$ satisfies
\[
\|R_{\alpha_1} - R_{\alpha_0}\|_{C^2} < \delta.
\]

We wish to check that this $C^3$-neighborhood $\mathscr{A}_0$ has the desired property. Let $\alpha\in \mathscr{A}_0$. Then $\alpha_1 = (\pi/T_{\min}(\alpha)) \alpha$ satisfies 
\[
T_{\min}(\alpha_1) = \pi \qquad \mbox{and} \qquad \|R_{\alpha_1} - R_{\alpha_0}\|_{C^2} < \delta.
\] 
By Proposition \ref{reduction} and by the choice of $\delta$, there is a diffeomorphism $\psi:S^3 \rightarrow S^3$ such that the contact form $\alpha_2:= \psi^* \alpha_1$ satisfies
\[
R_{\alpha_2} = R_{\alpha_0} \qquad \mbox{on } \Gamma
\]
and 
\[
\|R_{\alpha_2} - R_{\alpha_0}\|_{C^2} < \rho.
\]
Since
\[
\frac{T_{\min}(\alpha)^2}{\mathrm{vol}(S^3,\alpha\wedge d\alpha)} = \frac{T_{\min}(\alpha_1)^2}{\mathrm{vol}(S^3,\alpha_1\wedge d\alpha_1)} = \frac{T_{\min}(\alpha_2)^2}{\mathrm{vol}(S^3,\alpha_2\wedge d\alpha_2)} = \frac{\pi^2}{\mathrm{vol}(S^3,\alpha_2\wedge d\alpha_2)},
\]
and since $\alpha$ is Zoll if and only if $\alpha_2$ is Zoll, it is enough to prove that
\[
\mathrm{vol}(S^3,\alpha_2\wedge d\alpha_2) \geq \pi^2,
\]
with equality holding if and only if $\alpha_2$ is Zoll. Since we already know that the equality holds when $\alpha_2$ is Zoll (see Proposition \ref{zoll}), we must show that if $\alpha_2$ is not Zoll then the strict inequality 
\begin{equation}
\label{DADIM}
\mathrm{vol}(S^3,\alpha_2\wedge d\alpha_2) > \pi^2
\end{equation}
holds.

By Proposition \ref{surfsec} and by the choice of $\rho$, there are a smooth 2-form $\omega$ on $\D$ which is positive on the interior, a smooth primitive $\lambda$ of $\omega$ on $\D$, and an element $\tilde{\varphi}\in \widetilde{\mathrm{Diff}}(\D,\omega)$ such that:
\begin{enumerate}[(i)]
\item $\mathrm{vol}(S^3,\alpha_2\wedge d\alpha_2) = \pi^2 + \mathrm{CAL}(\tilde{\varphi},\omega)$; 
\item every interior fixed point $z$ of $\varphi:= p(\tilde{\varphi})$ corresponds to a closed orbit of $R_{\alpha_2}$ of period
\[
\tau(z) = \pi + \sigma_{\tilde{\varphi},\lambda}(z);
\]
\item $\|\varphi - \mathrm{id}\|_{C^1} < \epsilon$.
\end{enumerate}

The fact that $\alpha_2$ is not Zoll implies that $\tilde{\varphi}$ is not the identity in $\widetilde{\mathrm{Diff}}(\D,\omega)$: if $\tilde{\varphi}$ is the identity, then all the points of $\D$ are fixed points with zero action, and  (ii) implies that all the orbits of $R_{\alpha_2}$ are closed and have period $\pi$.

If by contradiction (\ref{DADIM}) does not hold, then (i) implies that the Calabi invariant of $\tilde{\varphi}$ is non-positive.
By (iii) and by the choice of $\epsilon$, the diffeomorphism $\varphi$ belongs to $\mathscr{U}$. Therefore, we can apply Corollary \ref{fixpoint}, which implies that $\varphi$ has an interior fixed point $z$ with negative action. By (ii), this fixed point corresponds to a closed Reeb orbit of $R_{\alpha_2}$ with period less than $\pi$. This contradicts the fact that $T_{\min}(\alpha_2)=\pi$ and proves the theorem. 
\end{proof}

\section{Proof of the corollaries to Theorem \ref{main1}}
\label{proofs}

\subsection{General facts about starshaped domains and capacities}

In this section, we fix some notation and recall some general facts which are used in the proof of the corollaries of the main theorem.  

The characteristic distribution on a smooth hypersurface $S$ in $\R^{2n}$ is the 1-dimensional smooth distribution given by the kernel of the restriction of the symplectic form $\omega_0$ to $S$. Curves on $S$ which are tangent to the characteristic distribution are called characteristic curves. When $S$ is co-oriented, meaning that the normal bundle of $S$ is oriented, the characteristic distribution is oriented. Boundaries of smooth domains are always co-oriented by declaring the outer normal direction to be positive.
The action of a closed characteristic $\gamma$ on the co-oriented hypersurface $S$ is the number
\[
A(\gamma) := \int_{\tilde{\gamma}} \omega_0,
\]
where $\tilde{\gamma}$ is any parametrised disk in $\R^{2n}$ whose boundary spans $\gamma$.

A smooth bounded domain $A\subset \R^{2n}$ which is starshaped with respect to the origin is said to be strictly starshaped if the radial vector field is transverse to $\partial A$. Equivalently, $A$ has the form
\[
A = \{rz \in \R^{2n} \mid z\in S^{2n-1}, \; 0\leq r < a(z) \},
\]
where $a: S^{2n-1} \rightarrow \R$ is a smooth positive function. The identification $a \mapsto A$ allows us to push the $C^k$ topology to the space of bounded strictly starshaped smooth domains. If $A$ is as above, then the Liouville form $\lambda_0$ which is defined in (\ref{liouville-n}) restricts to a contact form on $\partial A$, and the radial diffeomorphism
\begin{equation}
\label{rad}
\rho: S^{2n-1} \rightarrow \partial A, \qquad \rho(z) = a(z) z,
\end{equation}
satisfies
\begin{equation}
\label{radice}
\rho^*(\lambda_0|_{\partial A}) = a^2  \lambda_0|_{S^{2n-1}}.
\end{equation}
Therefore, $\lambda_0|_{\partial A}$ pulls back to a contact form on $S^{2n-1}$ which induces the standard contact structure $\xi_0 := \ker \lambda_0|_{S^{2n-1}}$. Conversely, all contact forms on $S^{2n-1}$ inducing the standard contact structure $\xi_0$ are obtained in this way.

The fact that $\lambda_0$ is a primitive of $\omega_0$ implies that the 1-dimensional oriented distribution on $\partial A$ which is determined by the Reeb vector field of $\lambda_0|_{\partial A}$ coincides with the characteristic distribution of $\partial A$.  Moreover, the action of a closed orbit of the Reeb flow coincides with its period: if $\gamma: \R/T\Z \rightarrow \partial A$ is a $T$-periodic orbit of the Reeb flow of $\lambda_0|_{\partial A}$ and $\tilde{\gamma}$ is a parametrised disk in $\R^{2n}$ whose boundary spans $\gamma$, then by Stokes theorem
\[
A(\gamma) = \int_{\tilde{\gamma}} \omega_0 = \int_{\gamma} \lambda_0 = T,
\]
because $\lambda_0$ equals 1 on the Reeb vector field.

Now let $C$ be a bounded smooth convex domain. Let $c$ be either the Ekeland-Hofer (see \cite{eh89}) or the Hofer-Zehnder capacity (see \cite{hz94}). Then $c(C)$ coincides with the minimal action of a closed characteristic on $\partial C$:
\[
c(C) = \min \{ A(\gamma) \mid \gamma \mbox{ closed characteristic of } \partial C \}.
\]
In the case of the Hofer-Zehnder capacity, this is proved in  \cite[Proposition 4]{hz90}. In the case of the Ekeland-Hofer capacity, this is stated in \cite[Proposition 3.10]{vit89}.
Up to a translation, which does not affect the characteristic distribution, we may assume that $C$ is a neighbourhood of the origin, and hence a strictly starshaped domain. Then the action of a closed characteristic on $\partial C$ agrees with its period, when this curve is seen as a periodic orbit of the Reeb vector field of $\lambda_0|_{\partial C}$, and we have
\[
c(C) = T_{\min} (\lambda_0|_{\partial C}).
\]
In the special case of the ball of radius $r$, we find
\[
c(rB) = T_{\min} (\lambda_0|_{\partial (rB)}) = T_{\min} ( r^2 \lambda_0|_{\partial B}) = \pi r^2.
\]
The next result follows from Clarke's dual characterisation of the minimal action of a closed characteristic on the boundary of a convex set, see \cite[Theorem 4.1 (iv)]{ama15}.

\begin{prop}
\label{projection}
Let $c$ be a symplectic capacity on symplectic vector spaces whose value on the boundary of any bounded smooth convex domain $C$ coincides with the minimal action of a closed characteristic on $\partial C$. Let $P: \R^{2n} \rightarrow \R^{2n}$ be the symplectic projector onto a $2k$-dimensional symplectic subspace $V\subset \R^{2n}$. Then for every bounded convex smooth domain $C$ there holds 
\[
c(P(C)) \geq c(C),
\]
where $c(P(C))$ denotes the symplectic capacity of $P(C)$ as a convex domain in $V$.
\end{prop}

The following result is well known (see e.g.\ \cite[Theorem 4.1 (v)]{ama15}).

\begin{prop}
\label{cont}
The function $C \mapsto T_{\min}(\lambda_0|_{\partial C})$ is continuous with respect to the Hausdorff distance on the space of smooth bounded convex neighbourhoods of the origin.
\end{prop}

\subsection{Proof of Corollary \ref{cor1}}

Let $A\subset \R^4$ be a smooth domain whose closure is symplectomorphic to the closed ball of radius $r$. Since symplectomorphisms preserve the characteristic foliation, all the characteristic curves on $\partial A$ are closed and have action $\pi r^2$. If moreover $A$ is strictly starshaped, then the contact form $\lambda_0|_{\partial A}$ is Zoll and all the corresponding Reeb orbits have period $\pi r^2$. The next result says that the converse is also true.

\begin{prop}
\label{isaball}
Let $A$ be a bounded strictly starshaped smooth domain in $\R^4$. If the restriction of $\lambda_0$ to $\partial A$ is Zoll and $T$ is the common period of the orbits of the corresponding Reeb flow, then the closure of $A$ is symplectomorphic to a closed ball of radius $\sqrt{T/\pi}$.
\end{prop}

\begin{proof}
Notice that the restriction of $\lambda_0$ to the boundary of the homothetic domain $rA$ is a Zoll contact form with Reeb orbits having common period $Tr^2$. By replacing the domain $A$ by the homothetic domain $\sqrt{\pi/T} A$, we may assume that $T=\pi$. In this case, we have to show that the closure of $A$ is symplectomorphic to the closed unit ball $\overline{B}$.

By Propositon \ref{zoll}, there is a diffeomorphism
\[
\varphi: S^3 = \partial B \rightarrow \partial A
\]
such that
\[
\varphi^*(\lambda_0|_{\partial A}) = \alpha_0 = \lambda_0|_{\partial B}.
\]
Then the positively 1-homogeneous extension of $\varphi$, i.e.\ the smooth map 
\[
\tilde{\varphi} : \R^4 \setminus \{0\} \rightarrow \R^4 \setminus \{0\}, \qquad \tilde{\varphi}(r\hat{z}) = r \varphi(\hat{z}), \qquad \forall \hat{z}\in S^3, \; r>0,
\]
is a symplectomorphism mapping $\overline{B}\setminus \{0\}$ onto $\overline{A}\setminus \{0\}$. This map extends continuously in 0 by setting $\tilde{\varphi}(0)=0$, but this extension is in general not differentiable at 0. However, it is possible to modify $\tilde{\varphi}$ in an arbitrarily small neighbourhood of $0$ and make it a smooth symplectomorphism of $\R^4$ onto itself. This follows from the following theorem of Gromov and McDuff (see \cite[Theorem 9.4.2]{ms94}): Let $(M,\omega)$ be a connected symplectic 4-manifold which  has no symplectically embedded 2-sphere with self-intersection number $-1$, let $K_1\subset \R^4$ and $K_2\subset M$ be compact subsets with $K_1$ starshaped, and let $\tilde{\psi}: \R^4 \setminus K_1 \rightarrow M \setminus K_2$ be a symplectomorphism. Then for every neighbourhood $U$ of $K_2$ there exists a symplectomorphism $\psi: \R^4 \rightarrow M$ which coincides with $\tilde{\psi}$ on $\psi^{-1}(M\setminus U)$. Indeed, if we apply this theorem to $(M,\omega) = (\R^4,\omega_0)$, $K_1=K_2=\{0\}$, $\tilde{\psi}=\tilde{\varphi}$ and $U$ any neighborhood of $0$ which is contained in $A$, we obtain a symplectomorphism $\psi: \R^4 \rightarrow \R^4$ which maps the closure of $B$ onto the closure of $A$.
\end{proof} 

\begin{rem}
There is an alternative way of smoothing the symplectomorphism $\tilde{\varphi}$ near the origin, which does not make use of Gromov's and McDuff's theorem. It is based on the following fact: the space of contactomorphisms of the tight contact three-sphere which preserve the orientation of the contact distribution is connected (see \cite[Corollary 2.4.3]{eli92} or \cite[Theorem 4]{cs16}). Indeed, using this fact one can show that the diffeomorphism $\varphi: S^3 \rightarrow \partial A$ of the above proof can be connected to the identity by a path of diffeomorphisms $\{\varphi_t : S^3 \rightarrow \partial A_t\}_{t\in [0,1]}$, where the $A_t$ are bounded strictly starshaped smooth domains and $\varphi^*_t(\lambda_0|_{\partial A_t})= \lambda_0|_{S^3}$ for every $t\in [0,1]$. By considering the positively 1-homogeneous extension of these diffeomorphisms, we find  that the symplectomorphism $\tilde{\varphi} : \R^4 \setminus \{0\} \rightarrow \R^4 \setminus \{0\}$ is symplectically isotopic to the identity. Since $\R^4 \setminus \{0\}$ is simply connected, the symplectic isotopy from the identity to $\tilde{\varphi}$ is induced by a time-dependent Hamiltonian on $\R^4 \setminus \{0\}$. By multiplying this Hamiltonian by a smooth function on $\R^4$ which is supported in $\R^4 \setminus\{0\}$ and takes the value 1 outside of a small neighbourhood of $0$, we obtain a smooth time-dependent Hamiltonian on $\R^4$, whose time-1 map is a symplectomorphism mapping the closure of $B$ onto the closure of $A$.
\end{rem} 

\begin{rem}
An argument similar to the proof of Proposition \ref{isaball} allows one to prove the following fact: if the restrictions of $\lambda_0$ to the smooth boundaries of two bounded strictly starshaped domains $A_1,A_2\subset \R^4$ are strictly contactomorphic, then the closures of $A_1$ and $A_2$ are symplectomorphic.
\end{rem}

\begin{proof}[Proof of Corollary \ref{cor1}] We denote by $z_C\in \R^4$ the baricenter of the bounded convex domain $C\subset \R^4$. By Proposition \ref{isaball}, the bounded convex smooth domain $C$ belongs to the set $\mathscr{B}$ of all convex domains in $\R^4$ whose closure is symplectomorphic to a closed ball if and only if the restriction of $\lambda_0$ to the boundary of $C-z_C$ is Zoll.

By the identity (\ref{radice}), we can find a $C^3$-neighborhood $\mathscr{C}$ of $\mathscr{B}$ in the set of all smooth bounded convex domains in $\R^4$ such that for every $C\in \mathscr{C}$ the restriction $\alpha_C$ of $\lambda_0$ to the boundary of $C-z_C$ pulls back by the radial diffeomorphism $\rho: S^3 \rightarrow \partial(C-z_C)$ to a contact form on $S^3$ which belongs to the $C^3$-open set $\mathscr{A}$ of Theorem \ref{main1}. By Theorem \ref{main1} we obtain
\[
\begin{split}
c(C)^2 &= c(C-z_C)^2 = T_{\min}(\alpha_C)^2 = T_{\min}( \rho^* \alpha_C)^2 \leq \mathrm{vol} (S^3, \rho^* \alpha_C \wedge d(\rho^* \alpha_C)) \\ &= \mathrm{vol} (\partial (C-z_C), \alpha_C \wedge d\alpha_C) = \mathrm{vol} (\partial (C-z_C), (\lambda_0 \wedge d\lambda_0)|_{\partial (C-z_C)}) \\ &= \mathrm{vol} (C-z_C, \omega_0 \wedge \omega_0) = \mathrm{vol} (C, \omega_0 \wedge \omega_0) = 2 \, \mathrm{vol}(C),
\end{split}
\]
where the last volume is in terms of the standard volume form $dx_1 \wedge dy_1 \wedge dx_2 \wedge dy_2$ on $\R^4$, which coincides with $(1/2) \omega_0\wedge \omega_0$. Moreover, the equality holds if and only if $\alpha_C$ is Zoll, that is, if and only if $C-z_C$, or equivalently $C$, belongs to $\mathscr{B}$. This concludes the proof of Corollary \ref{cor1}.
\end{proof}

\subsection{Proof of Corollaries \ref{cor2} and \ref{cor3}}

The Euclidean inner product $(\cdot,\cdot)$ of $\R^{2n}$ is compatible with the symplectic form $\omega_0$, meaning that the isomorphims $J : \R^{2n} \rightarrow \R^{2n}$ which is defined by the identity
\[
\omega_0(u,v) = (Ju,v) \qquad \forall u,v\in \R^{2n},
\]
is a complex structure on $\R^{2n}$: $J^2=-I$. This $J$ is actually the standard complex structure of $\R^{2n}$, namely
\[
J (x_1,y_1,\dots,x_n,y_n) = (-y_1,x_1,\dots,-y_n,x_n), \qquad \forall (x_1,y_1,\dots,x_n,y_n) \in \R^{2n}.
\]
The next result is a simple generalisation of Theorem 1 in \cite{am13}. For the sake of completeness, we explain in detail how this statement can be deduced from \cite[Theorem 1]{am13}.

\begin{prop}
\label{linear}
Let $V$ be a $2k$-dimensional symplectic subspace of $\R^{2n}$ and let $P: \R^{2n} \rightarrow \R^{2n}$ be the symplectic projector onto $V$. Then for every linear symplectomorphism $\Phi: \R^{2n} \rightarrow \R^{2n}$ there holds
\begin{equation}
\label{inez}
\mathrm{vol}(P\Phi(B), \omega^k_0|_V) \geq \pi^k,
\end{equation}
with equality holding if and only if the subspace $\Phi^{-1}V$ is $J$-invariant. In the latter case, we have
\begin{equation}
\label{ide}
P \Phi(B) = \Phi(B \cap \Phi^{-1}V).
\end{equation}
\end{prop}

\begin{proof}
Theorem 1 in \cite{am13} has the extra assumption that $V$ is $J$-invariant. Its conclusion is that in this case (\ref{inez}) holds, with equality holding if and only if $\Phi^T V$ is $J$-invariant. The fact that the linear mapping $\Phi$ is symplectic and the $J$-invariance of $V$ imply that $\Phi^T V = \Phi^{-1} V$. Therefore, the equality holds in (\ref{inez}) if and only if $\Phi^{-1} V$ is $J$-invariant, as stated above. Assume that we are in this case: both $V$ and $\Phi^{-1} V$ are $J$-invariant. The first fact implies that the symplectic projector $P$ is an orthogonal projector, and hence self-adjoint: $P^T=P$. Moreover, using the fact that the image of the unit ball by a linear transformation $A: \R^{2n} \rightarrow V$ is $A(B\cap \ran A^T)$, where $\ran A^T$ denotes the image of the adjoint mapping $A^T: V \rightarrow \R^{2n}$, we find
\[
\begin{split}
P \Phi(B) &= P \Phi ( B \cap \ran (P\Phi)^T ) = P \Phi (B \cap \ran(\Phi^T P) ) \\ &= P\Phi (B \cap \Phi^T V) = P\Phi (B \cap \Phi^{-1} V) = \Phi (B \cap \Phi^{-1} V),
\end{split} 
\]
proving (\ref{ide}). This concludes the proof of the proposition in the case in which $V$ is $J$-invariant.

The general case can be deduced from the above one as follows. Let $(\cdot,\cdot)'$ be an $\omega_0$-compatible inner product on $\R^{2n}$ such that the projector $P$ is orthogonal. Let $B'$ be the corresponding unit ball, and let $J'$ be the corresponding  complex structure, which satisfies $J'V=V$. Let $\Psi: (\R^{2n},\omega_0,J') \rightarrow (\R^{2n},\omega_0,J)$ be a symplectic and complex linear isomorphism. Then $\Psi$ is an isometry from $(\R^{2n}, (\cdot,\cdot)')$ to $(\R^{2n}, (\cdot,\cdot))$, and hence $\Psi(B') = B$. By applying the previous case to the symplectic isomorphism $\Phi\Psi$ we obtain
\[
\mathrm{vol}(P\Phi(B), \omega^k_0|_V) = \mathrm{vol}(P\Phi\Psi(B'), \omega^k_0|_V) \geq \pi^k, 
\]
with equality holding if and only if $\Psi^{-1} \Phi^{-1} V$ is $J'$-invariant. From the identity $J'\Psi^{-1} = \Psi^{-1} J$ we obtain that the latter condition is equivalent to the fact that $\Phi^{-1} V$ is $J$-invariant:
\[
J' \Psi^{-1} \Phi^{-1} V = \Psi^{-1} \Phi^{-1} V \quad \iff \quad \Psi^{-1} J \Phi^{-1} V = \Psi^{-1} \Phi^{-1} V \quad \iff \quad J \Phi^{-1} V = \Phi^{-1} V.
\]
Now assume that the equality holds in (\ref{inez}). Then (\ref{ide}) holds for the symplectomorphism $\Phi\Psi$, and we deduce that
\[
\begin{split}
P \Phi(B) &= P \Phi \Psi(B') = \Phi \Psi (B' \cap (\Phi \Psi)^{-1} V) = \Phi \Psi ( \Psi^{-1} (B) \cap \Psi^{-1}( \Phi^{-1} V)) \\ &= \Phi \Psi ( \Psi^{-1} (B \cap \Phi^{-1} V)) = \Phi ( B \cap \Phi^{-1} V),
\end{split}
\]
proving (\ref{ide}) in the general case.
\end{proof}

\begin{proof}[Proof of Corollary \ref{cor2}]
Let $\Phi: \R^{2n} \rightarrow \R^{2n}$ be a linear symplectomorphism. We must show that if $\varphi: \overline{B} \rightarrow \R^{2n}$ is a smooth symplectomorphism with $\|\varphi-\Phi\|_{C^3(B)}$ small enough, then
\begin{equation}
\label{diadem}
\mathrm{vol}(P \varphi(B),\omega^2_0|V)\geq \pi^2.
\end{equation}
If 
\[
\mathrm{vol}(P \Phi(B),\omega^2_0|V)> \pi^2, 
\]
then the inequality (\ref{diadem}) trivially holds for any symplectomorphism $\varphi$ which is $C^0$-close to $\Phi$. Therefore, we can can assume that
\[
\mathrm{vol}(P \Phi(B),\omega^2_0|V)= \pi^2.
\]
By Proposition \ref{linear}, the convex set $P\Phi(B)$ is symplectomorphic to the unit ball of $V$ by a linear symplectomorphism. Up to the identification of $(V,\omega_0|_V)$ with $(\R^4,\omega_0)$ by means of a linear symplectomorphism mapping the unit ball of $V$ onto a ball in $\R^4$, we deduce that $P\Phi(B)$ belongs to the set $\mathscr{B}$. If $\varphi$ is $C^3$-close enough to $\Phi$, then the sets $\varphi(B)$ and $P\varphi(B)$ are smooth convex domains in $\R^{2n}$ and $V\cong \R^4$, respectively, and $P\varphi(B)$ belongs to the set $\mathscr{C}$ of Corollary \ref{cor1}. By using this corollary together with Proposition \ref{projection}, we find the chain of inequalities
\[
\pi^2 = c(B)^2 = c(\varphi(B))^2 \leq c(P \varphi(B))^2 \leq 2\, \mathrm{vol}(P \varphi(B)),
\]
where $\mathrm{vol}$ is the volume on $V$ which is given by the volume form $(1/2) \omega_0^2|_V$. The inequality (\ref{diadem}) follows.
\end{proof}

\begin{proof}[Proof of Corollary \ref{cor3}]
Let $\varphi: A \rightarrow \R^{2n}$ be a smooth symplectomorphism from an open subset $A\subset \R^{2n}$.
Corollary \ref{cor3} is an immediate consequence of Corollary \ref{cor2}, because for every $z\in A$ the path of symplectomorphisms 
\[
\varphi_r (\zeta) := \frac{1}{r} \varphi( z + r \zeta), \qquad \zeta\in \overline{B},
\]
converges to the linear symplectomorphism $D\varphi(z)$ for $r\downarrow 0$ in the norm of $C^k(\overline{B},\R^{2n})$, for every $k\in \N$, and this convergence is uniform for $z$ varying in a compact subset of $A$.
\end{proof}

\subsection{Proof of Corollary \ref{cor4}}
\label{proofcor4}

Let $F$ be a Finsler metric on $S^2$, that is, a fiberwise positively 1-homogeneous function $F: TS^2 \rightarrow [0,+\infty)$ such that $F$ is smooth away from the zero-section and the fiberwise second differential of $F^2$ is positive definite away from the zero section. Denote by $F^*$ the dual metric and by
\[
D^*(S^2,F) := \{\xi \in T^* S^2 \mid F^*(\xi) \leq 1\}, \qquad
S^*(S^2,F) := \{ \xi \in T^* S^2 \mid F^*(\xi) = 1\}
\]
the induced unit cotangent disk bundle and unit cotangent sphere bundle. Denote by $\lambda$ the standard Liouville 1-form on $T^* S^2$, which in local coordinates has the expression
\[
\lambda = p_1 \, dq_1 + p_2 \, dq_2,
\]
and by $\omega := d\lambda$ the standard symplectic form on $T^* S^2$. The restriction of $\lambda$ to $S^*(S^2,F)$ is a contact form, which we denote by $\lambda_F$.
The Holmes-Thompson area of $(S^2,F)$ is the positive number
\[
\mathrm{area} (S^2,F) := \frac{1}{2\pi} \, \mathrm{vol}( D^*(S^2,F), \omega\wedge \omega ) = 
\frac{1}{2\pi} \, \mathrm{vol}( S^*(S^2,F),\lambda_F \wedge d\lambda_F).
\]
The Reeb flow of $\lambda_F$ on $S^*(S^2,F)$ corresponds via Legendre duality to the geodesic flow on the unit tangent sphere bundle of $(S^2,F)$. In particular,
\[
\ell_{\min}(F) = T_{\min}(\lambda_F).
\]
The 3-manifold $S^*(S^2,F)$ is diffeomorphic to $SO(3)$, and we have a non-trivial double cover
\[
p_F : S^3 \rightarrow S^*(S^2,F).
\]
This double cover can be obtained by composing the double cover
\[
S^3 \rightarrow SO(3) \cong S^*(S^2,F_{\mathrm{round}})
\]
given by the round Riemannian metric $F_{\mathrm{round}}$ with the radial diffeomorphism 
\[
S^*(S^2,F_{\mathrm{round}}) \rightarrow S^*(S^2,F).
\] 
With this construction, the map $F\mapsto p_F$ is continuous, when the space of Finsler metrics and the space $C^{\infty}(S^3,T^*S^2)$ are equipped with their $C^k$ topologies.

For any Finsler metric $F$, we set $\alpha_F := p_F^* \lambda_F$. Then we have
\begin{equation}
\label{volcov}
\mathrm{vol}(S^3,\alpha_F \wedge d\alpha_F) = 2 \, \mathrm{vol}(S^*(S^2,F),\lambda_F \wedge d\lambda_F) = 4\pi \, \mathrm{area}(S^2,F).
\end{equation}
The proof of Corollary \ref{cor4} makes use of the following result of  Bangert.

\begin{prop}[{\cite[Corollary 1]{ban86}}] 
\label{bangert}
Let $\phi: \R \times M \rightarrow M$ be a $C^1$ flow  on a compact manifold $M$, such that all the orbits of $\phi$ are $T$-periodic. For every $\epsilon>0$ there exists a neighbourhood $\mathscr{V}$ of $\phi$ in the $C^1_{\mathrm{loc}}$-topology of $C^1(\R\times M,M)$ such that the period of every closed orbit of any flow in $\mathscr{V}$ either belongs to the interval $(T-\epsilon,T+\epsilon)$ or is larger than $1/\epsilon$.
\end{prop}

\begin{proof}[Proof of Corollary \ref{cor4}]
Let $F_0$ be a Finsler Zoll metric on $S^2$. We must find a $C^3$-neighborhood $\mathscr{F}_0$ of $F_0$ in the space of Finsler metrics such that
\begin{equation}
\label{ddaddim}
\ell_{\min}(F)^2 \leq \pi  \, \mathrm{area}(S^2,F), \qquad \forall F\in \mathscr{F}_0,
\end{equation}
with equality holding if and only if the metric $F$ is Zoll. Without loss of generality, we may assume that the common length of the closed geodesics of $F_0$ is $2\pi$, and in particular
\[
T_{\min}(\lambda_{F_0}) = \ell_{\min}(F_0) = 2\pi.
\]
By Theorem B.2 in \cite{abhs17}, the Reeb flow of $\lambda_{F_0}$ is smoothly conjugate to the Reeb flow of $\lambda_{F_{\mathrm{round}}}$. Since the prime periodic orbits of the latter flow are not contractible in $S^*(S^2,F_{\mathrm{round}})\cong SO(3)$, also the prime periodic orbits of the Reeb flow of $\lambda_{F_0}$ are not contractible in $S^*(S^2,F_0)$. Therefore, $\alpha_{F_0}$ is a Zoll contact form on $S^3$ all of whose orbits have period $2\, T_{\min}(\lambda_{F_0}) = 4\pi$.

We claim that $F_0$ has  a $C^3$-neighborhood $\mathscr{F}_0$ in the space of Finsler metrics on $S^2$ such that for every $F\in \mathscr{F}_0$ the following facts hold:
\begin{enumerate}[(i)]
\item $\alpha_F$ belongs to the neighbourhood $\mathscr{A}$ which is given by Theorem \ref{main1};
\item $|T_{\min}(\alpha_F) - T_{\min}(\alpha_{F_0})| = |T_{\min}(\alpha_F) - 4\pi| < \pi$;
\item the period of every closed Reeb orbit of $\lambda_F$ is either in the interval $(\pi,3\pi)$ or larger than $5\pi$.
\end{enumerate}
Indeed, the possibility of obtaining (i) follows from the fact that the map $F \mapsto \alpha_F$ is continuous with respect to the $C^3$ topology on both sides. By Proposition \ref{contTmin}, the real function $\alpha\mapsto T_{\min}(\alpha)$ is continuous with respect to the $C^3$ topology on a  $C^3$-neighbourhood of the Zoll contact form $\alpha_{F_0}$. This implies that we can achieve (ii). Finally, the map sending $F$ into the Reeb vector field of $\lambda_F$ is continuous from the $C^2$ topology to the $C^1$ topology, so we obtain (iii) 
from Proposition \ref{bangert}.

Thanks to (i) together with Theorem \ref{main1} and to (\ref{volcov}), every $F\in \mathscr{F}_0$ satisfies
\begin{equation}
\label{dolman}
T_{\min}(\alpha_F)^2 \leq \mathrm{vol}(S^3, \alpha_F \wedge d\alpha_F) = 4\pi \, \mathrm{area}(S^2,F),
\end{equation}
with equality holding if and only if $\alpha_F$ is a Zoll contact form. Let $z$ be a closed Reeb orbit of $\alpha_F$ with period $T_{\min}(\alpha_F)$. The projected curve $p_F\circ z$ is a periodic orbit of the Reeb flow of $\lambda_F$ whose period is either $T_{\min}(\alpha_F)$ or $T_{\min}(\alpha_F)/2$. Since by (ii) we have
\begin{equation}
\label{dovealpha}
3\pi < T_{\min}(\alpha_F) < 5 \pi,
\end{equation}
condition (iii) implies that the period of $p_F\circ z$ cannot be $T_{\min}(\alpha_F)$. So the period of $p_F\circ z$ is $T_{\min}(\alpha_F)/2$ and (\ref{dolman}) implies that
\begin{equation}
\label{ultima}
\ell_{\min}(F)^2 = T_{\min}(\lambda_F)^2 \leq \frac{T_{\min}(\alpha_F)^2}{4} \leq \pi \, \mathrm{area}(S^2,F),
\end{equation}
proving (\ref{ddaddim}). If the metric $F$ is Zoll, then $\lambda_F$ is a Zoll contact form on $S^*(S^2,F)\cong SO(3)$, and the equality in (\ref{ddaddim}) holds by Theorem B.2 in \cite{abhs17} and identity (\ref{volcov}).
Conversely, assume that the equality holds in (\ref{ddaddim}), or equivalently in (\ref{ultima}). The equality holds a fortiori in
(\ref{dolman}), and hence $\alpha_F$ is Zoll. Then every Reeb orbit of $\lambda_F$ is closed and has period either $T_{\min}(\alpha_F)$ or $T_{\min}(\alpha_F)/2$. But by (\ref{dovealpha}) and (iii) no closed Reeb orbit of $\lambda_F$ can have period $T_{\min}(\alpha_F)$. Therefore, all the Reeb orbits of $\lambda_F$ have common period $T_{\min}(\alpha_F)/2$, so $\lambda_F$ is a Zoll contact form.
Therefore, $F$ is a Zoll Finsler metric. This concludes the proof of Corollary \ref{cor4}. 
\end{proof}


\providecommand{\bysame}{\leavevmode\hbox to3em{\hrulefill}\thinspace}
\providecommand{\MR}{\relax\ifhmode\unskip\space\fi MR }
\providecommand{\MRhref}[2]{%
  \href{http://www.ams.org/mathscinet-getitem?mr=#1}{#2}
}
\providecommand{\href}[2]{#2}

\end{document}